\documentclass[10pt]{siamltex}
\usepackage{amsmath,amsfonts,amscd,amssymb,bm,cite}
\usepackage{mathrsfs,listings,url}
\usepackage{graphics,color,ulem}
\usepackage{float}
\usepackage[titletoc,page]{appendix}
\usepackage{subfigure}
\usepackage[colorlinks=true, bookmarksopen,
pdfauthor={CST},
pdfcreator={pdftex},
pdfsubject={algorithms},
linkcolor={blue},
anchorcolor={black},
citecolor={firebrick},
filecolor={magenta},
menucolor={black},
pagecolor={red},
plainpages=false,pdfpagelabels,
urlcolor={db} ]{hyperref}
\usepackage[capitalise]{cleveref}
\usepackage{comment}
\usepackage{verbatim}
\usepackage{marginnote}
\usepackage{float}
\usepackage[makeroom]{cancel}
\usepackage[pdftex]{graphicx}
\usepackage[section]{placeins}
\usepackage{mathptmx}
\usepackage{cases}
\usepackage{subeqnarray}
\usepackage{hhline}
\usepackage{xcolor}
\usepackage{tikz}
\usepackage{caption} 
\usepackage{lmodern}
\usepackage{algorithm}
\usepackage{algpseudocode}

\usetikzlibrary{arrows}
\usetikzlibrary{positioning,shapes,snakes,calc,decorations,decorations.markings}

\newtheorem{remark}{Remark}

\newtheorem{defn}[theorem]{Definition}
\newtheorem{lem}[theorem]{Lemma}
\newtheorem{thm}[theorem]{Theorem}
\newtheorem{sche}{Scheme}

\restylefloat{table}
\allowdisplaybreaks
\definecolor{db}{rgb}{0.0470,0,0.5294}
\definecolor{dg}{rgb}{0.0,0.392,0.0}
\definecolor{firebrick}{rgb}{0.698,0.133,0.133}
\definecolor{bl}{rgb}{0.0,0.0,0.0}
\definecolor{linen}{rgb}{0.980,0.941,0.902}
\definecolor{ivory}{rgb}{1.0,1.0,0.941}
\definecolor{aliceblue}{rgb}{0.941,0.973,1.0}
\definecolor{beige}{rgb}{0.961,0.961,0.863}
\definecolor{tan}{rgb}{0.824,0.706,0.549}
\definecolor{lightsteelblue}{rgb}{0.690,0.769,0.871}
\definecolor{paleturquoise}{rgb}{0.686,0.933,0.933}
\definecolor{lightblue}{rgb}{0.678,0.847,0.902}
\definecolor{skyblue}{rgb}{0.529,0.808,0.922}
\definecolor{palegoldenrod}{rgb}{0.933,0.910,0.667}
\definecolor{lightgoldenrod}{rgb}{0.933,0.867,0.510}
\definecolor{lightyellow}{rgb}{1.0,1.0,0.878}
\definecolor{yellow}{rgb}{1.0,1.0,0.0}
\definecolor{lightyellow1}{rgb}{1.0,1.0,0.878}
\definecolor{lemonchiffon}{rgb}{1.0,0.980,0.804}
\definecolor{myyellow}{rgb}{1,1,.9}
\definecolor{darkgreen}{rgb}{0.0,0.392,0.0}
\definecolor{darkviolet}{rgb}{0.580,0.0,0.827}
\definecolor{lightsalmon}{rgb}{1.0,0.627,0.478}
\definecolor{orange}{rgb}{1.0,0.647,0.0}
\definecolor{darkblue}{rgb}{0.00,0.00,0.55}
\numberwithin{equation}{section}
\Crefname{table}{Table}{Tables}
\Crefname{figure}{Figure}{Figures}

\newcommand\titlelowercase[1]{\texorpdfstring{\lowercase{#1}}{#1}}

\begin{document}
	
	\title{\large{U\titlelowercase{nconditionally} S\titlelowercase{table}, V\titlelowercase{ariable} S\titlelowercase{tep} DLN M\titlelowercase{ethods for the} A\titlelowercase{llen}-C\titlelowercase{ahn} A\titlelowercase{ctive} F\titlelowercase{luid} M\titlelowercase{odel}: A D\titlelowercase{ivergence-free} P\titlelowercase{reserving} A\titlelowercase{pproach}}} 
	\author{Nan Zheng
			\thanks{
            School of Mathematics, Shandong University, Jinan, Shandong 250100, People’s Republic of China. Email: \href{mailto:202311835@mail.sdu.edu.cn}{202311835@mail.sdu.edu.cn}. } 
            \and  
            Wenlong Pei
            \thanks{
            Department of Mathematics, The Ohio State University, Columbus, OH 43210,USA. Email: \href{mailto:pei.176@osu.edu}{pei.176@osu.edu}.} 
            \and  
            Qingguang Guan
            \thanks{
            School of Mathematics and Natural Sciences, University of Southern Mississippi, 118 College Drive, Hattiesburg, MS, 39406, USA   
            Email: \href{mailto:qingguang.guan@usm.edu}{qingguang.guan@usm.edu}.
            }
            \and  
            Wenju Zhao
            \thanks{
            School of Mathematics, Shandong University, Jinan, Shandong 250100, People’s Republic of China.
            Email: \href{mailto:zhaowj@sdu.edu.cn}{zhaowj@sdu.edu.cn}.
            }
		    }
	\date{\emty}
	\maketitle

    \begin{abstract}
		This paper addresses the divergence-free mixed finite element method (FEM) for nonlinear fourth-order Allen-Cahn phase-field coupled active fluid equations. 
        By introducing an auxiliary variable \( w = \Delta u \), the original fourth-order problem is converted into a system of second-order equations, thereby easing the regularity constraints imposed on standard \( H^2 \)-conforming finite element spaces. 
        To further refine the formulation, an additional auxiliary variable \( \xi \), analogous to the pressure, is introduced, resulting in a mixed finite element scheme that preserves the divergence-free condition in $w = \Delta u$ inherited from the model. 
        A fully discrete scheme is then established by combining the spatial approximation by the divergence-free mixed finite element method with the variable-step Dahlquist–Liniger–Nevanlinna (DLN) time integrator. 
        The boundedness of the scheme is rigorously derived under suitable regularity assumptions. 
        Additionally, an adaptive time-stepping strategy based on the minimum dissipation criterion is carried out to enhance computational efficiency. 
        Several numerical experiments validate the theoretical findings and demonstrate the method’s effectiveness and accuracy in simulating complex active fluid dynamics.

	\end{abstract}
	
	

    \section{Introduction}
    The Allen-Cahn active fluid system, comprising self-propelled particles, exhibits unique dynamics resulting from the interaction between active fluid motion and interfacial phase-field evolution. 
    The system extends classical models of incompressible active fluids by incorporating the phase-field dynamics,
    and thus becomes essential for describing a wide range of physical phenomena, such as phase-field governed active suspensions, bacterial colonies, and other biophysical systems, where interfacial effects play a crucial role \cite{ramaswamy2019active, annurev-fluid-010816-060049, pnas.1722505115, qi2022emergence,MR4927938}.
    In the report, we consider the following system of Allen-Cahn active fluid equations on the domain $D$ over time 
    interval $[0,T]$:
	\begin{equation}\label{ACAFs}
		\begin{cases}
			u_{t} - \mu \Delta u + \gamma \Delta^2 u+\nu (u \cdot \nabla)u  + \rho u +  \lambda |u|^2u+ \nabla p =m \nabla \phi,    
			&\text{in}  ~D \times (0,T], \\
			\nabla \cdot u = 0,   
			&\text{in} ~ D \times (0,T],\\
			\phi_t + u \cdot \nabla \phi =- \sigma m ,   
			&\text{in}  ~D \times (0,T],\\
			m= \kappa( - \Delta \phi  + f(\phi))
			&\text{in}  ~D \times (0,T],
		\end{cases}
	\end{equation}
	subject to the following initial and boundary conditions: 
	\begin{align*}
        \begin{cases}
            u(x,0) = u_{0} \\
            \phi(x,0) = \phi_{0}
        \end{cases} \text{in} ~D, \qquad
        \begin{cases}
            u =\Delta u =  0  \\
            \partial_{n} \phi = \partial_{n} m= 0
        \end{cases} ~\text{on}   ~\partial D.
	\end{align*}
    Here $u$ is the velocity field, $p$ the pressure, $\phi$ the phase function taking values in $[-1,1]$ and distinguishing between the two fluid components, and $m$ is the chemical potential.
    The domain $D$ is bounded with Lipschitz boundary $\partial D$. $\partial_{n}$ denotes the partial derivative in the direction of the outward normal vector to $\partial D$.
	Non-negative parameters $\mu$, $\gamma$, and $\nu$ represent the viscosity coefficient, the generic stability coefficient, and the density coefficient, respectively.
	Terms $\rho u$ ($\rho  \in \mathbb{R}$) and $\lambda |u|^2 u$ ($\lambda \geq 0$) correspond to a quartic Landau velocity potential 
	\cite{ramaswamy2019active,toner1998flocks}. 
    The parameter \( \sigma \) is related to the surface tension parameter, and \( \kappa \) the mixing energy density. 
    The Helmholtz free-energy density $F(\phi)$ and its derivative $f(\phi)$ take the form
	\begin{equation}\label{F phi}
		F(\phi) = \frac{1}{4 \eta^2}(\phi^{2} - 1)^{2}, 
		\quad 
		f(\phi) = \frac{d}{d \phi}F(\phi) = \frac{\phi^{3} - \phi}{\eta^2}.
	\end{equation}
    with the parameter $\eta$. 
    The total energy $W$ of the system is the sum of the kinetic energy and the Ginzburg-Landau type of Helmholtz free energy:
	\[
	\mathcal{E} = \int_{D} \frac{1}{2}|u|^2 + \kappa \Big( \frac{|\nabla \phi|^2}{2} + F(\phi) \Big) dx.
	\]
	It satisfies the energy dissipation law, i.e., $W$ is a non-increasing function with respect to time. 

    In the system, the phase-field function \( \phi \) distinguishes between different fluid components, with the sharp interface between them implicitly tracked by a thin, smooth transition layer, known as the diffuse interface.
    This approach avoids the complexity of explicit interface tracking and performs simulations on a fixed mesh grid, providing a convenient numerical method for simulating various interfacial problems.
    However, the coupling of the Allen-Cahn phase field variable with the velocity field introduces additional complexities, particularly in capturing the dynamics of the interface, phase separation, and pattern formation in active matter.
    Meanwhile, the mathematical characterization of active fluid dynamics often involves generalized forms of the Navier-Stokes equations augmented with higher-order dissipative terms and nonlinear active forcing. 
    Hence, the non-linear system of fourth-order equations presents substantial analytical and numerical challenges, 
    despite the effectiveness in capturing rich dynamical behaviors. 
    To address these issues and reduce the regularity requirements, we 
    introduce an auxiliary variable  \( w = -\Delta u \), and 
    reformulate the original fourth-order active fluid 
    equations \eqref{ACAFs} into the following equivalent system of second-order equations 
    \begin{equation}\label{ACAFs 2nd}
		\begin{cases}
			u_{t} - \mu \Delta u - \gamma \Delta w+\nu (u \cdot \nabla)u  + \rho u +  \lambda |u|^2u+ \nabla p =m \nabla \phi,    
			&\text{in}  ~D \times (0,T], \\
			w = -\Delta u, &\text{in} ~ D \times (0, T], \\
			\nabla \cdot u = 0,   
			&\text{in} ~ D \times (0,T],\\
			\phi_t + u \cdot \nabla \phi =- \sigma m ,   
			&\text{in}  ~D \times (0,T],\\
			m=\kappa  (- \Delta \phi  + f(\phi))
			&\text{in}  ~D \times (0,T].
		\end{cases}
	\end{equation}
	subject to the following modified initial-boundary conditions:
    \begin{align*} 
        \begin{cases}
            u(x,0) = u_{0} \\
            \phi(x,0) = \phi_{0}
        \end{cases} \text{in} ~D, \qquad
        \begin{cases}
            u = w =  0  \\
            \partial_{n} \phi = \partial_{n} m= 0
        \end{cases} ~\text{on}   ~\partial D.
	\end{align*}
    The resulting system \eqref{ACAFs 2nd} is eligible for the use of finite element methods based on $H_0^1$-conforming basis functions without 
	the restrictive \( H^2 \)-regularity requirements.
    More importantly, the auxiliary variable inherently satisfying the divergence-free constraint $(\nabla \cdot w = 0)$ preserves physical fidelity and the incompressibility condition, which are crucial for realistic simulations \cite{da2025error,da2017divergence}.

    Finite element methods (FEM), widely recognized for their robustness, adaptability, and efficiency, have become essential tools for spatial discretization in Navier-Stokes equations-related models \cite{MR4293957, he2007stability, MR4835947, abgrall2023hybrid, abgrall2020analysis, ayuso2005postprocessed}.
    For temporal discretization, various approaches have been rigorously studied and applied, including the Euler method, Crank-Nicolson scheme, Runge-Kutta methods, and backward differentiation formulas \cite{ern2022invariant, MR4835947, AAMM-16-5, hou2025unconditionally, baker2024numerical, banjai2012runge, BAI20123265, decaria2022general}.
    Recently, significant efforts have been devoted to the numerical analysis of variable time-stepping schemes, especially regarding their potential for time adaptivity \cite{ait2023time, MR4471049}.

    The family of DLN methods, with one parameter \(\theta \in [0,1]\), ensures unconditional stability and second-order accuracy for dissipative non-linear systems with arbitrary time grids \cite{dahlquist1983stability, LPT21_AML, LPT23_ACSE}.
    This family of schemes has been successfully applied to highly stiff differential equations and various complex fluid models \cite{QHPL21_JCAM, layton2022analysis, MR4866010, CLPX2025_JSC, pei2024semi, SP24_IJNAM,pei2025ensemble}.
    Given the time interval $[0,T]$, $\{ t_{n} \}_{n=0}^{M}$ denotes its partition and  $k_n = t_{n+1} - t_n$ the local time step size. 
	For the initial value problem, $y'(t) = g(t,y(t))$ with $t \in [0,T], \ y(0) = y_{0} \in \mathbb{R}^{d}$, 
	the family of variable time-stepping DLN methods for the problem reads 
    \begin{gather}
		\sum_{\ell =0}^{2}{\alpha _{\ell }}y_{n-1+\ell }
		= \widehat{k}_{n} g \Big( \sum_{\ell =0}^{2}{\beta _{\ell }^{(n)}}t_{n-1+\ell } ,
		\sum_{\ell =0}^{2}{\beta _{\ell }^{(n)}}y_{n-1+\ell} \Big), \qquad n=1,\ldots,M-1.
		\label{eq:1leg-DLN}
	\end{gather}
	$y_{n}$ represents the DLN solution of $y(t)$ at time $t_{n}$,
    and the coefficients of the family of schemes in \eqref{eq:1leg-DLN} are 
	\begin{gather}
		\label{DLNcoeff}
		\begin{pmatrix}
			\alpha _{2} \vspace{0.2cm} \\
			\alpha _{1} \vspace{0.2cm} \\
			\alpha _{0} 
		\end{pmatrix}
		= 
		\begin{pmatrix}
			\frac{1}{2}(\theta +1) \vspace{0.2cm} \\
			-\theta \vspace{0.2cm} \\
			\frac{1}{2}(\theta -1)
		\end{pmatrix}, \ \ \ 
		\begin{pmatrix}
			\beta _{2}^{(n)}  \vspace{0.2cm} \\
			\beta _{1}^{(n)}  \vspace{0.2cm} \\
			\beta _{0}^{(n)}
		\end{pmatrix}
		= 
		\begin{pmatrix}
			\frac{1}{4}\Big(1+\frac{1-{\theta }^{2}}{(1+{%
					\varepsilon _{n}}{\theta })^{2}}+{\varepsilon _{n}}^{2}\frac{\theta (1-{%
					\theta }^{2})}{(1+{\varepsilon _{n}}{\theta })^{2}}+\theta \Big)\vspace{0.2cm%
			} \\
			\frac{1}{2}\Big(1-\frac{1-{\theta }^{2}}{(1+{\varepsilon _{n}}{%
					\theta })^{2}}\Big)\vspace{0.2cm} \\
			\frac{1}{4}\Big(1+\frac{1-{\theta }^{2}}{(1+{%
					\varepsilon _{n}}{\theta })^{2}}-{\varepsilon _{n}}^{2}\frac{\theta (1-{%
					\theta }^{2})}{(1+{\varepsilon _{n}}{\theta })^{2}}-\theta \Big)%
		\end{pmatrix},
	\end{gather}
    where $\varepsilon _{n} = (k_n - k_{n-1})/(k_n + k_{n-1}) \in (-1,1)$ is the step variability.
    The weighted average of time step  $\widehat{k}_n = \alpha_2 k_n - \alpha_0 k_{n-1}$ is to ensure second-order accurate. 
    This family is reduced to the midpoint rule on $[t_{n}, t_{n+1}]$ if $\theta = 1$ and midpoint rule on $[t_{n-1}, t_{n+1}]$ if $\theta = 0$.

	Building on these advancements, we integrate the divergence-free mixed finite element spatial discretization with the DLN temporal discretization, proposing a family of computationally efficient, unconditionally stable, and accurate schemes for the complex dynamics of the Allen-Cahn active fluid equations.
	Additionally, we develop a time-adaptive strategy based on the minimal dissipation criterion \cite{capuano2017minimum}, to strike a balance between computational efficiency and accuracy.

	The key contributions of this work are 
	\begin{itemize}
		\item developing the divergence-free mixed finite element spatial discretization for the fourth-order system Allen-Cahn active fluid equations \eqref{ACAFs} that simplifies the problem and relaxes regularity conditions,
		\item employing the family of variable time-stepping DLN temporal integrators for full discretization, and providing a rigorous proof that the fully discrete schemes adhere to the energy dissipation law unconditionally, regardless of the time grid,
		\item designing a time-adaptive strategy based on the minimum dissipation criterion, significantly improving computational efficiency.
	\end{itemize}

	The remainder of the paper is organized as follows. 
	We introduce necessary preliminaries, notation, and weak formulations of \eqref{ACAFs 2nd} in Section~\ref{sec:sec2}. 
	In Section~\ref{sec:sec3}, we present a family of computationally efficient algorithms via the combination of the divergence-free mixed finite element method and the family of variable time-stepping DLN integrators. 
	We establish the discrete energy-dissipation law for the family of schemes with strict proof in Section~\ref{sec:sec4}.
	In Section~\ref{sec:sec5}, we carry out several numerical tests, including the space–time convergence test with known exact solutions, tests of spinodal decomposition with binary fluids, bubble merging, and 
	circular-bubble shrinking, and the time-adaptive algorithm driven by a minimum-dissipation criterion for the three-dimensional motion by mean curvature.
	Section~\ref{sec:sec6} summarizes the main findings and outlines directions for future research.

	\section{Notation and preliminaries} \label{sec:sec2}
	Given the domain $D \subset \mathbb{R}^{d}$ ($d = 2,3$), 
	$W^{r,p}(D)$ ($1 \leq p \leq \infty$, $r \in \mathbb{N}$) 
	is the Sobolev space equipped with the norm \( \|\cdot\|_{W^{r,p}} \).
	For $p = 2$, it reduces to the Hilbert space \( H^r(D) \) with the norm \( \|\cdot\|_r \). 
	Specifically, the Lebesgue space \( L^2(D) := H^{0}(D) \) is endowed with the standard inner product \((\cdot, \cdot)\) and the corresponding norm \( \|\cdot\| \). 
	For time-dependent functions, we introduce the Bochner space \( L^p([0,T]; H^r(D)) \) equipped with the following norms
	\[
	\| v \|_{\infty,r} = \max_{0 \leq t \leq T} \|v(t)\|_r, 
	\quad
	\| v \|_{p,r} = \Big( \int_{0}^{T}\|v(t)\|_r^p \, \mathrm{d}t \Big)^{\frac{1}{p}}, \text{ for } 1 \leq p < \infty. 
	\]
	The solution space for the velocity \( u \), pressure \( p \), and phase field \( \phi \) in the system \eqref{ACAFs 2nd} are 
	\begin{align*}
		\text{velocity space: } \qquad 
		X &= \Big\{ v \in [H^1(D)]^d : v = 0 ~\text{on}~ \partial D \Big\}, \\
		\text{pressure space: } \qquad 
		Q &= \Big\{ q \in L^2(D) : \int_{D} q \,\mathrm{d}x = 0 \Big\}, \\
		\text{phase field space: } \qquad 
		\Phi  &= H^1(D). 
	\end{align*}
	The divergence-free subspace for the velocity is 
	\begin{equation*}\label{V_space}
		V = \Big\{ v \in X : \nabla \cdot v = 0 \text{ in } D \Big\}.
	\end{equation*}
	We define the trilinear forms $b(\cdot, \cdot, \cdot)$ on $X \times X \times X$, by
	\begin{align}\label{b form}
		b(u,v,w) = \frac{1}{2}(u \cdot \nabla v, w) - \frac{1}{2}(u \cdot \nabla w, v), \quad 
		&u,v,w \in X.
	\end{align}
	The weak form of the system \eqref{ACAFs 2nd} is:  
	finding $(u,w,p,\phi,m) \in (X,V,Q,\Phi,\Phi)$ such that for all $(v,\varphi,q,\vartheta,\varsigma) \in (X,V,Q,\Phi,\Phi)$, and all $t \in (0,T]$ 
	\begin{align} 
		\begin{split} \label{variational-formula-1} 
			&(u_{t}, v)
			+ \mu(\nabla u, \nabla v) 
			+\gamma (\nabla w,\nabla v)
			+ \nu b( u,  u, v)  
			+ \rho (u,v) 
			+ \lambda (|u|^2 u ,v)
			\\
			&\quad 
			- (p, \nabla \cdot v)
			= (m \nabla \phi, v), 
		\end{split}
		\\
		&(w,\varphi) = (\nabla u,\nabla \varphi), \label{variational-formula-2}
		\\
		&(\nabla \cdot u, q) = 0,  \label{variational-formula-3}
		\\
		&	(\phi_{t}, \vartheta)
		+ (u \cdot \nabla \phi,  \vartheta)
		=-\sigma (m,\vartheta), \label{variational-formula-4}
		\\
		&	(m, \varsigma)
		=\kappa (\nabla \phi,\nabla  \varsigma)
		+\kappa(f(\phi), \varsigma), \label{variational-formula-5}
	\end{align} 
	The following equivalent variational formulation can also be derived
	by enforcing the incompressibility constraints in a weak sense:
	finding $(u,w,\xi,p,\phi,m) \in (X,X,Q,Q,\Phi,\Phi)$, such that for all $(v,\varphi,\zeta,q,\vartheta,\varsigma) \in (X,X,Q,Q,\Phi,\Phi)$, and all $t \in (0,T]$
	\begin{align} 
		\begin{split} \label{variational_formula_second_1} 
			&(u_{t}, v)
			+ \mu(\nabla u, \nabla v) 
			+\gamma (\nabla w,\nabla v)
			+ \nu b( u,  u, v)  
			+ \rho (u,v) 
			+ \lambda (|u|^2 u ,v)
			\\
			&\quad 
			- (p, \nabla \cdot v)
			= (m \nabla \phi, v), 
		\end{split}
		\\
		&(w,\varphi)-(\xi,\nabla \cdot \varphi) = (\nabla u,\nabla \varphi), \label{variational_formula_second_2}
		\\
		&(\nabla \cdot u, q) = 0.  \label{variational_formula_second_3}
		\\
		&(\nabla \cdot w, \zeta) = 0,  \label{variational_formula_second_4}
		\\
		&(\phi_{t}, \vartheta)
		+ (u \cdot \nabla \phi,  \vartheta)
		= -\sigma (m,\vartheta),\label{variational_formula_second_5}
		\\
		&(m, \varsigma)
		= \kappa(\nabla \phi,\nabla  \varsigma)
		+\kappa(f(\phi), \varsigma). \label{variational_formula_second_6}
	\end{align}
	Given arbitrary sequence $\{ z_{n} \}_{n=0}^{\infty}$,  we adopt the following notations for convenience
	\begin{align*}
		z_{n,\theta} &= \frac{1+\theta}{2} z_n + \frac{1-\theta}{2} z_{n-1}, 
		\quad 
		z_{n,\alpha} = \sum_{\ell=0}^{2} \alpha_{\ell} z_{n-1+\ell}, 
		\quad 
		z_{n,\beta} = \sum_{\ell=0}^{2} \beta_{\ell}^{(n)} z_{n-1+\ell}
	\end{align*}
	where $\{ \alpha_{\ell} \}_{\ell=0}^{2}$ and $\{ \beta_{\ell}^{(n)} \}_{\ell=0}^{2}$ are coefficents of the DLN method in \eqref{DLNcoeff}.

	\section{Spatial and temporal discretization} \label{sec:sec3}
	\subsection{Spatial discretization}
	Let $\mathscr{T}_h$ ($0<h<1$) be a regular triangulation of $D$ with mesh size $h$
	For the fixed positive integer $r \geq 1$, $P_k(K)$ denotes the space of all polynomials of degree less than or equal to $r$ on the element $K \in \mathscr{T}_h$. 
	The finite element spaces for spatial discretization are:
	\begin{align}
		&X_{h} = \Big\{ v^{h} \in [C^{0}(\bar{D})]^2 \cap X: v^{h} \mid_{K} \in P_{r+1}(K), \ \forall K \in \mathscr{T}_{h} \Big\},   
		\label{zeq:FEM-space-Xh-1} \\
		&Q_{h} = \Big\{ q^{h} \in C^{0}(\bar{D}) \cap L^2(D): q^{h} \mid_{K} \in P_{r}(K), \ \forall K \in \mathscr{T}_{h} \Big\},   
		\label{zeq:FEM-space-Qh-1} \\
		&\Phi_{h} = \Big\{ \vartheta^{h} \in C^{0}(\bar{D}) \cap \Phi: \vartheta^{h} \mid_{K} \in P_{r+1}(K), \ \forall K \in \mathscr{T}_{h} \Big\},   \label{zeq:FEM-space-Jh-1} \\
		&V_{h} = \Big\{ v^{h} \in X_{h}: \big( \nabla \cdot v^{h}, q^{h} \big) = 0, \ \forall q^{h} \in Q_{h} \Big\}.  \label{zeq:FEM-space-Vh-1}
	\end{align}
	The spaces $X_h \times Q_h$ in \eqref{zeq:FEM-space-Xh-1}-\eqref{zeq:FEM-space-Qh-1} should satisfy the discrete Ladyzhenskaya-Babuska-Brezzi condition ($LBB^h$ condition):
	for each $ q^h \in Q_h$, there exists a positive constant $C$ such that
	\begin{equation}\label{LBB}
		\|q^h\| \leq C \sup_{v^h \in X_h \setminus \{0\}} \frac{(\nabla \cdot v^h, q^h)}{\|\nabla v^h\|}.
	\end{equation}
	Taylor-Hood element space, Mini element space, and Scott-Vogelius element space~\cite{MR813691} are all typical finite element spaces having $LBB^h$ condition in \eqref{LBB}. 
	We choose Taylor-Hood element space (with $r=1$) for spatial discretization throughout numerical tests.

	We introduce Ritz projection \cite{he2005stabilized}
	$\mathcal{R}_h: \Phi \rightarrow \Phi_h$ 
	\begin{equation}\label{Ritz-projection-defination}
		(\nabla(\phi - \mathcal{R}_h \phi), \nabla \phi^h)=0, \quad
		( \phi - \mathcal{R}_h \phi, 1 ) = 0, \quad 
		\forall \varphi \in \Phi_h
	\end{equation}
	We also define a Stokes-type projection \( \mathcal{S}_h : (X, X, Q, Q) \rightarrow (X_h, X_h, Q_h, Q_h) \) through the following system: given  \( (u, w, p, \phi) \in (X, X, Q, Q) \), 
	$$ \mathcal{S}_h (u, w, p, \phi) = (\mathcal{S}_h u, \mathcal{S}_h w, \mathcal{S}_h p, \mathcal{S}_h \phi) \in (X_h, X_h, Q_h, Q_h)$$
	uniquely solves 
	\begin{equation}\label{Stokes-type-projection-defination}
		\begin{cases}
			\mu(\nabla(u - \mathcal{S}_{h} u), \nabla v^h) 
			+ \gamma (\nabla(w - \mathcal{S}_{h} w), \nabla v^h) 
			- (p - \mathcal{S}_{h} p, \nabla \cdot v^h) = 0, 
			\\
			(w - \mathcal{S}_{h} w, \varphi^h) - (\phi - \mathcal{S}_{h} \phi, \nabla \cdot \varphi^h)
			= (\nabla(u - \mathcal{S}_{h} u), \nabla \varphi^h),
			\\
			(\nabla \cdot \mathcal{S}_h u, q^h) = 0,
			\\
			(\nabla \cdot \mathcal{S}_h w, \zeta^h) = 0.
		\end{cases}
	\end{equation}
	for all \( (v^h, \varphi^h, q^h, \zeta^h) \in (X_h, X_h, Q_h, Q_h) \).

	\subsection{Fully discrete DLN scheme} 
	We denote $u_n^h$, $w_n^h$, $\xi_n^h$, $p_{n}^h$, $\phi_n^h$, $m_{n}^h$ to be the numerical approximations of $u$, $w$, $\xi$, $p$, $\phi$, $m$ at time $t_n$.  
	The fully discrete DLN scheme of \eqref{variational-formula-1}-\eqref{variational-formula-5} is:
	\begin{sche}[Fully discrete DLN scheme] \label{fully discrete formulations scheme}
		Given $u_n^h, u_{n-1}^h \in X_h,\phi_n^h$, $\phi_{n-1}^h \in \Phi_h$,
		$p_n^h, p_{n-1}^h \in Q_h$ for $n= 1, 2, \dots, M-1$, 
		find $(u_{n+1}^h, w_{n+1}^h,\xi_{n+1}^h, p_{n+1}^h,\phi_{n+1}^h,m_{n+1}^h) \in ( X_h, X_h, Q_h,Q_h,\Phi_h,\Phi_h)$ such that 
		\begin{align} \label{fully discrete formulations 1}
			\begin{split}
				&\frac{1}{\widehat{k}_n}(u_{n,\alpha}^h, v^h)
				+ \mu(\nabla u_{n,\beta}^h, \nabla v^h) 
				+\gamma (\nabla w_{n,\beta}^h,\nabla v^h)
				+ \nu b( u_{n,\beta}^h,  u_{n,\beta}^h, v^h) 
				\\
				&\quad  
				+ \rho (u_{n,\beta}^h,v^h) 
				+ \lambda (|u^h_{n,\beta}|^2 u^h_{n,\beta} ,v^h)
				- (p_{n,\beta}^h, \nabla \cdot v^h)
				= ( m_{n,\beta}^h \nabla \phi_{n,\beta}^h, v^h), 
			\end{split}
			\\
			\begin{split}
				&(w_{n+1}^h,\varphi^h) -(\xi_{n+1}^h,\nabla \cdot \varphi^h)= (\nabla u_{n+1}^h,\nabla \varphi^h),\label{fully discrete formulations 2}
			\end{split}
			\\
			\begin{split}
				&(\nabla \cdot u_{n+1}^h, q^h) = 0, \label{fully discrete formulations 3}
			\end{split}
			\\
			\begin{split}
				&(\nabla \cdot w_{n+1}^h, \zeta^h) = 0, \label{fully discrete formulations 4}
			\end{split}
			\\
			&\frac{1}{\widehat{k}_n}(\phi_{n,\alpha}^h, \vartheta^h)
			+ (u_{n,\beta}^h \cdot \nabla \phi_{n,\beta}^h,  \vartheta^h)
			=- \sigma (m_{n,\beta}^h,\vartheta^h),\label{fully discrete formulations 5}
			\\
			&	(m_{n,\beta}^h, \varsigma^h)
			=\kappa(\nabla \phi_{n,\beta}^h,\nabla  \varsigma^h)
			+\kappa(\widetilde{f}(\phi_{n+1,\theta}^h, \phi_{n,\theta}^h), \varsigma^h), \label{fully discrete formulations 6}
		\end{align} 
		true for all $(v^h, \varphi^h, q^h,\zeta^h,\vartheta^h,\varsigma^h) \in (X_h, X_h,Q_h,Q_h,\Phi_h,\Phi_h)$. Here 
		\begin{align*}
			\widetilde{f}(\phi_{n+1,\theta}^h, \phi_{n,\theta}^h)
			= 
			\begin{cases}
				\dfrac{F(\phi_{n+1,\theta}^h) - F(\phi_{n,\theta}^h)}{\phi_{n+1,\theta}^h - \phi_{n,\theta}^h}, & \text{if } \phi_{n+1,\theta}^h \neq \phi_{n,\theta}^h, \\[2ex]
				f \Big( \dfrac{\phi_{n+1,\theta}^h + \phi_{n,\theta}^h}{2} \Big), & \text{if } \phi_{n+1,\theta}^h = \phi_{n,\theta}^h. 
			\end{cases}
		\end{align*}
	\end{sche}
	\begin{remark}
		In practice, we set $u_0^h = \mathcal{S}_h u(t_0)$, 
		$w_0^h = \mathcal{S}_h w(t_0)$, $p_0^h = \mathcal{S}_h p(t_0)$, 
		$\phi^h_0 = \mathcal{R}_h \phi(t_0)$, 
		$m^h_0 = \mathcal{R}_h m(t_0)$, 
		and employ the fully-implicit Crank-Nicolson scheme to solve for $u_1^h$, $w_1^h$, $\xi_1^h$, $p_1^h$, $\phi_1^h$, $m_1^h$.
		In return, the solution at the first time step is second-order accurate in time and numerically stable due to the stability and consistency of Crank-Nicolson scheme.
	\end{remark}
	\begin{remark}
		$\widetilde{f}$ comes from the variable time-stepping modified DLN schemes for the Allen-Cahn model. 
		For $\theta \in [0,1)$, $\widetilde{f}(\phi_{n+1,\theta}, \phi_{n,\theta})$ is a first-order approximation to $f(\phi(t_{n,\beta}))$ in time under variable time steps and a second-order approximation in time under constant time steps \cite{MR4927938}. 
		The case $\theta = 1$ is reduced to the modified Crank-Nicolson scheme and becomes a second-order approximation in time under arbitrary time grids \cite{XLWB19_CMAME}. 
	\end{remark}

	\section{Discrete energy dissipation law} 
	\label{sec:sec4} \ 
	We utilize the $G$-stability property of the variable time-stepping DLN method \eqref{eq:1leg-DLN} to prove that \Cref{fully discrete formulations scheme} satisfies the discrete energy dissipation law unconditionally under arbitrary time grids.
	\begin{defn}
		For $\theta \in [0 , 1]$ and  $u, v \in [L^2(D)]^2$, the $G$-norm $\| \cdot \|_{G(\theta)}$ is 
		\begin{equation} \label{G-norm}
			\begin{Vmatrix}
				u \\
				v
			\end{Vmatrix}_{G(\theta )}
			=
			\int_{D}
			\begin{bmatrix}u^\top & v^\top 
			\end{bmatrix} G(\theta) 
			\begin{bmatrix}
				u \\
				v
			\end{bmatrix}
			\mathrm{d}D
			=\frac{1}{4}(1+\theta)\|u\|^2 +\frac{1}{4}(1-\theta)\|v\|^2,
		\end{equation}
		where the notation $\top$ represents the transpose of a vector, 
		the symmetric semi-positive definite matrix $G(\theta)$ is a symmetric semi-positive definite matrix with $\mathbb{I}_{2\times 2}$ identity matrix defined as
		\[
		G(\theta) = \begin{bmatrix}
			\frac{1}{4}(1+\theta)\mathbb{I}_{d} & 0 \\
			0 & \frac{1}{4}(1-\theta)\mathbb{I}_{d}
		\end{bmatrix},
		\]
		and $\mathbb{I}_{d} \in \mathbb{R}^{d \times d}$ is the identity matrix.
	\end{defn}
	\begin{lem} \label{G-stable-lemma}
		For any sequence $\{ v_{n} \}_{n=0}^{M} \subset [L^2(D)]^2$, the following $G$-stability identity holds
		\begin{equation}\label{G-stable}
			\big( v_{n,\alpha}, v_{n,\beta} \big)  
			=
			\begin{Vmatrix}
				v_{n+1} \\
				v_n
			\end{Vmatrix}_{G(\theta )}
			-
			\begin{Vmatrix}
				v_{n} \\
				v_{n-1}
			\end{Vmatrix}_{G(\theta )}
			+
			\Big\| \sum_{\ell=0}^{2} a_{\ell}^{(n)} v_{n-1+{\ell}} \Big\|^2,
		\end{equation}
		holds for all $n = 1,2, \cdots, M-1$ and any fixed $\theta \in [0,1]$. Here $\{a_{\ell}^{(n)}\}_{\ell=0}^2$ 
		in \eqref{G-stable} are 
		\begin{equation*} 
			a_1^{(n)}  = -\frac{\sqrt{\theta (1-\theta^2)}}{\sqrt{2}(1+\varepsilon_n \theta)},  \quad
			a_2^{(n)} = -\frac{1-\varepsilon_n}{2} a_1^{(n)} ,   \quad
			a_0^{(n)} = -\frac{1+\varepsilon_n}{2} a_1^{(n)}.   
		\end{equation*}
	\end{lem}
	\begin{proof}
		The proof of $G$-stability identity in \eqref{G-stable} is just a algebraic calculation. 
	\end{proof}

	\begin{defn}\label{Energy numerical}
		The discrete model energy associated with \Cref{fully discrete formulations scheme} at time $t_n$ is 
		\begin{equation}
			\mathcal{E}_n := 
			\begin{Vmatrix}
				u^h_{n+1} \\
				u^h_{n} 
			\end{Vmatrix}_{G(\theta )}
			+\kappa \Big(
			\begin{Vmatrix}
				\nabla \phi^h_{n+1} \\
				\nabla \phi^h_{n} 
			\end{Vmatrix}_{G(\theta )}
			+ \int_{D} F(\phi_{n+1,\theta}^h) dx	\Big), 
		\end{equation}
	\end{defn}

	\begin{thm}[Discrete Energy dissipation Law] \label{Energy dissipation}
		If $\rho \geq 0$, 
		\Cref{fully discrete formulations scheme} satisfies the following
		discrete energy dissipation law unconditionally  
		\begin{align} 
			\mathcal{E}_{n+1} \leq  \mathcal{E}_n , \quad n = 1,2, \cdots, M-1.
			\label{Energy dissipation inequality}
		\end{align}
		Thus \Cref{fully discrete formulations scheme} is unconditional stable in model energy under arbitrary time grids.
	\end{thm}
	\begin{proof}
		By \eqref{fully discrete formulations 2}--\eqref{fully discrete formulations 4}, we have 
		\begin{align*}
			(w_{n,\beta}^h,\phi^h) - (\phi_{n,\beta}^h,\nabla \cdot \phi^h) 
			&= (\nabla u_{n,\beta}^h,\nabla \phi^h), 
			\quad \forall \phi^h \in X_h, \\
			(\nabla \cdot u_{n,\beta}^h,q^h) &=0, \quad \quad \quad \quad \quad \quad \ 
			\forall q^h \in Q_h, \\
			(\nabla \cdot w_{n,\beta}^h, \zeta^h) &= 0, 
			\quad \quad \quad \quad \quad \quad \  \forall \zeta^h \in Q_h.
		\end{align*}
		We choose $\phi^h = w_{n,\beta}^h$, $\zeta^h = \phi_{n,\beta}^h$, 
		$q^h = p^h_{n,\beta}$, 
		$v^h = u^h_{n,\beta}$ in \eqref{fully discrete formulations 1}, add four equations together and use the skew-symmetric property of $b$ to obtain 
		\begin{align} \label{Energy dissipation-eq1}
			\begin{split}
				&\frac{1}{\widehat{k}_n} (u_{n,\alpha}^h, u^h_{n,\beta}) 
				+ \mu \| \nabla u^h_{n,\beta} \|^2 + \gamma \| w_{n,\beta}^h \|^2
				+\rho \| u^h_{n,\beta} \|^2 + \lambda \| u^h_{n,\beta} \|_{L^4}^4 
				\\
				&
				= (m_{n,\beta}^h \nabla \phi_{n,\beta}^h, u^h_{n,\beta}).
			\end{split}
		\end{align}
		By setting $\vartheta^h = m^h_{n,\beta}$ in \eqref{fully discrete formulations 5}, $\varsigma^h =\frac{1}{\widehat{k}_n} \phi_{n,\alpha}^h = \frac{1}{\widehat{k}_n}(\phi^h_{n+1,\theta} - \phi^h_{n,\theta})$ in \cref{fully discrete formulations 6}, we have
		\begin{align}
			&\frac{1}{\widehat{k}_n}(\phi_{n,\alpha}^h, m_{n,\beta}^h)
			+ (u_{n,\beta}^h \cdot \nabla \phi_{n,\beta}^h,  m_{n,\beta}^h)
			= -\sigma (m_{n,\beta}^h,m_{n,\beta}^h),
			\label{Energy dissipation-eq2}
			\\
			&\frac{1}{\widehat{k}_n}	(m_{n,\beta}^h, \phi_{n,\alpha}^h)
			= \frac{1}{\widehat{k}_n}\kappa(\nabla \phi_{n,\beta}^h,\nabla  \phi_{n,\alpha}^h)
			+\frac{1}{\widehat{k}_n}\kappa (\widetilde{f}(\phi_{n+1,\theta}^h, \phi_{n,\theta}^h), \phi^h_{n+1,\theta} - \phi^h_{n,\theta}). 
			\label{Energy dissipation-eq3}
		\end{align}
		We add \eqref{Energy dissipation-eq1} - \eqref{Energy dissipation-eq3} and use the fact 
		\begin{gather*}
			\widetilde{f} \big( \phi_{n+1,\theta}^h,\, \phi_{n,\theta}^h \big) \big( \phi_{n+1,\theta}^h - \phi_{n,\theta}^h \big)    
			= \big( F ( \phi_{n+1,\theta}^h ) - F ( \phi_{n,\theta}^h ), 1 \big)  
		\end{gather*}
		to derive
		\begin{align*}
			\begin{split}
				& (u_{n,\alpha}^h, u^h_{n,\beta}) 
				+ \mu \widehat{k}_n \| \nabla u^h_{n,\beta} \|^2 + \gamma \widehat{k}_n \| w_{n,\beta}^h \|^2
				+\rho \widehat{k}_n \| u^h_{n,\beta} \|^2 + \lambda\widehat{k}_n \| u^h_{n,\beta} \|_{L^4}^4 
				\\
				&
				+ \kappa(\nabla \phi_{n,\beta}^h,\nabla  \phi_{n,\alpha}^h)+ \kappa \big( F ( \phi_{n+1,\theta}^h ) - F ( \phi_{n,\theta}^h ), 1 \big) = -\sigma\widehat{k}_n (m_{n,\beta}^h,m_{n,\beta}^h).
			\end{split}
		\end{align*}
		Next, we apply $G$-stability identity in \eqref{G-stable} to 
		the above equation and achieve 
		\begin{align*}
			&\begin{Vmatrix}
				u^h_{n+1} \\
				u^h_{n} \\
			\end{Vmatrix}_{G(\theta )}
			+\kappa \begin{Vmatrix}
				\nabla \phi^h_{n+1} \\
				\nabla \phi^h_{n} \\
			\end{Vmatrix}_{G(\theta )}
			+ \kappa \int_{D} F(\phi_{n+1,\theta}^h) dx 
			+ \mu \widehat{k}_n \| \nabla u^h_{n,\beta} \|^2 
			+ \gamma \widehat{k}_n \| w_{n,\beta}^h \|^2
			\\
			&+ \!\Big\| \sum_{\ell=0}^{2} a_{\ell}^{(n)} u^h_{n-1+\ell} \Big\|^2
			\!\!+\! \Big\| \sum_{\ell=0}^{2} a_{\ell}^{(n)} \nabla \phi^h_{n-1+\ell} \Big\|^2
			\!\!+\! \rho\widehat{k}_n \| u_{n,\beta}^h \|^2
			\!+\! \lambda \widehat{k}_n \| u^h_{n,\beta} \|_{L^4}^4
			\!+\! \sigma \widehat{k}_n \| m_{n,\beta}^h \|^2
			\\
			&\leq 
			\begin{Vmatrix}
				u^h_{n} \\
				u^h_{n-1} \\
			\end{Vmatrix}_{G(\theta )}
			+ \kappa \begin{Vmatrix}
				\nabla \phi^h_{n} \\
				\nabla \phi^h_{n-1} \\
			\end{Vmatrix}_{G(\theta )}
			+  \kappa \int_{D} F \big( \phi_{n,\theta}^h \big) dx 
		\end{align*}
		which implies \eqref{Energy dissipation inequality} under the assumption $\rho \geq 0$.
	\end{proof}

	\section{Numerical results}  
	\label{sec:sec5}
	In this section, we present several numerical experiments to validate theoretical results of \Cref{fully discrete formulations scheme}
	(with the parameter $\theta = 0.3$).  
	For spatial discretization, we use Taylor-Hood $\mathbb{P}2/\mathbb{P}1$ finite element space for the velocity $u$ and pressure $p$, $\mathbb{P}2$ space for other variables throughout all experiments. 
	We first conduct one convergence test to verify both the spatial and temporal accuracy of \Cref{fully discrete formulations scheme},
	and then examine its performance for the phase-field shape relaxation and fluid self-organization in active fluid with 
	the setup of parameters in~\cite{MR4736040,MR4500252} and random initial conditions for the velocity field.
	This investigation serves to assess the robustness of \Cref{fully discrete formulations scheme} as well as the effectiveness of the time-adaptive strategy following the minimal dissipation criterion.
	
	\subsection{Convergence Test}
	To validate the convergence rate of Scheme \ref{fully discrete formulations scheme} in both space and time, we construct the test problem on the unit square domain \( D = [0,1] \times [0,1] \) with the following exact solution
	\begin{align*}
		&\begin{bmatrix}
			u_{1} \\ u_{2}
		\end{bmatrix} = 
		\begin{bmatrix}
			( -\cos(2\pi x+\pi)-1)\sin(2 \pi y) \exp(2t) \\
			-\sin(2\pi x) \cos(2\pi y) \exp(2t)
		\end{bmatrix},  \\
		&\begin{bmatrix}
			w_{1} \\ w_{2}
		\end{bmatrix} = 
		\begin{bmatrix}
			(-3x^2 + 3y^2 + 8\pi^2 \sin(2\pi y) \cos(2\pi x) - 4\pi^2 \sin(2\pi y)) \exp(2t) \\
			(6xy - 8\pi^2\sin(2\pi x)\cos(2\pi y))\exp(2t)
		\end{bmatrix},  \\
		&\ \ \xi = (x^3-3xy^2)\exp(2t), \quad  p = \sin(3\pi^2x)\cos(3\pi^2 y)\exp(-t)
		\\
		&\ \ \phi =\big( \cos(4\pi x) \cos(4\pi y)+\cos(3\pi x) \cos(3\pi y)\big)\exp(-t),
		\quad m = - \Delta \phi + \phi^3-\phi.
	\end{align*}
	We set the physical parameters to be \( \mu = 1 \), \( \gamma = 1 \), \( \nu = 1 \), \( \rho = 1 \), \( \lambda = 1 \), \(\kappa=1\), \(\sigma=1\), and simulate the problem on the time interval $[0,1]$.
	The exact solution decides the source function and boundary conditions.

	We set the constant time step size \( \Delta t = \frac{1}{4}, \frac{1}{8}, \frac{1}{16}, \frac{1}{32} \) and fix the uniform mesh diameter \( h = \frac{1}{64} \) to verify the convergence rate in time. 
	Meanwhile we adjust \( h = \frac{1}{4}, \frac{1}{8}, \frac{1}{16}, \frac{1}{32} \) and keep \( \Delta t = 1 \times 10^{-5} \) to confirm the convergence rate in space. 
	The convergence rates are obtained by least squares fitting of the logarithmic values of the errors with respect to the time step \( \Delta t \) and uniform mesh size \( \Delta h \), and thus approximated by the slope of the best-fit line.

	As shown in \Cref{tab:$L^2$-errors and convergence rates in time}, the observed temporal convergence rates for the velocity \( u \), auxiliary variable \(w\), \(\xi\), pressure \( p \), phase field $\phi$, and $m$
	are second-order as expected. 
	Table~\ref{tab:$L^2$-errors and convergence rates in space} and \ref{tab:$H^1$-errors and convergence rates in space} demonstrate that the spatial convergence rates for the velocity component \( u \), auxiliary variable \( w \), phase field $\phi$ and $m$ are third-order in the \( L^2 \) and second-order in the \( H^1 \) norms, while the convergence rates for both \(\xi\) and pressure \( p \) are second-order in \( L^2 \)-norm and first-order in the \( H^1 \)-norm.

	\begin{table}
		\centering
		\caption{$L^2$-errors and convergence rates in time}
		\begin{tabular}{@{}lllllll@{}}
			\hline
			$1 / \Delta t$ & $\|u - u^h\|$ & $\|w - w^h\|$ & $\|\xi- \xi^h\|$ & $\|p - p^h \|$ & $\|\phi - \phi^h\|$ & $\|m - m^h\|$ \\
			\hline
			$4$   &3.01E-01  & 1.59E+01& 5.15E-01  & 1.59E+02& 6.67E-02& 2.87E+00 \\
			$8$   & 5.72E-02  &3.01E+00    & 9.78E-02& 3.01E+01  & 1.94E-02 & 6.89E-01  \\
			$16$ &1.18E-02   & 6.22E-01      & 2.02E-02   &6.21E+00   & 4.90E-03&1.77E-01 \\
			$32$ & 2.61E-03& 1.38E-01    & 4.52E-03 & 1.38E+00  & 1.22E-03& 2.21E-02  \\
			$Rate$ & 2.28 & 2.28    & 2.28   & 2.28  & 1.93 & 2.30  \\
			\hline
		\end{tabular}
		\label{tab:$L^2$-errors and convergence rates in time}
	\end{table}

	\begin{table}
		\caption{$L^2$-errors and convergence rates in space}
		\label{tab:L2_errors_convergence_space}
		\centering
		\begin{tabular}{@{}lllllllllllll@{}}
			\hline
			$1 / \Delta h$ & $\|u - u^h\|$  & $\|w - w^h\|$  & $\|\xi - \xi^h\|$ & $\|p - p^h\|$  & $\|\phi - \phi^h\|$ & $\|m - m^h\|$ \\
			\hline
			$16$   & 8.12E-04  & 6.20E-02   & 2.33E-03  & 9.40E-01  & 3.38E-03 & 5.64E+00 \\
			$32$   &1.02E-04   & 7.78E-03  & 3.31E-04& 8.48E-02  & 4.90E-04   & 7.83E-01  \\
			$64$ & 1.28E-05   & 9.74E-04    & 1.11E-04   &1.16E-02    & 6.35E-05& 2.08E-01 \\
			$128$ & 1.60E-06 & 1.22E-04    & 1.38E-05   & 2.35E-03  & 8.02E-06   & 2.58E-02   \\
			$Rate$ & 2.997 & 2.997   & 2.292  & 2.587  & 2.967  & 2.662   \\
			\hline
		\end{tabular}
		\label{tab:$L^2$-errors and convergence rates in space}
	\end{table}

	\begin{table}[h]
		\centering
		\caption{$H^1$-errors and convergence rates in space}
		\label{tab:H1-errors-space}
		\begin{tabular}{@{}lllllllllllll@{}}
			\hline
			$1 / \Delta h$ & $\|u - u^h\|$& $\|w - w^h\|$ & $\|\xi - \xi^h\|$  & $\|p - p^h\|$ & $\|\phi - \phi^h\|$ & $\|m - m^h\|$ \\
			\hline
			$16$   & 9.78E-02   & 7.56E+00   & 1.64E-01  & 2.54E+01 & 4.86E-01& 6.69E+02\\
			$32$   &2.47E-02  & 1.90E+00  & 7.71E-02 &1.06E+01  & 1.25E-01  & 2.21E+02 \\
			$64$ & 6.18E-03     & 4.76E-01  & 3.83E-02   &4.96E+00  & 3.17E-02 & 1.11E+02  \\
			$128$ & 1.55E-03  & 1.19E-01   & 1.91E-02  & 2.44E+00   & 7.95E-03  & 2.37E+01   \\
			$Rate$ & 1.997 & 1.998 & 1.007& 1.060  & 1.987& 1.611  \\
			\hline
		\end{tabular}
		\label{tab:$H^1$-errors and convergence rates in space}
	\end{table}

	\subsection{Phase-field spinodal decomposition and fluid self-organization in active fluid}
	\label{Phase-field spinodal decomposition and fluid self-organization in active fluid}
	We simulate the spinodal decomposition of binary fluids within the context of the Allen-Cahn active-fluid system \eqref{ACAFs}, using \Cref{fully discrete formulations scheme} to investigate the interface dynamics and the Allen-Cahn phase field coupled active fluid system. 
	The computational domain is \( D = [0, 1.5]^2 \), and the initial conditions for the phase field and velocity are given by
	\[
	\phi_0(x,y) = \big( \text{rand}_1(x,y), \, \text{rand}_1(x,y) \big),
	\ \ 
	u_0(x,y) = \big( \text{rand}_2(x,y), \, \text{rand}_2(x,y) \big), 
	\ \ (x,y) \in D, 
	\]
	where \(\text{rand}_1(x, y)\) represents a uniform random variable in the interval \([-0.001, 0.001]\), and \(\text{rand}_2(x, y)\) represents a uniform random variable in the interval \([-1, 1]\).
	The boundary conditions are $u|_{\partial D} = \Delta u|_{\partial D} = 0$ and $\partial_n \phi = \partial_n m = 0$ on $\partial D$.
	The parameters are set to be $\mu = \nu = \beta = \alpha = \gamma = 1$, $\sigma = 100$, and $\kappa = 0.01$. 
	The simulation is carried out over the time interval $[0,2]$
	with the constant time step size $\Delta t = 0.01$ and the uniform mesh size \( h = \frac{1}{64} \).  

	\Cref{fig:Phase-field spinodal decomposition and self-organization in active fluid} shows the evolution of the velocity vector field, velocity magnitude, and phase field at selected time instances. 
	At \( t = 0 \), the phase-field graph displays the initial mixture of two distinct fluids, with the velocity field exhibiting random perturbations. 
	As time $t$ goes to 0.3, the velocity field forms small vortical structures as two fluids begin to separate. 
	By \( t = 1 \), the separation continues, and the velocity vortices become more pronounced. 
	At the final time \( t = 2 \), two fluids are fully separated, and the velocity field has developed into a well-organized flow structure with distinct vortices, illustrating the self-organization of the active fluid system. 
	This evolution captures both the dynamics of spinodal decomposition and the formation of velocity vortices, demonstrating the effectiveness of \Cref{fully discrete formulations scheme}.
	\Cref{fig:Energy dissipation} validates the discrete energy dissipation law in \eqref{Energy dissipation inequality}: the discrete model energy
 	declines rapidly in an extremely short time at the beginning, and then decreases steadily over the rest of the time. 

	\begin{figure}[htbp] 
		\centering
		\begin{minipage}[t]{0.24\linewidth}
			\centering
			\includegraphics[width=3.2cm]{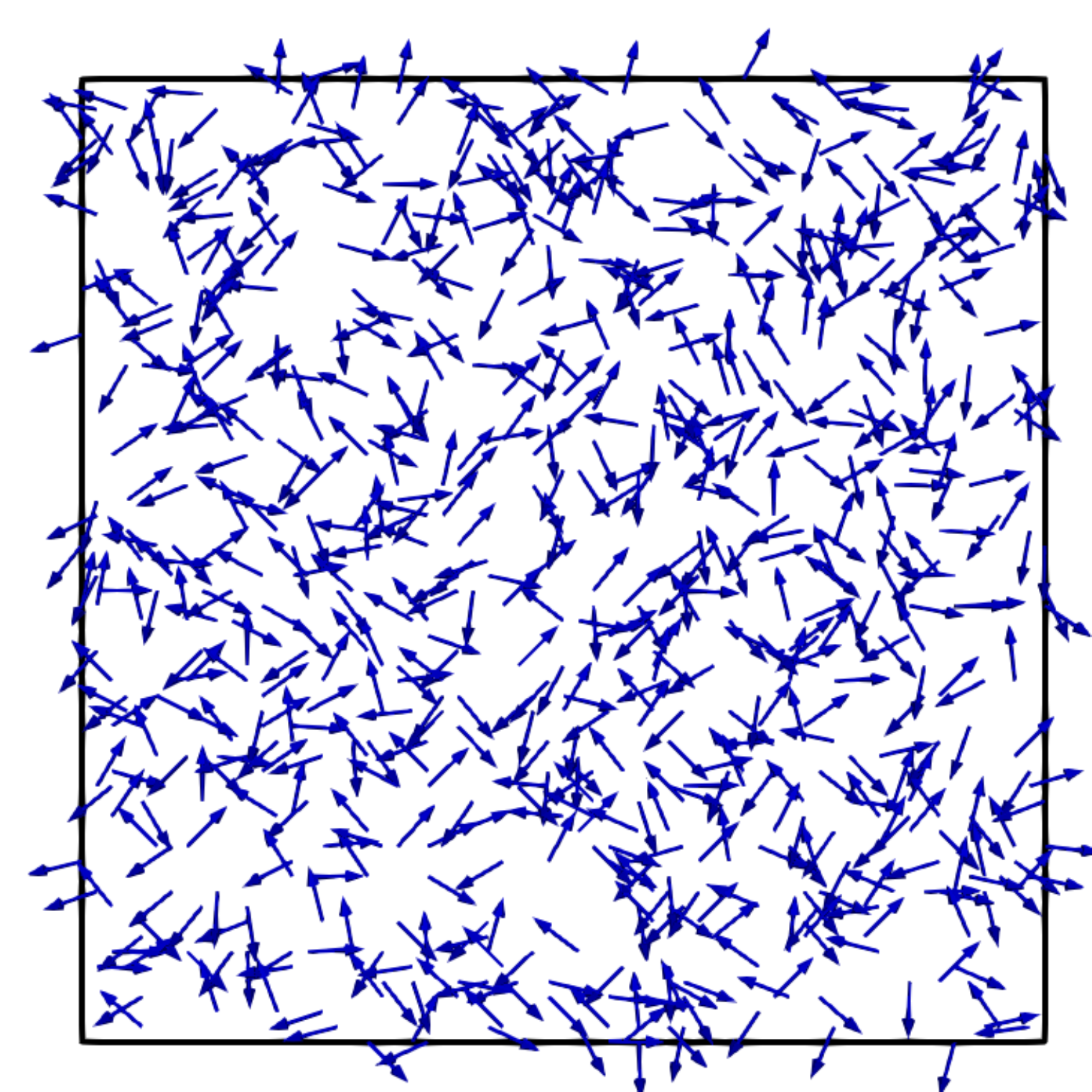} 
			% \\
			% \includegraphics[width=4cm]{Figure_1/velocity_t_0.eps}
			\caption*{$t=0$}
		\end{minipage}
		\begin{minipage}[t]{0.24\linewidth}
			\centering
			\includegraphics[width=3.2cm]{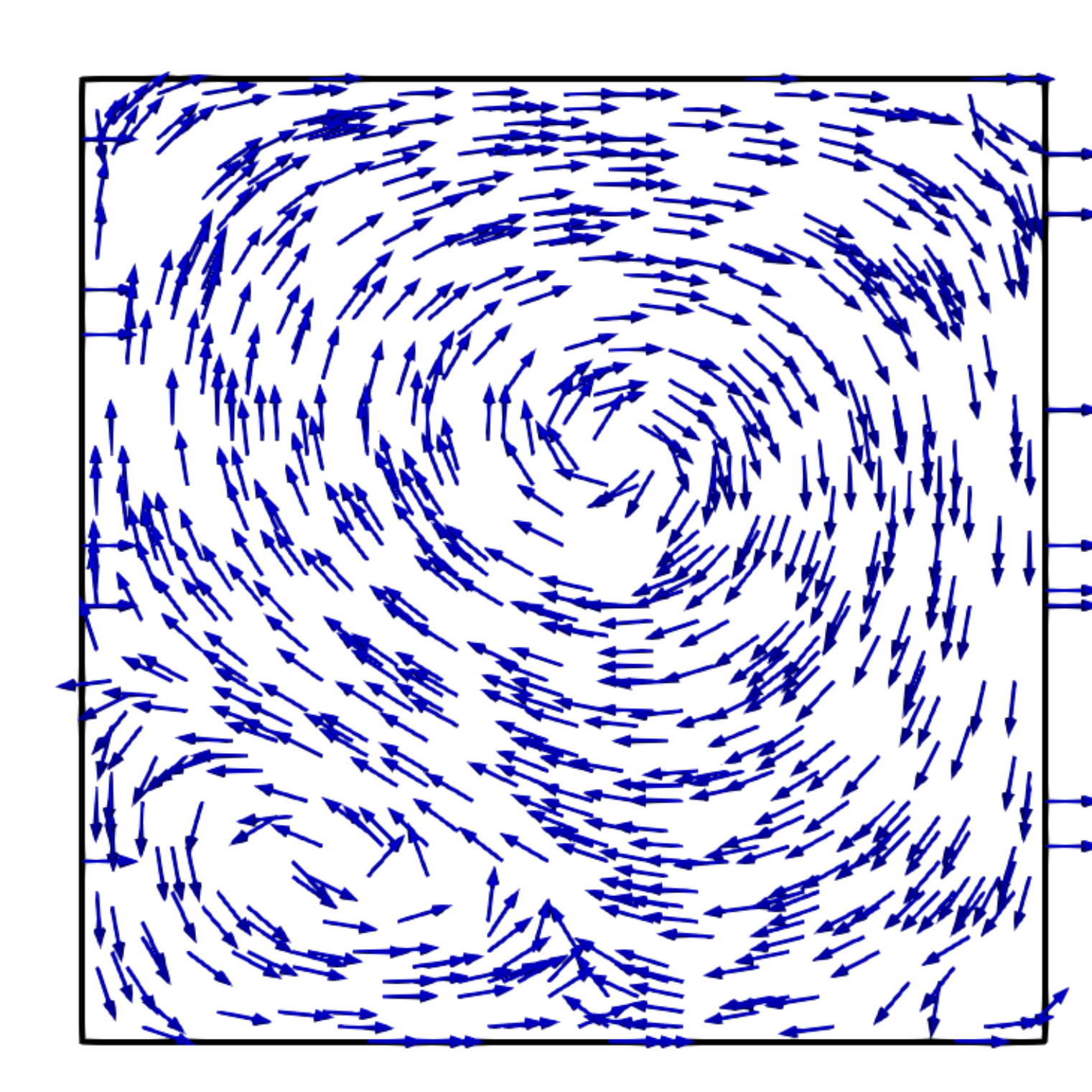} 
			% \\
			% \includegraphics[width=4cm]{Figure_1/velocity_t_0.15.eps}
			\caption*{$t=0.3$}
		\end{minipage}
		\begin{minipage}[t]{0.24\linewidth}
			\centering
			\includegraphics[width=3.2cm]{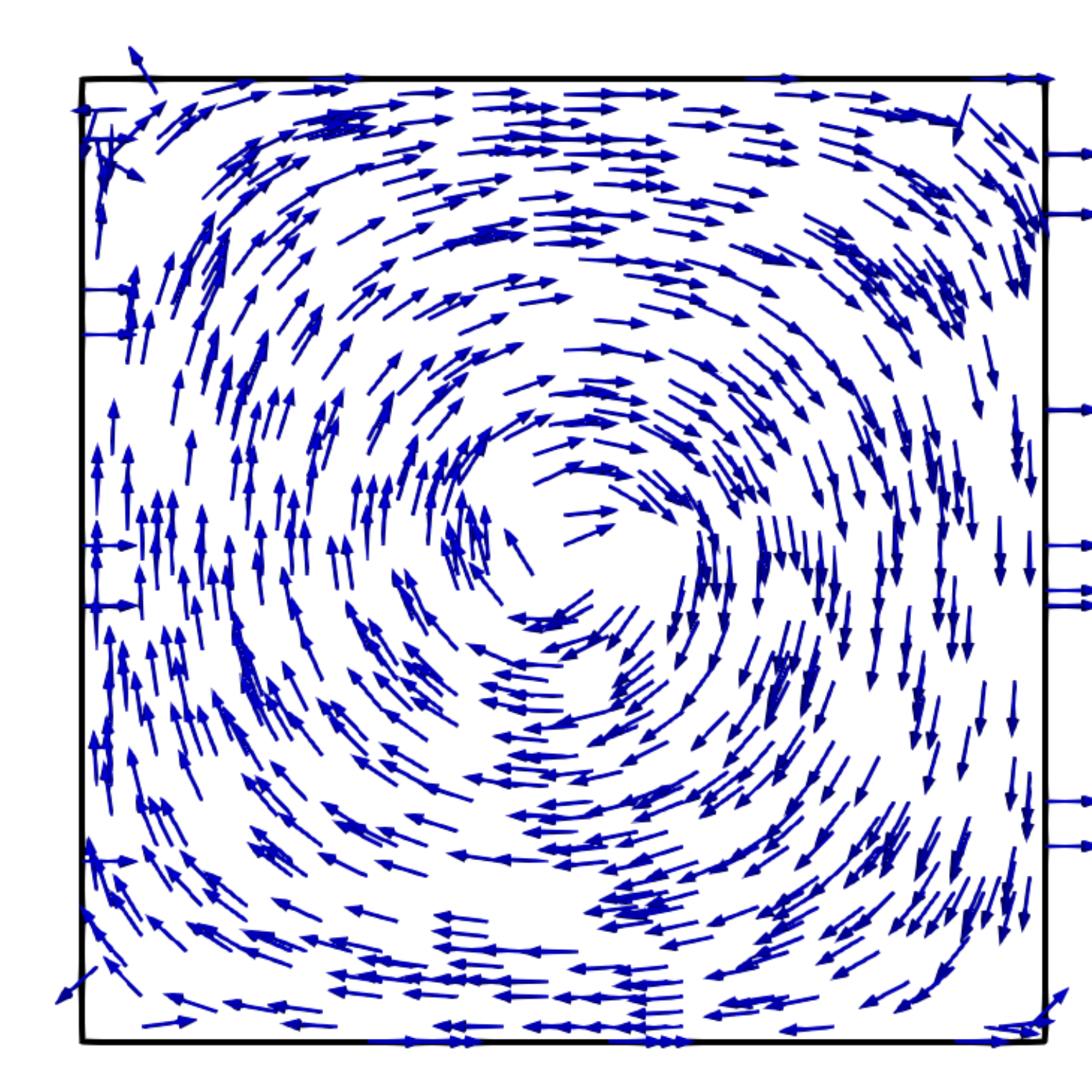} 
			% \\
			% \includegraphics[width=4cm]{Figure_1/velocity_t_0.3.eps}
			\caption*{$t=1$}
		\end{minipage} 
		\begin{minipage}[t]{0.24\linewidth}
			\centering
			\includegraphics[width=3.2cm]{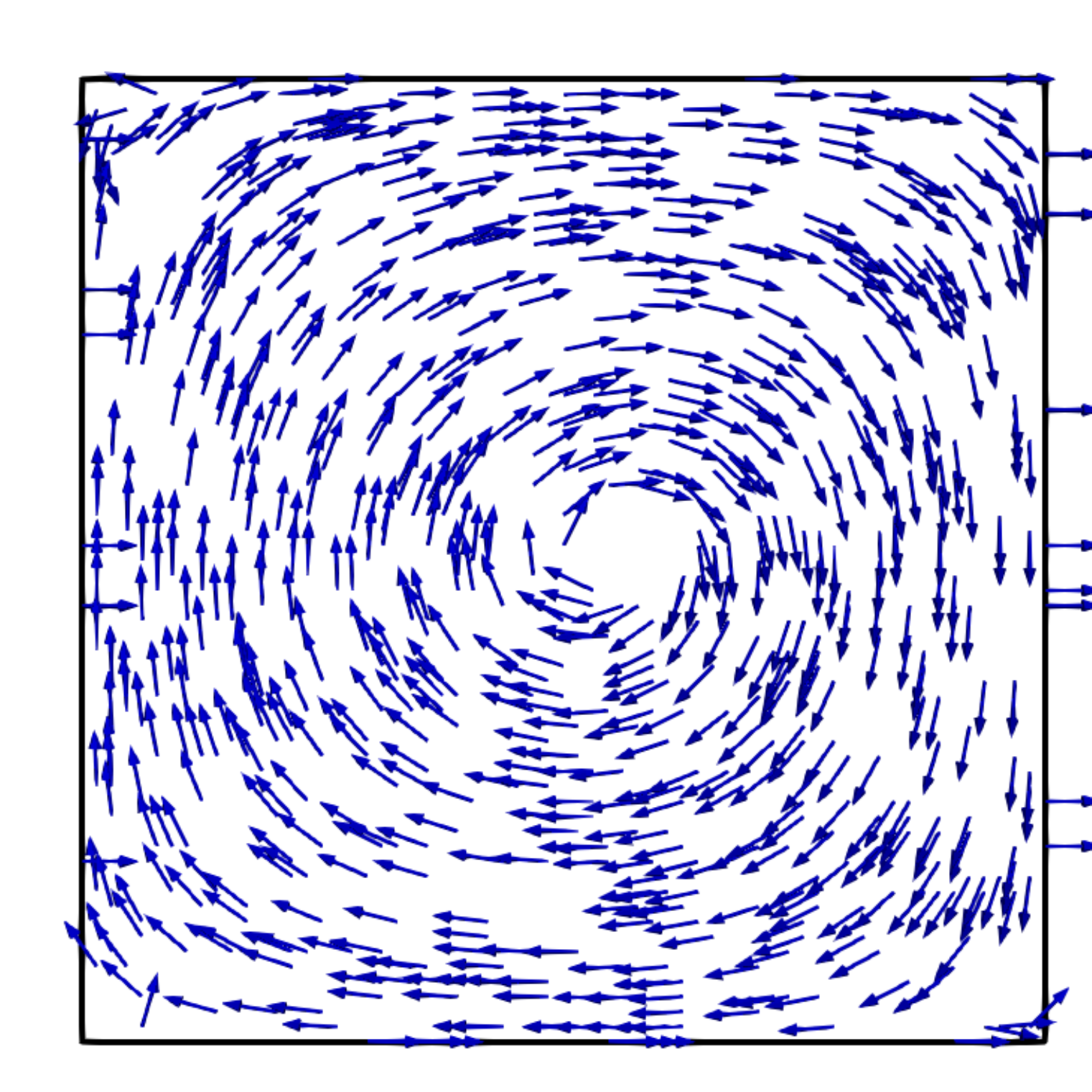} 
			% \\
			% \includegraphics[width=4cm]{Figure_1/velocity_t_0.3.eps}
			\caption*{$t=2$}
		\end{minipage} \\
		\vspace{-0.1cm}
		\begin{minipage}[t]{0.24\linewidth}
			\centering
			\includegraphics[width=3.1cm]{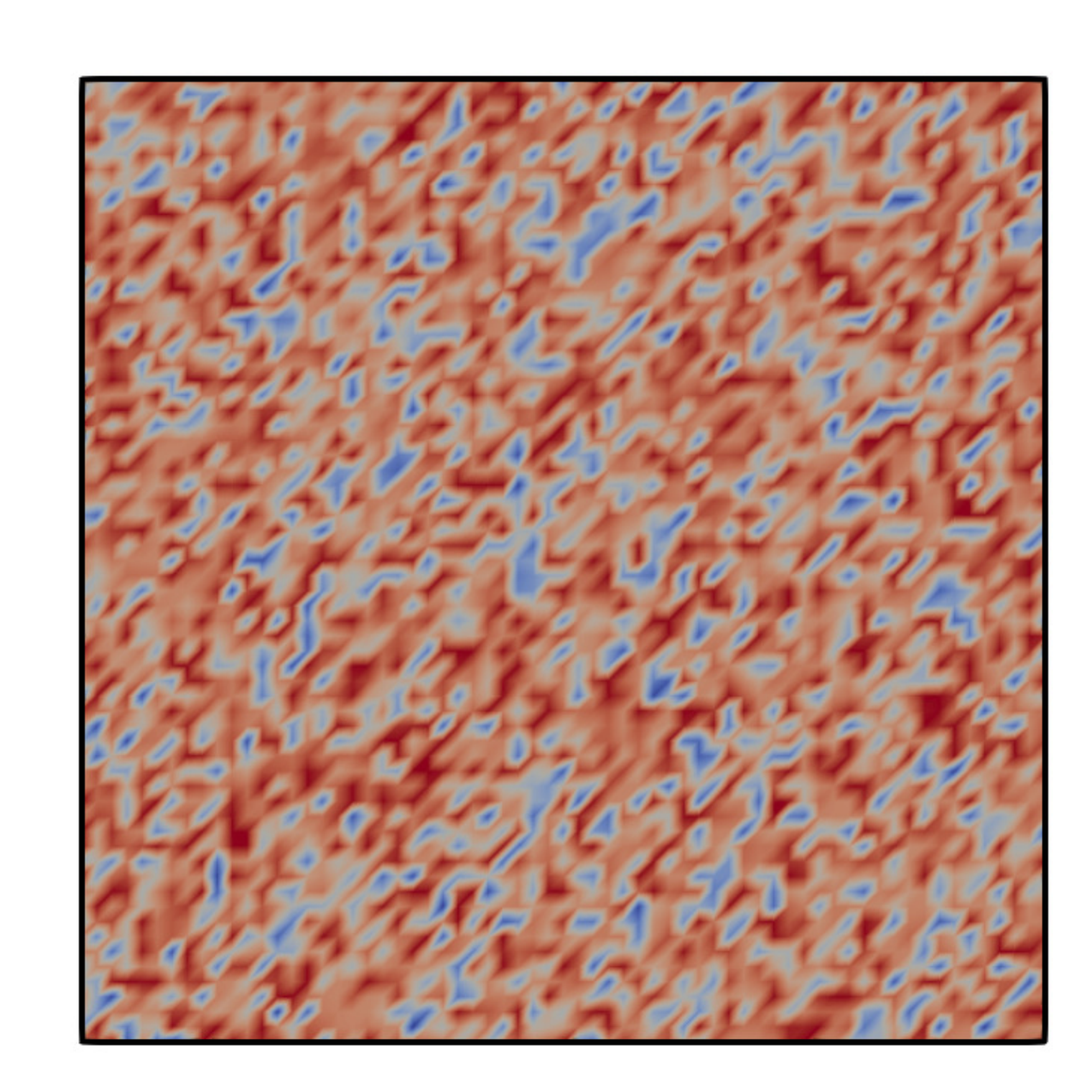} 
			% \\
			% \includegraphics[width=4cm]{Figure_1/velocity_t_0.eps}
			\caption*{$t=0$}
		\end{minipage}
		\begin{minipage}[t]{0.24\linewidth}
			\centering
			\includegraphics[width=3.1cm]{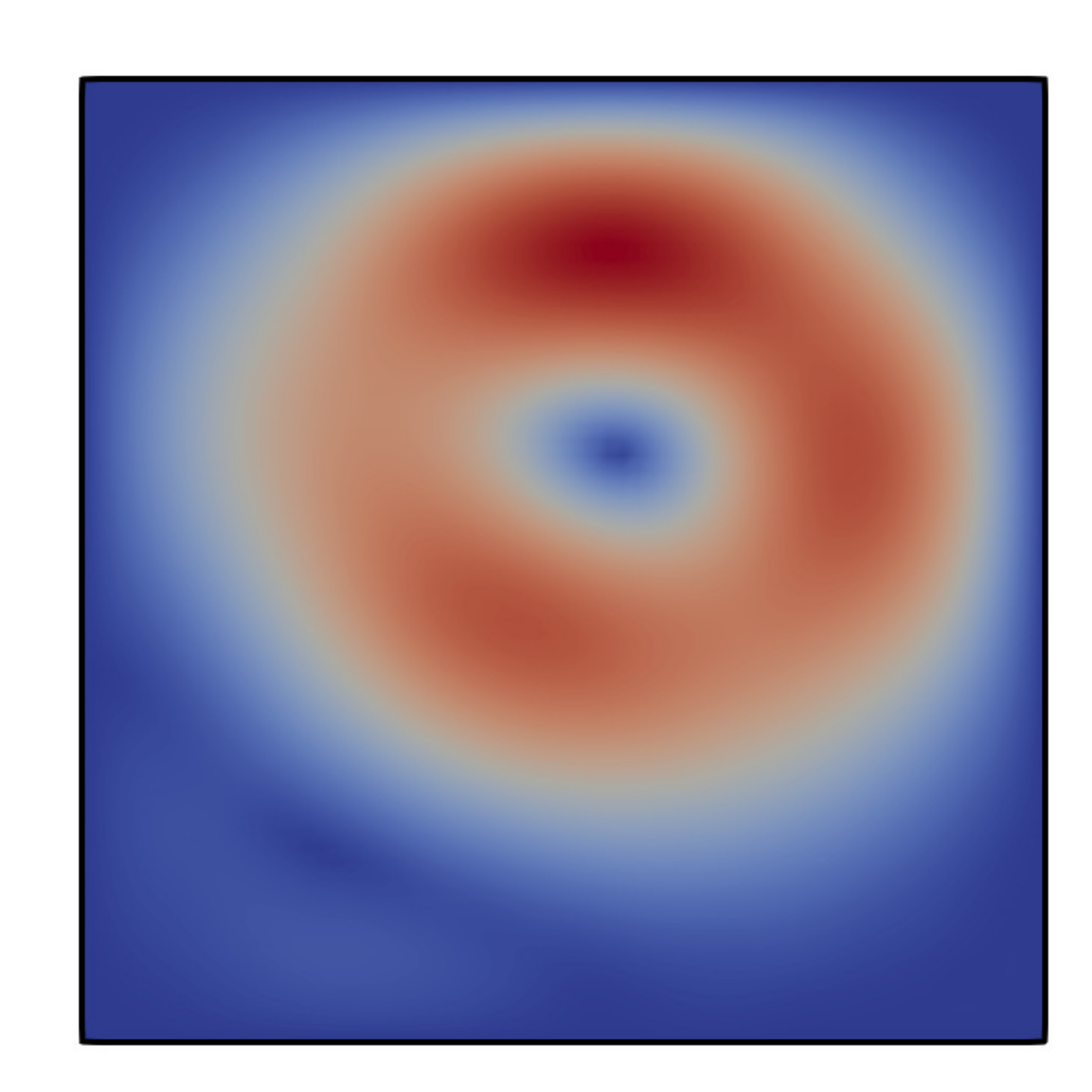} 
			% \\
			% \includegraphics[width=4cm]{Figure_1/velocity_t_0.15.eps}
			\caption*{$t=0.3$}
		\end{minipage}
		\begin{minipage}[t]{0.24\linewidth}
			\centering
			\includegraphics[width=3.1cm]{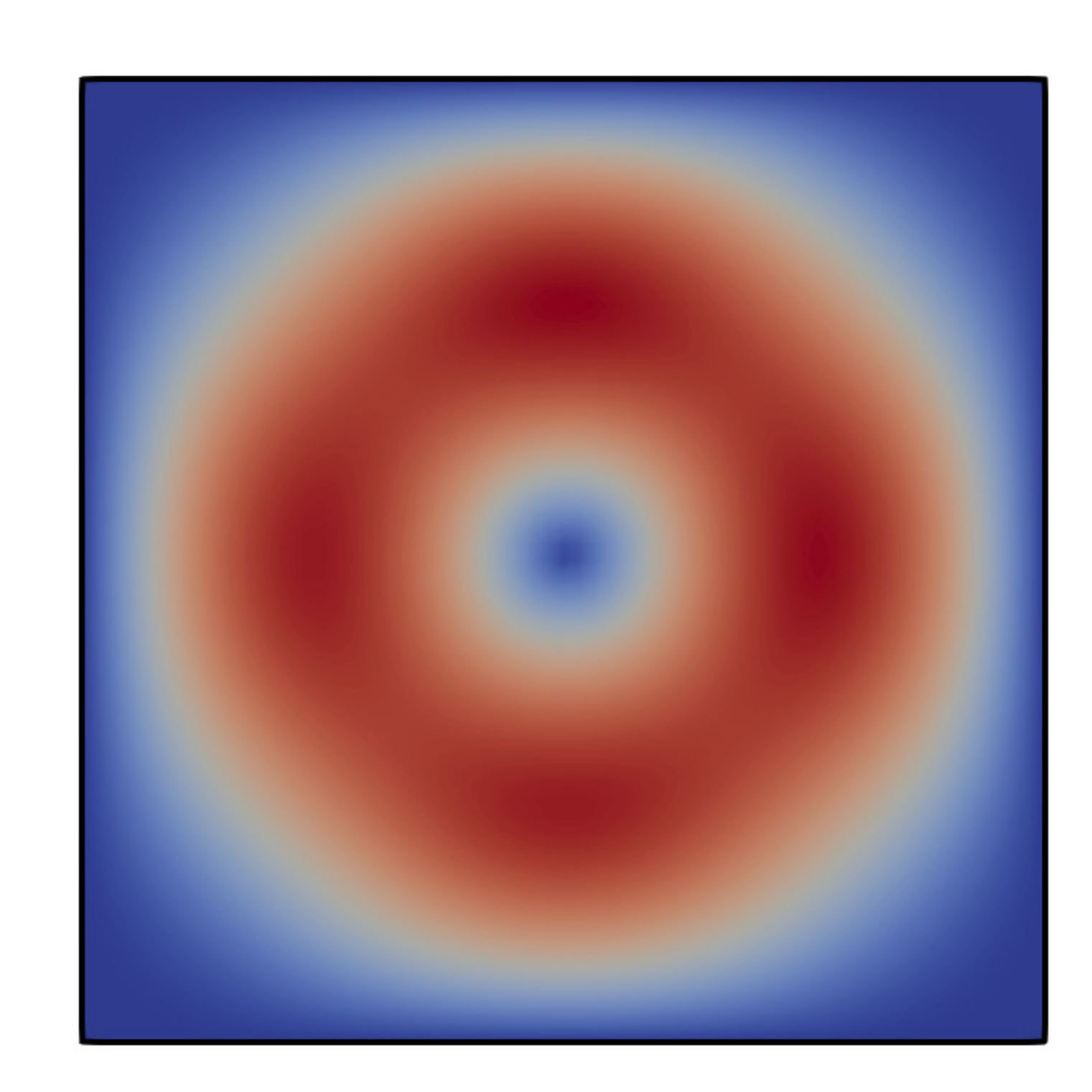} 
			% \\
			% \includegraphics[width=4cm]{Figure_1/velocity_t_0.3.eps}
			\caption*{$t=1$}
		\end{minipage} 
		\begin{minipage}[t]{0.24\linewidth}
			\centering
			\includegraphics[width=3.1cm]{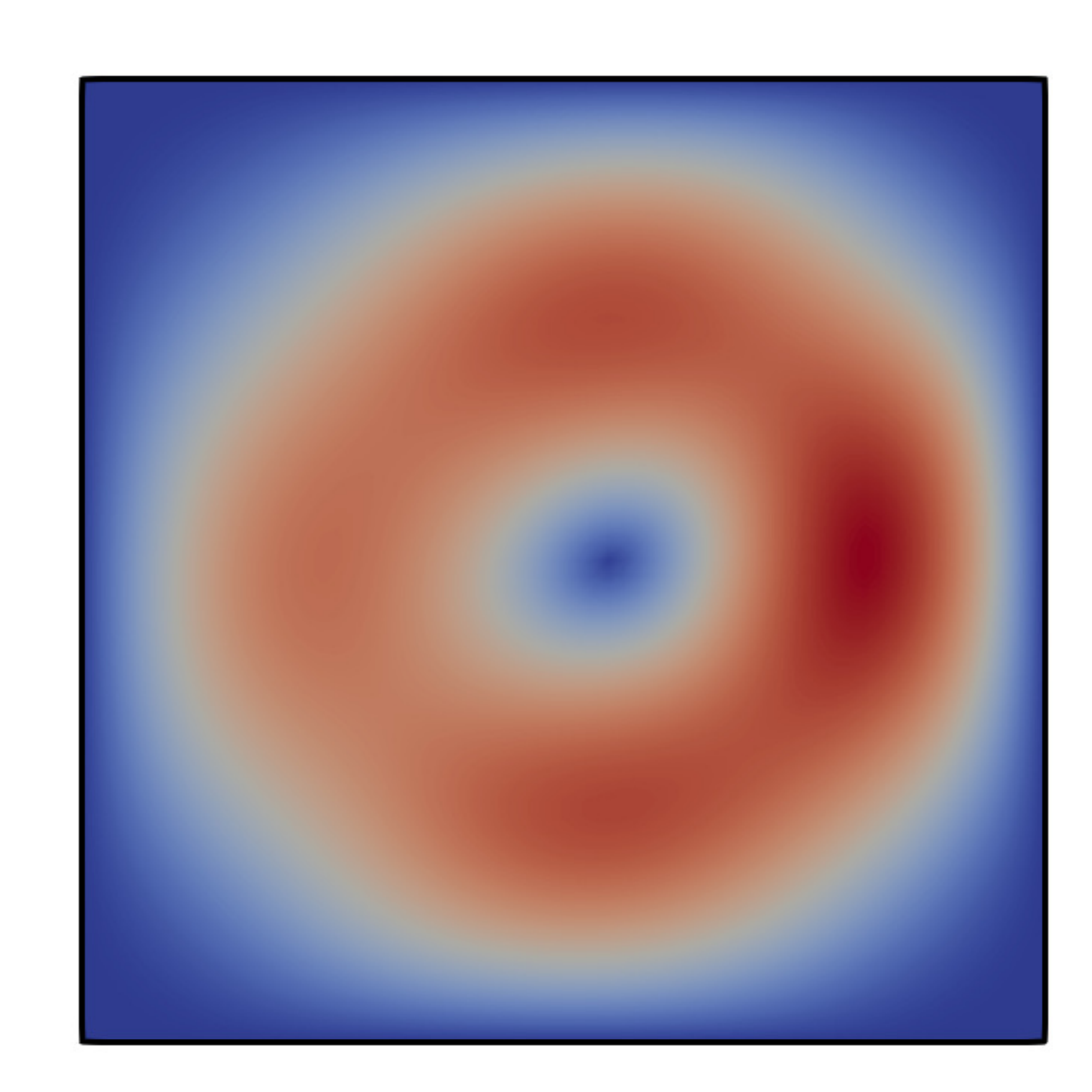} 
			% \\
			% \includegraphics[width=4cm]{Figure_1/velocity_t_0.1.eps}
			\caption*{$t=2$}
		\end{minipage}\\
		\vspace{-0.1cm}
		\begin{minipage}[t]{0.24\linewidth}
			\centering
			\includegraphics[width=3.1cm]{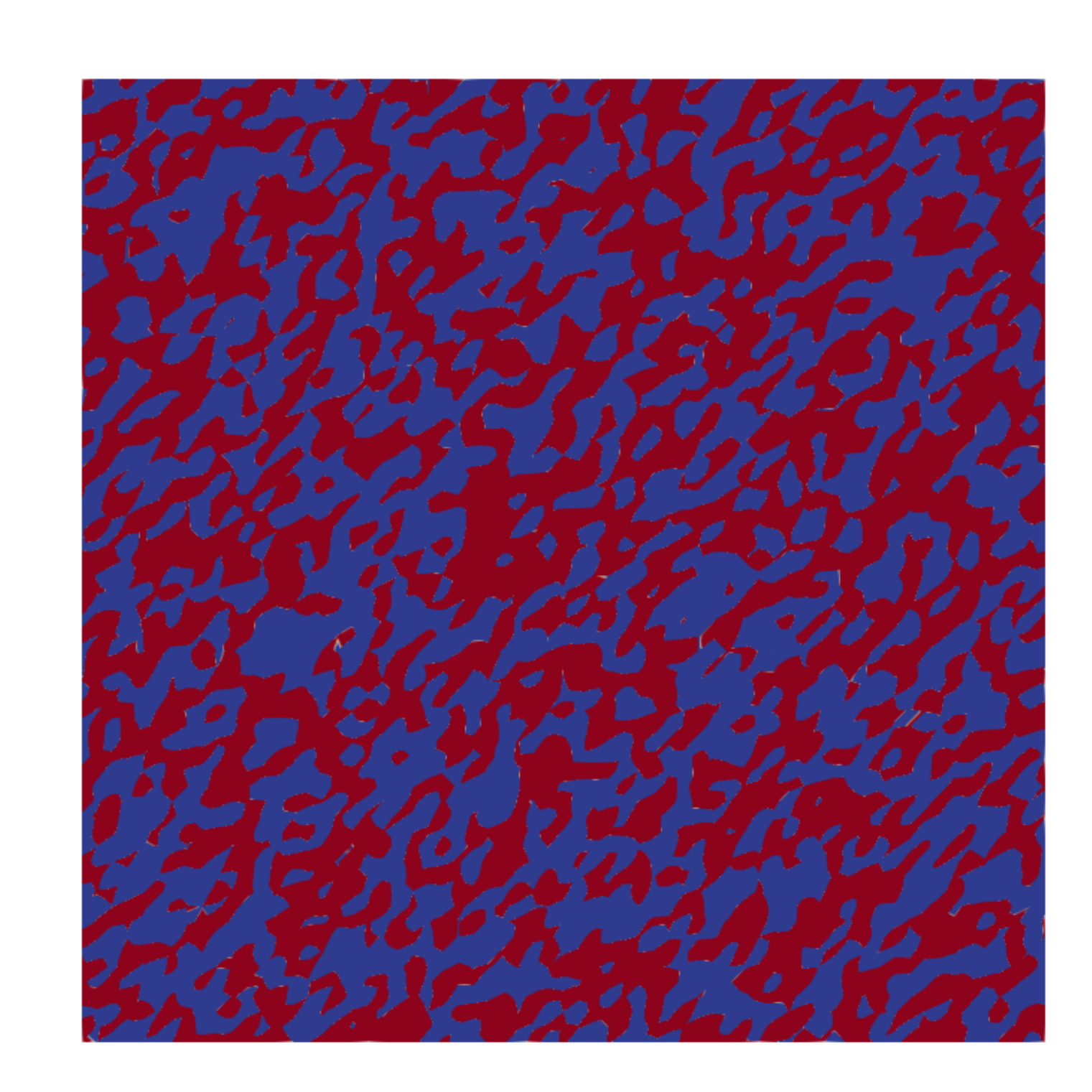} 
			% \\
			% \includegraphics[width=4cm]{Figure_1/velocity_t_0.eps}
			\caption*{$t=0$}
		\end{minipage}
		\begin{minipage}[t]{0.24\linewidth}
			\centering
			\includegraphics[width=3.1cm]{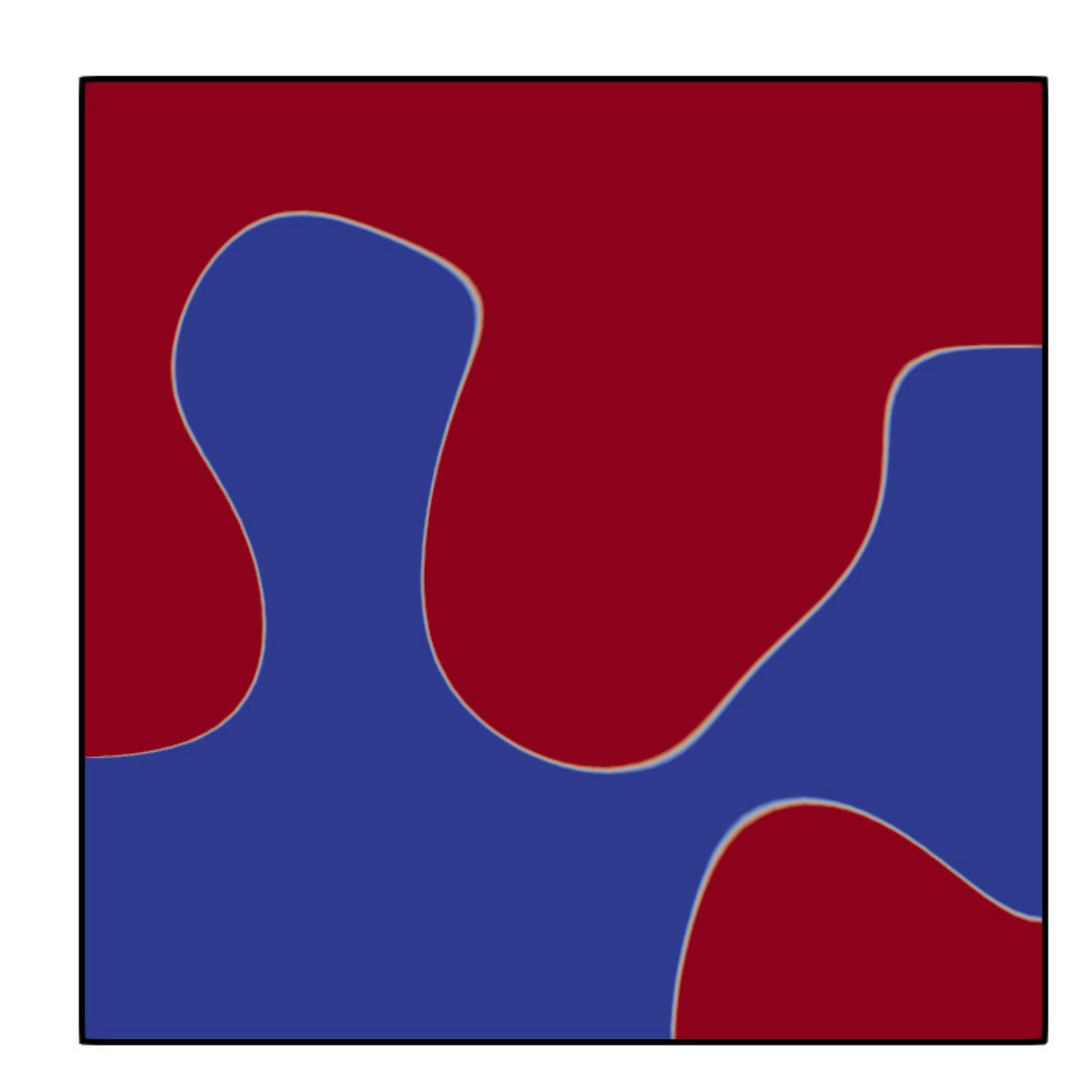} 
			% \\
			% \includegraphics[width=4cm]{Figure_1/velocity_t_0.15.eps}
			\caption*{$t=0.3$}
		\end{minipage}
		\begin{minipage}[t]{0.24\linewidth}
			\centering
			\includegraphics[width=3.1cm]{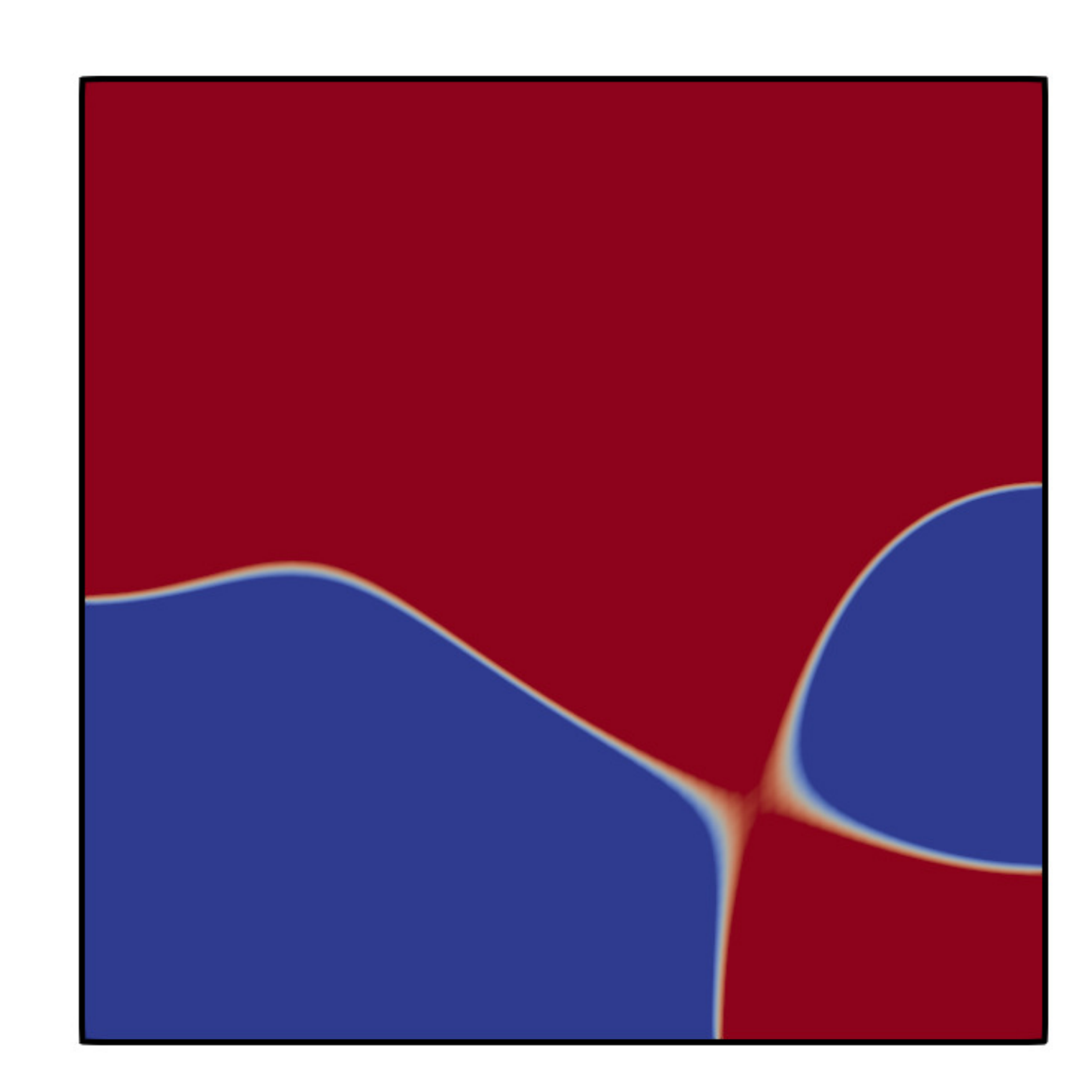} 
			% \\
			% \includegraphics[width=4cm]{Figure_1/velocity_t_0.3.eps}
			\caption*{$t=1$}
		\end{minipage} 
		\begin{minipage}[t]{0.24\linewidth}
			\centering
			\includegraphics[width=3.1cm]{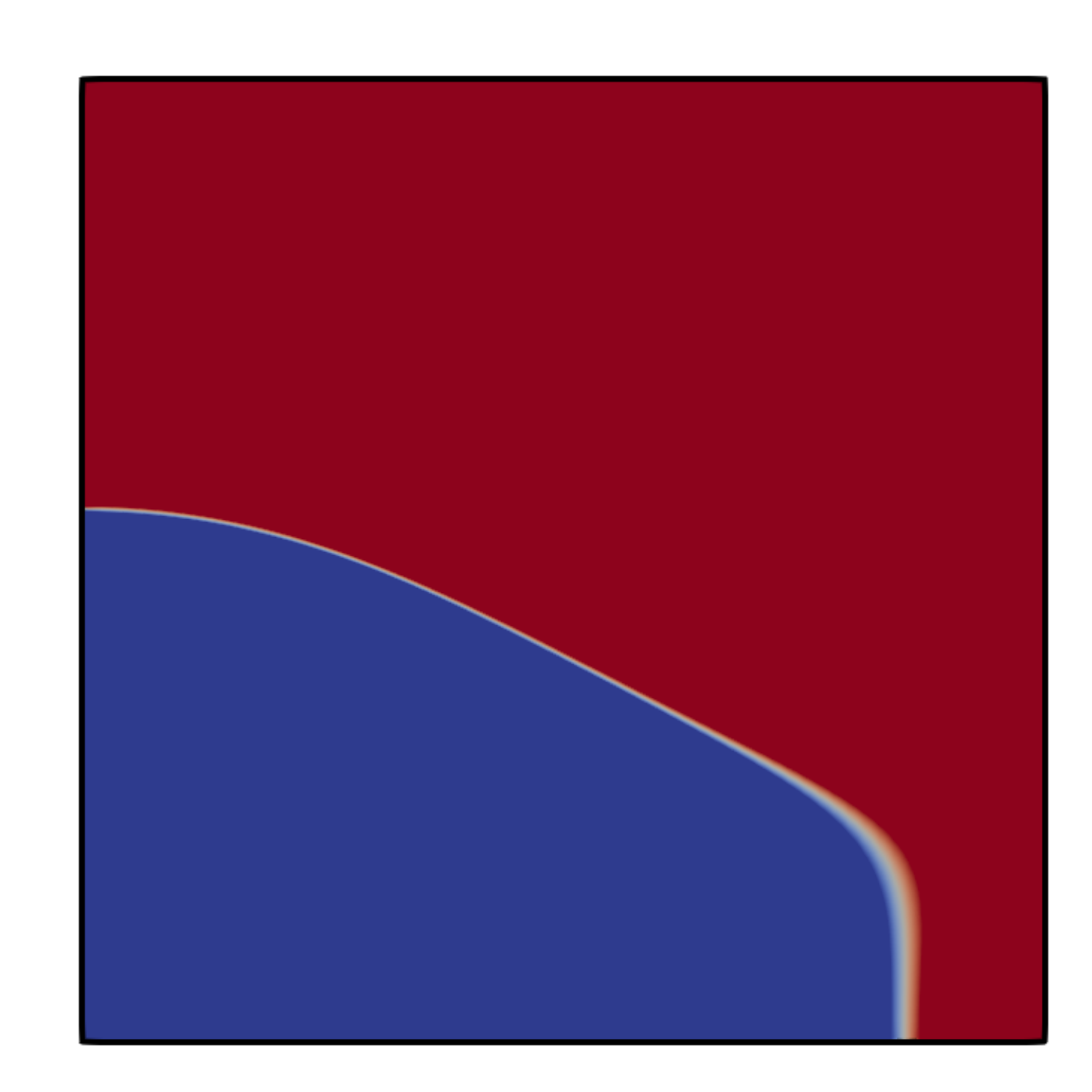} 
			% \\
			% \includegraphics[width=4cm]{Figure_1/velocity_t_0.1.eps}
			\caption*{$t=2$}
		\end{minipage}
		\vspace{-0.5cm}
		\caption{Time evolution of phase-Field spinodal decomposition and fluid self-organization in active fluid: the velocity field (top row), magnitude of the velocity field (middle row), and the phase field (bottom row) at selected time instances.}
		\label{fig:Phase-field spinodal decomposition and self-organization in active fluid}
	\end{figure}

	\begin{figure}[htbp]
		\centering
		\includegraphics[width=0.7\textwidth]{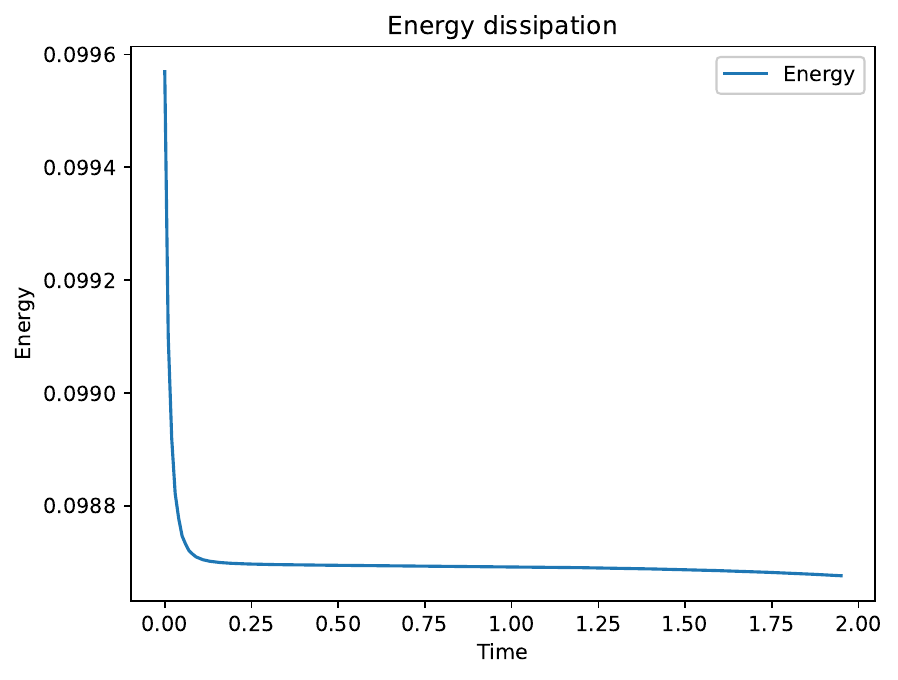}
		\caption{Energy dissipation in the phase-field spinodal decomposition and fluid self-organization. }
		\label{fig:Energy dissipation}
	\end{figure}

	\subsection{Phase-field shape relaxation and fluid self-organization in active fluid}
	\label{Phase-field shape relaxation and fluid self-organization in active fluid}
	In this subsection, we make use of \Cref{fully discrete formulations scheme} to simulate the merging process of two circular bubbles within the context of the Allen-Cahn active fluid system. 
	The computational domain is \( D = [0, 1.5]^2 \), and the initial condition for the phase field is 
	\[
	\phi_0(x,y) = 1 + \sum_{i=1}^{2} \tanh \Big( \frac{r_i - \sqrt{(x-x_i)^2+(y-y_i)^2}}{0.02} \Big),
	\]
	with two circular interfaces centered at $(x_1,y_1)=(0.5,0.75)$, $(x_2,y_2)=(1,0.75)$, respectively, and radius $r_1=r_2=0.25$. 
	The velocity is initialized by 
	\[
	u_0(x,y) = \big( \text{rand}_2(x,y), \, \text{rand}_2(x,y) \big), \qquad (x,y)\in D,
	\]
	where $\text{rand}_2(x,y)$ is the same random variable as in Subsection \ref{Phase-field spinodal decomposition and fluid self-organization in active fluid}. 
	The boundary conditions are: $u|_{\partial D} = \Delta u|_{\partial D} = 0$ and $\partial_n \phi = \partial_n m = 0$ on $\partial D$. 
	The parameters are set to be $\mu = \nu = \beta = \alpha = \gamma = 1$, $\sigma = 10$, $\kappa = 0.01$. 
	The simulation is carried out over the time interval $[0,0.4]$ with the constant time step size $\Delta t = 0.01$ and the uniform mesh size $h = \frac{1}{64}$. 

	Figure~\ref{fig:Phase-field shape relaxation and self-organization in active fluid} shows the evolution of the velocity vector field, velocity magnitude, and phase field at selected time instances. 
	At \( t = 0 \), two distinct phase-field bubbles are already present, while the velocity field exhibits random perturbations. 
	As time progresses to \( t = 0.1 \), the bubbles move closer together, and small vortical structures begin to emerge in the velocity field. 
	By \( t = 0.2 \), the bubbles continue to merge, and the vortices in the velocity field become more noticeable. 
	At \( t = 0.4 \), the two bubbles have fully merged into a single structure. In contrast, the velocity field has formed four distinct vortices arranged in a symmetric pattern, with each vortex rotating around a central point, which creates a well-organized flow structure. 
	The evolution process illustrates that \Cref{fully discrete formulations scheme} accurately captures the dynamics of the phase field
 	and velocity vortices in the active fluid system of self-organization.

	\begin{figure}[htbp] 
		\centering
		\begin{minipage}[t]{0.24\linewidth}
			\centering
			\includegraphics[width=3.2cm]{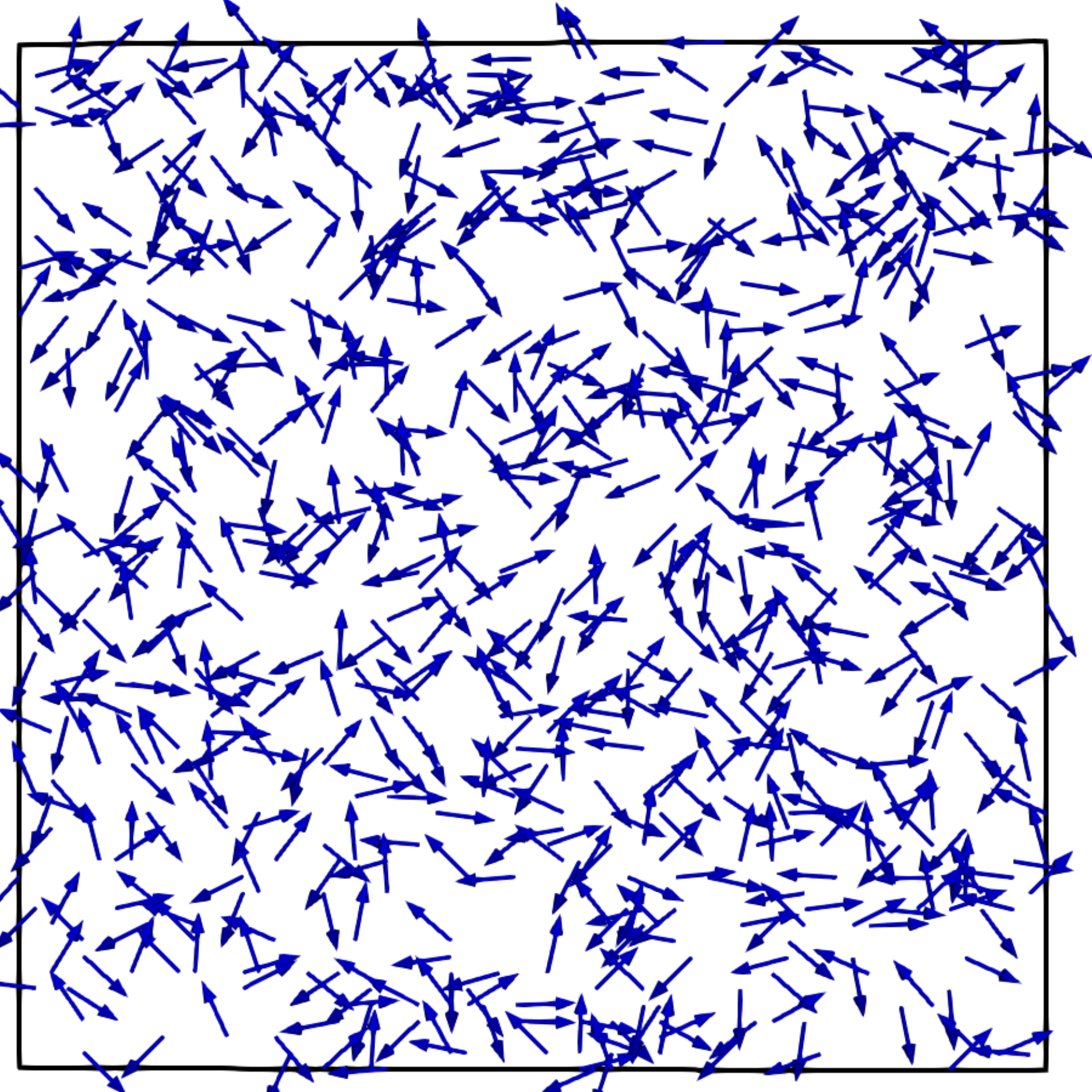} 
			% \\
			% \includegraphics[width=4cm]{Figure_1/velocity_t_0.eps}
			\caption*{$t=0$}
		\end{minipage}
		\begin{minipage}[t]{0.24\linewidth}
			\centering
			\includegraphics[width=3.2cm]{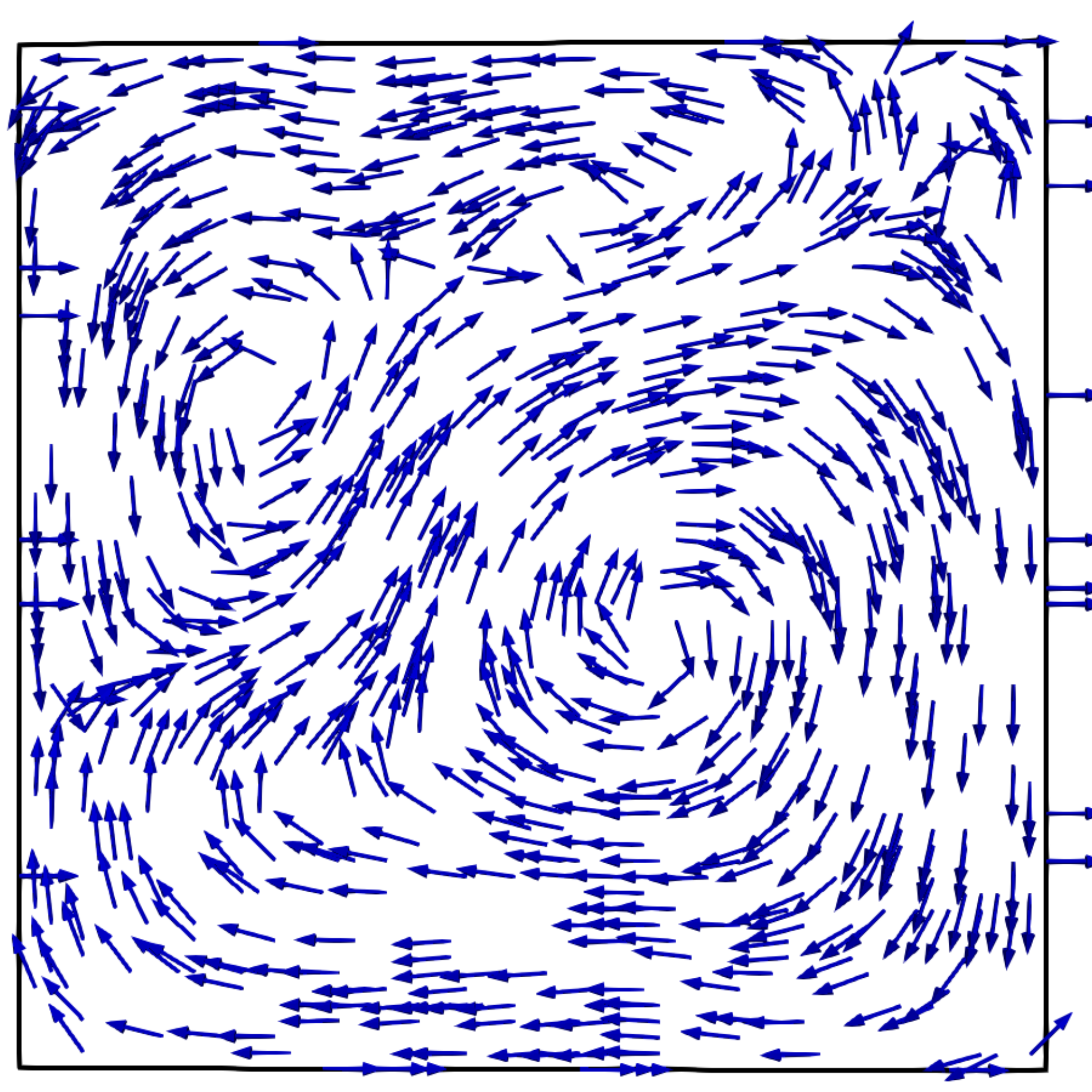} 
			% \\
			% \includegraphics[width=4cm]{Figure_1/velocity_t_0.15.eps}
			\caption*{$t=0.1$}
		\end{minipage}
		\begin{minipage}[t]{0.24\linewidth}
			\centering
			\includegraphics[width=3.2cm]{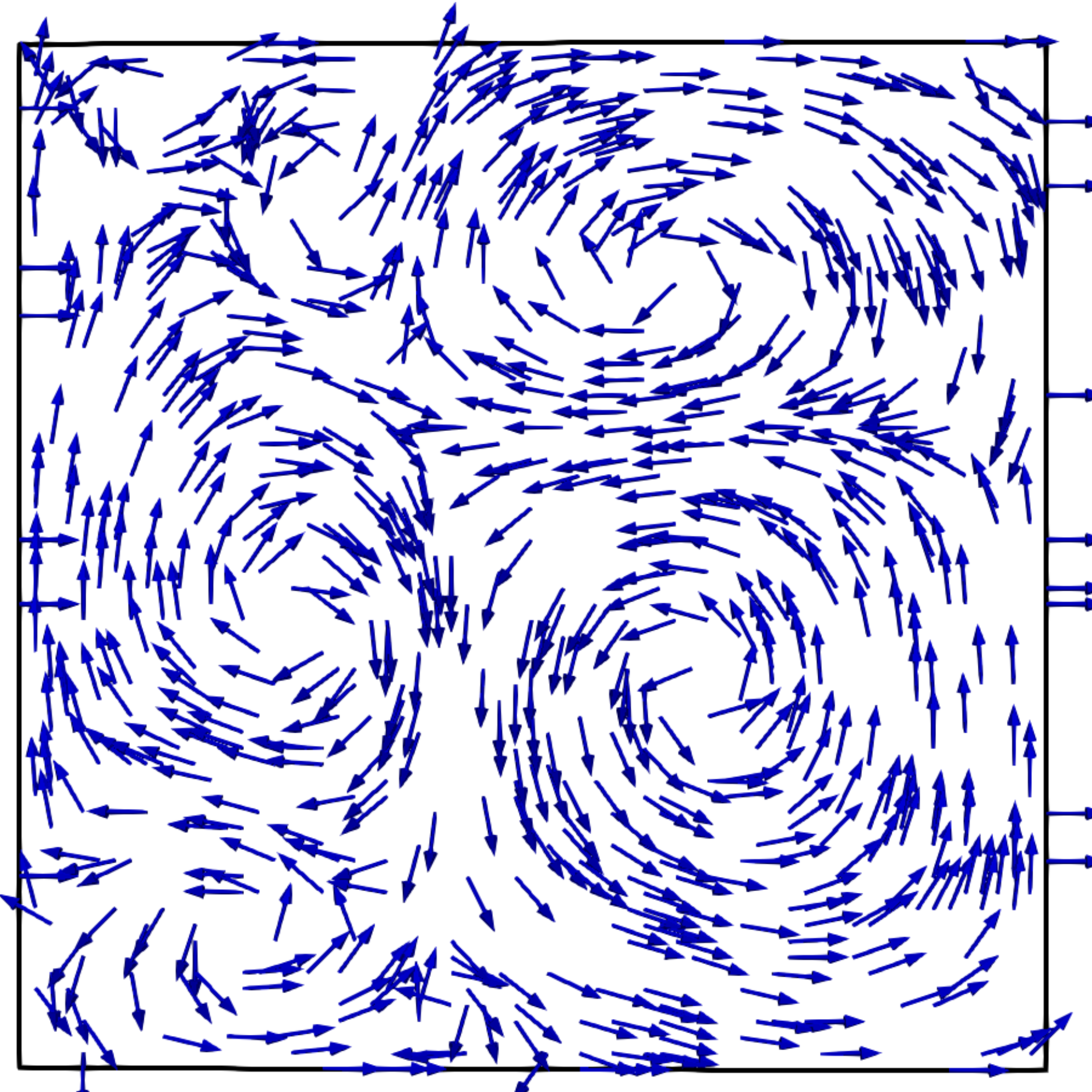} 
			% \\
			% \includegraphics[width=4cm]{Figure_1/velocity_t_0.3.eps}
			\caption*{$t=0.2$}
		\end{minipage} 
		\begin{minipage}[t]{0.24\linewidth}
			\centering
			\includegraphics[width=3.2cm]{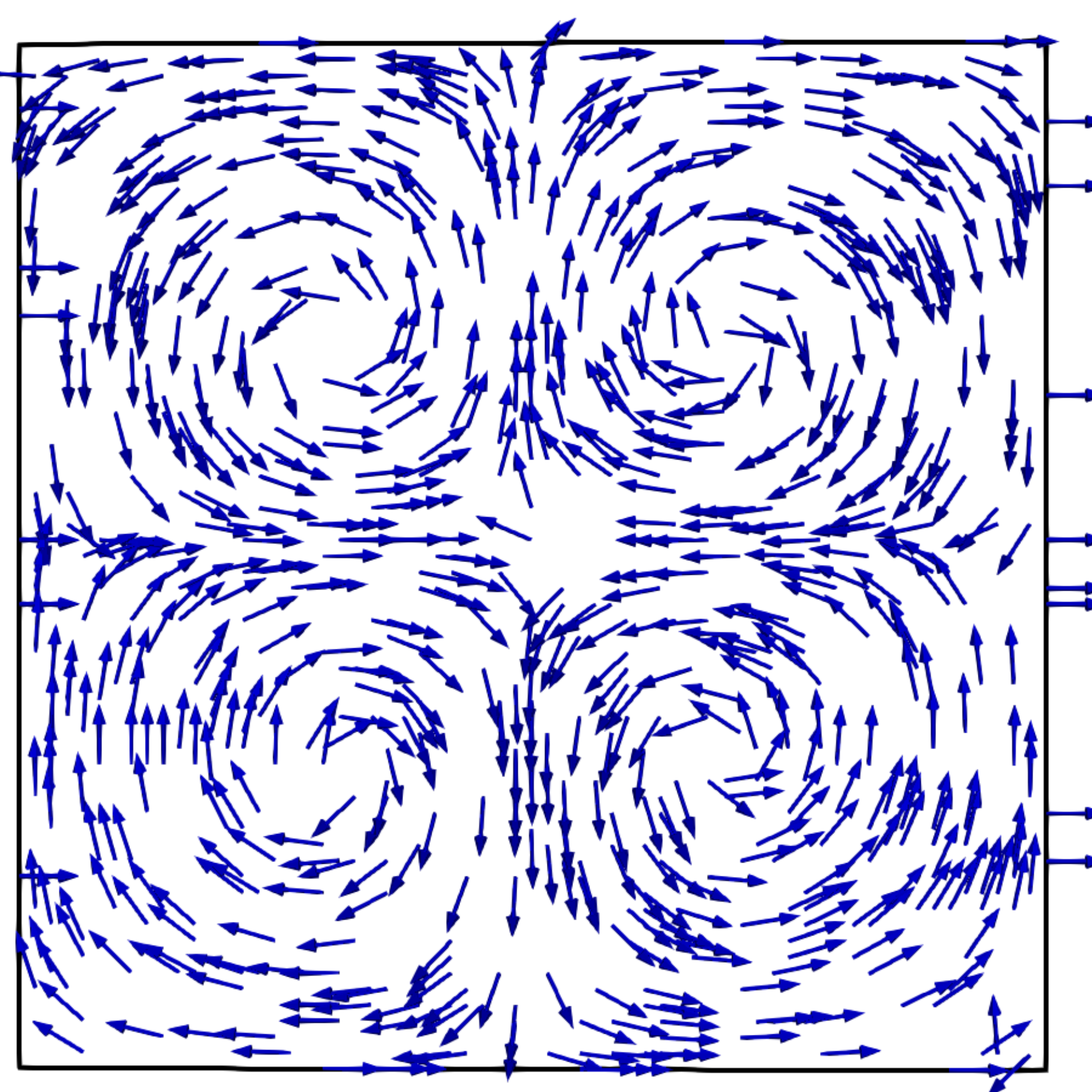} 
			% \\
			% \includegraphics[width=4cm]{Figure_1/velocity_t_0.3.eps}
			\caption*{$t=0.4$}
		\end{minipage} \\
		\vspace{-0.1cm}
		\begin{minipage}[t]{0.24\linewidth}
			\centering
			\includegraphics[width=3.1cm]{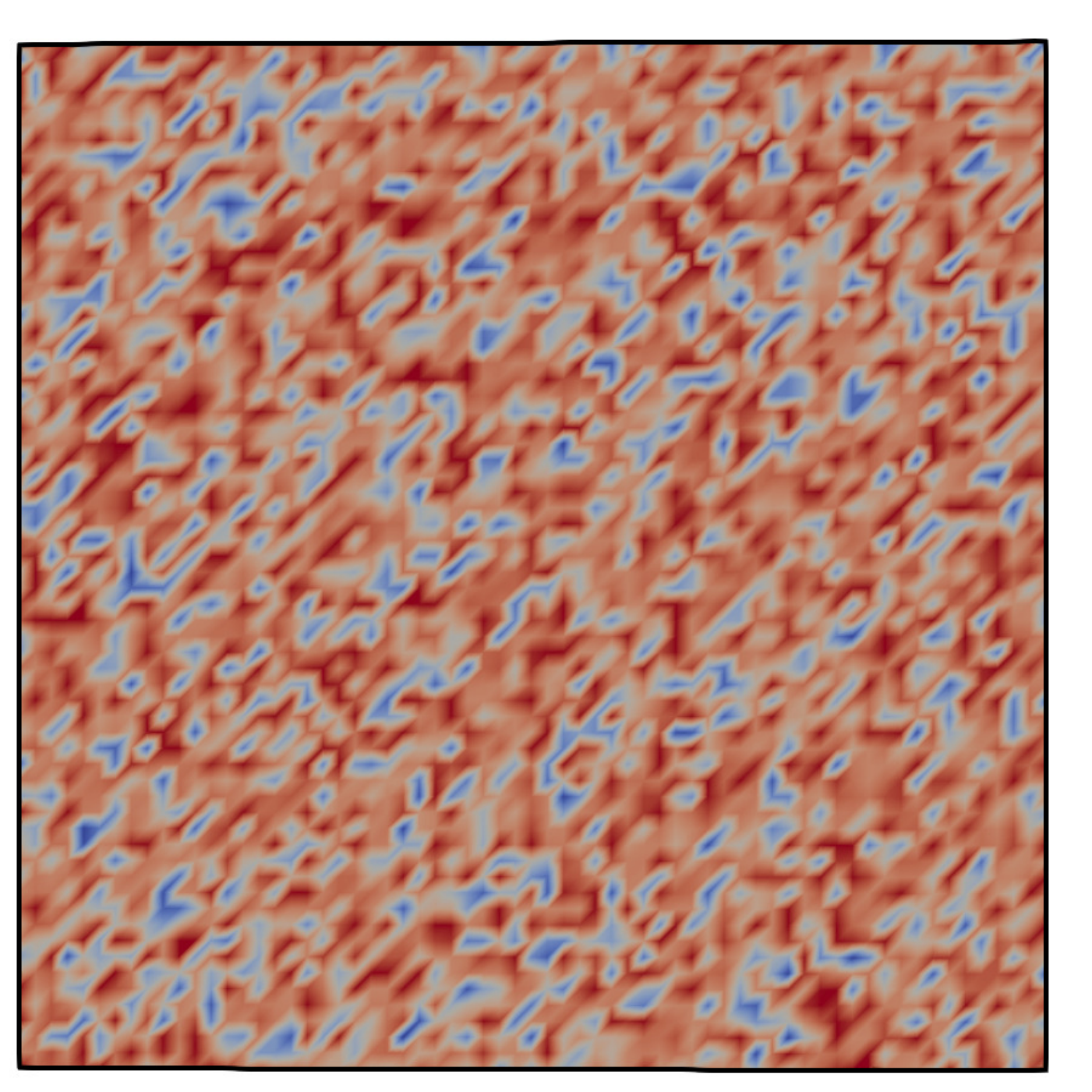} 
			% \\
			% \includegraphics[width=4cm]{Figure_1/velocity_t_0.eps}
			\caption*{$t=0$}
		\end{minipage}
		\begin{minipage}[t]{0.24\linewidth}
			\centering
			\includegraphics[width=3.1cm]{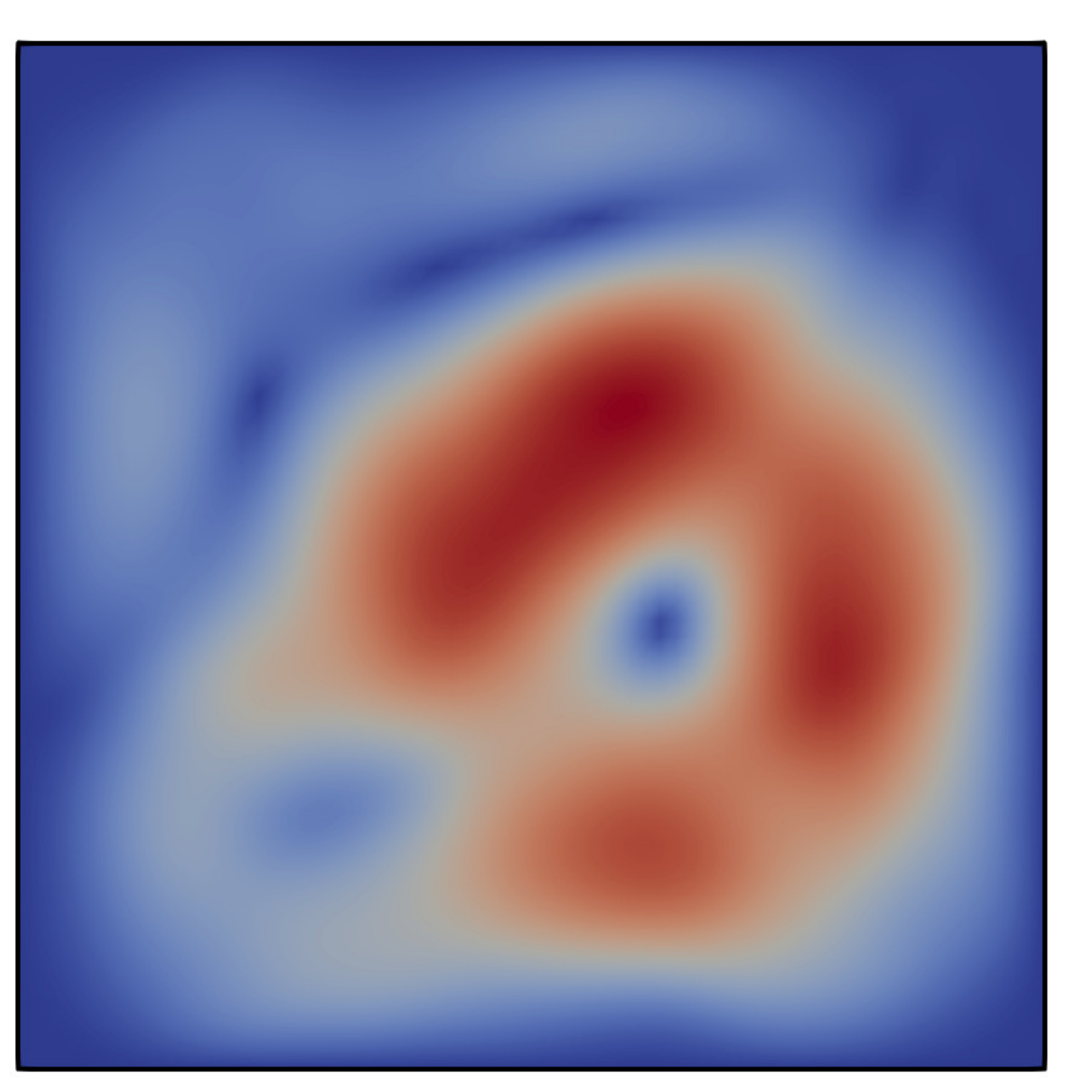} 
			% \\
			% \includegraphics[width=4cm]{Figure_1/velocity_t_0.15.eps}
			\caption*{$t=0.1$}
		\end{minipage}
		\begin{minipage}[t]{0.24\linewidth}
			\centering
			\includegraphics[width=3.1cm]{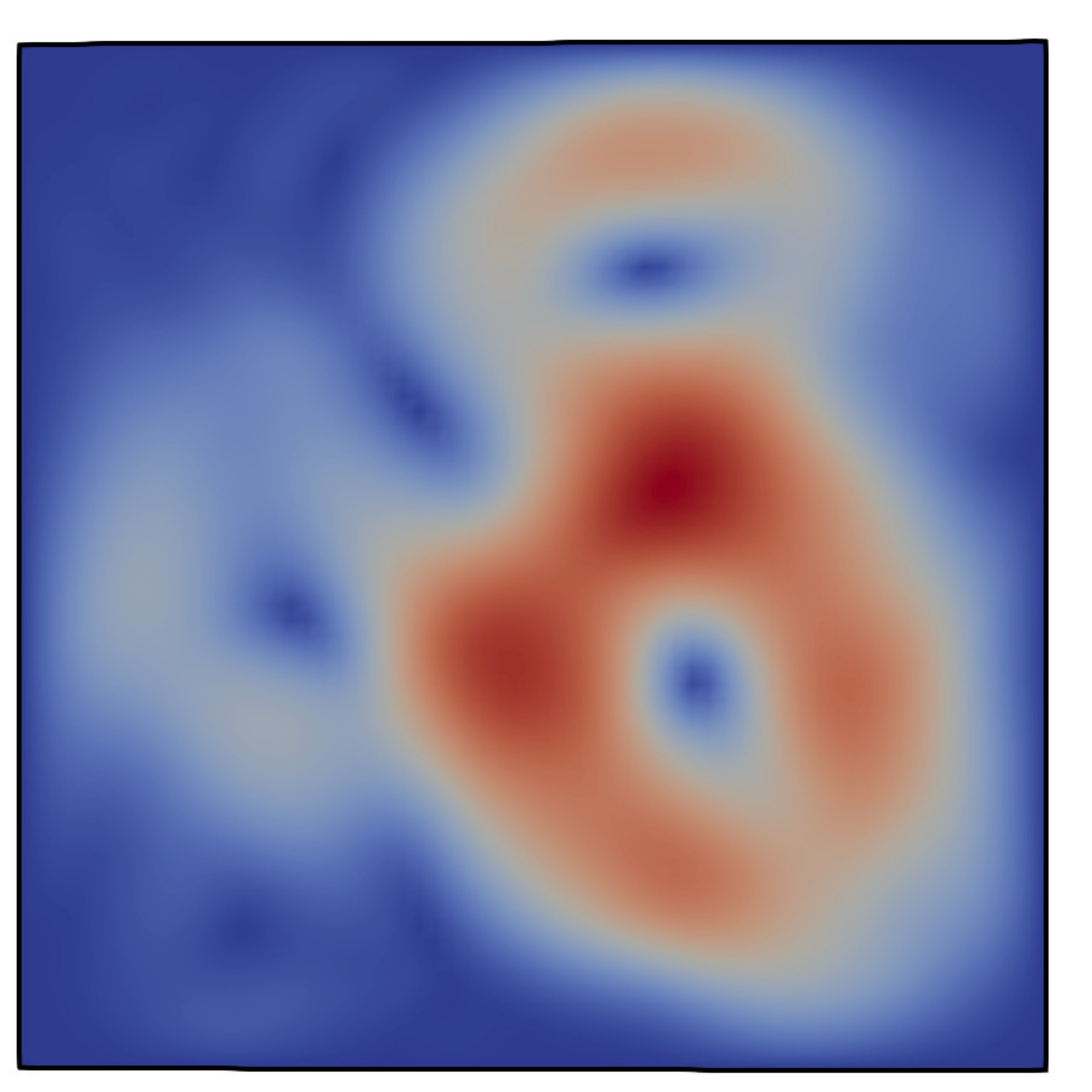} 
			% \\
			% \includegraphics[width=4cm]{Figure_1/velocity_t_0.3.eps}
			\caption*{$t=0.2$}
		\end{minipage} 
		\begin{minipage}[t]{0.24\linewidth}
			\centering
			\includegraphics[width=3.1cm]{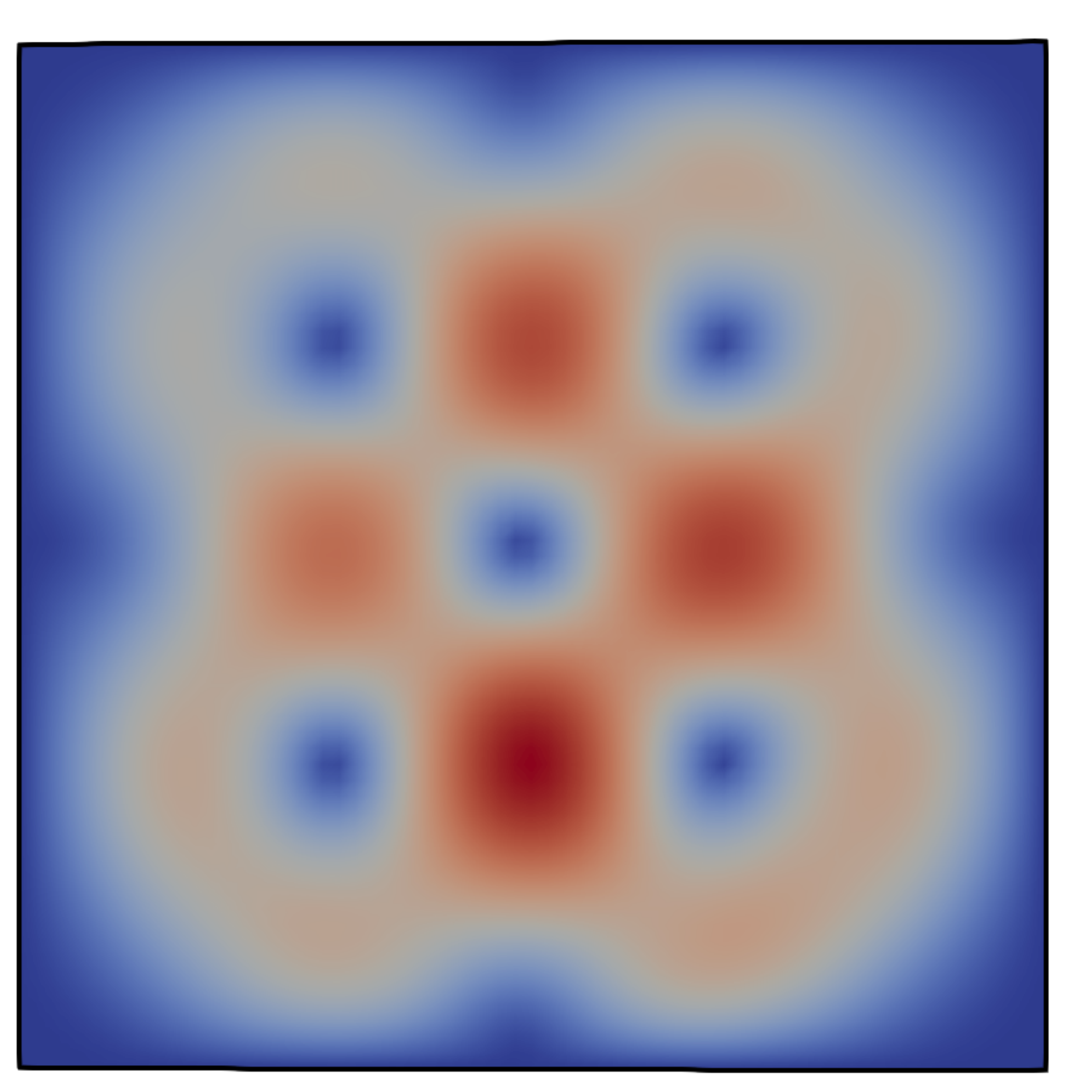} 
			% \\
			% \includegraphics[width=4cm]{Figure_1/velocity_t_0.1.eps}
			\caption*{$t=0.4$}
		\end{minipage}\\
		\vspace{-0.1cm}
		\begin{minipage}[t]{0.24\linewidth}
			\centering
			\includegraphics[width=3.1cm]{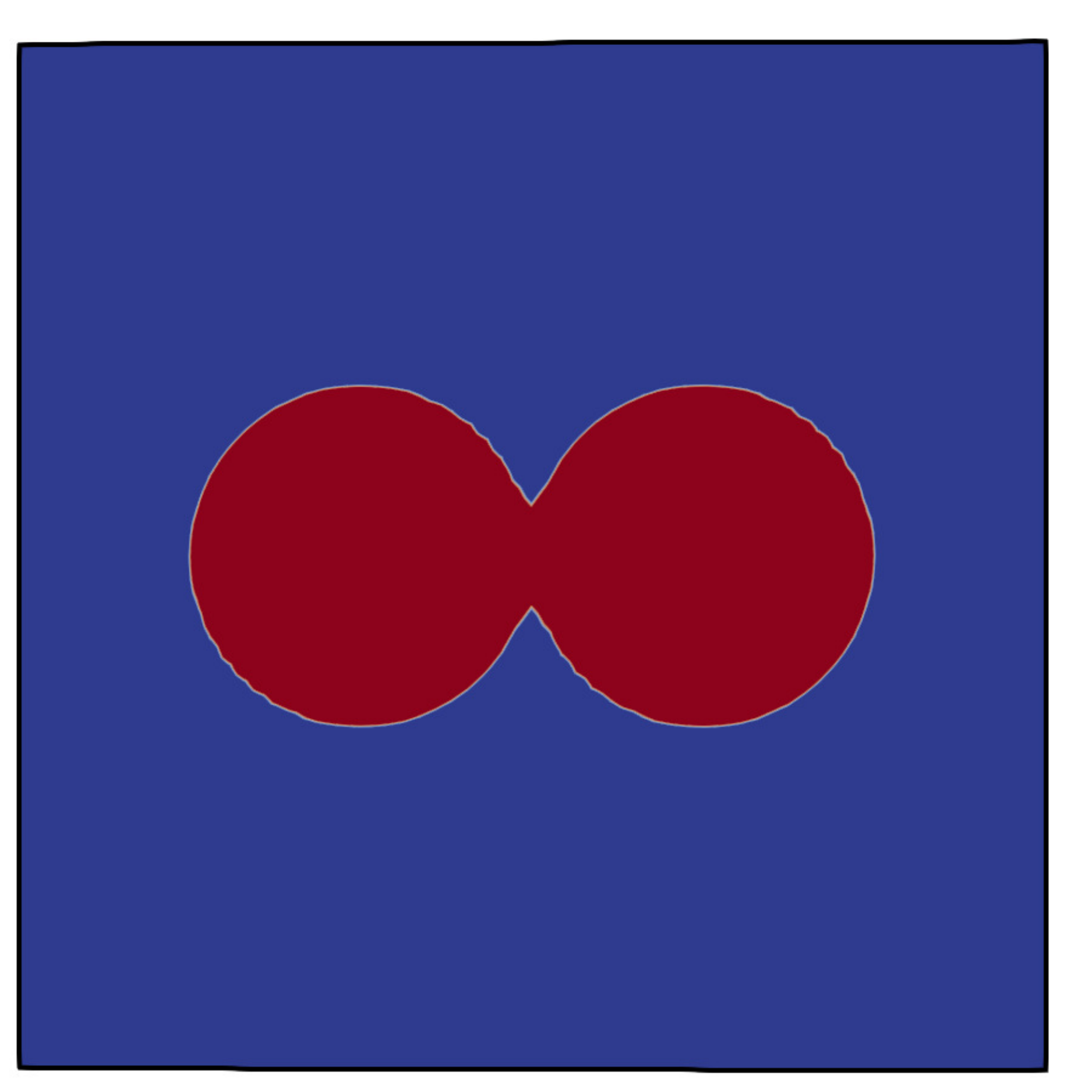} 
			% \\
			% \includegraphics[width=4cm]{Figure_1/velocity_t_0.eps}
			\caption*{$t=0$}
		\end{minipage}
		\begin{minipage}[t]{0.24\linewidth}
			\centering
			\includegraphics[width=3.1cm]{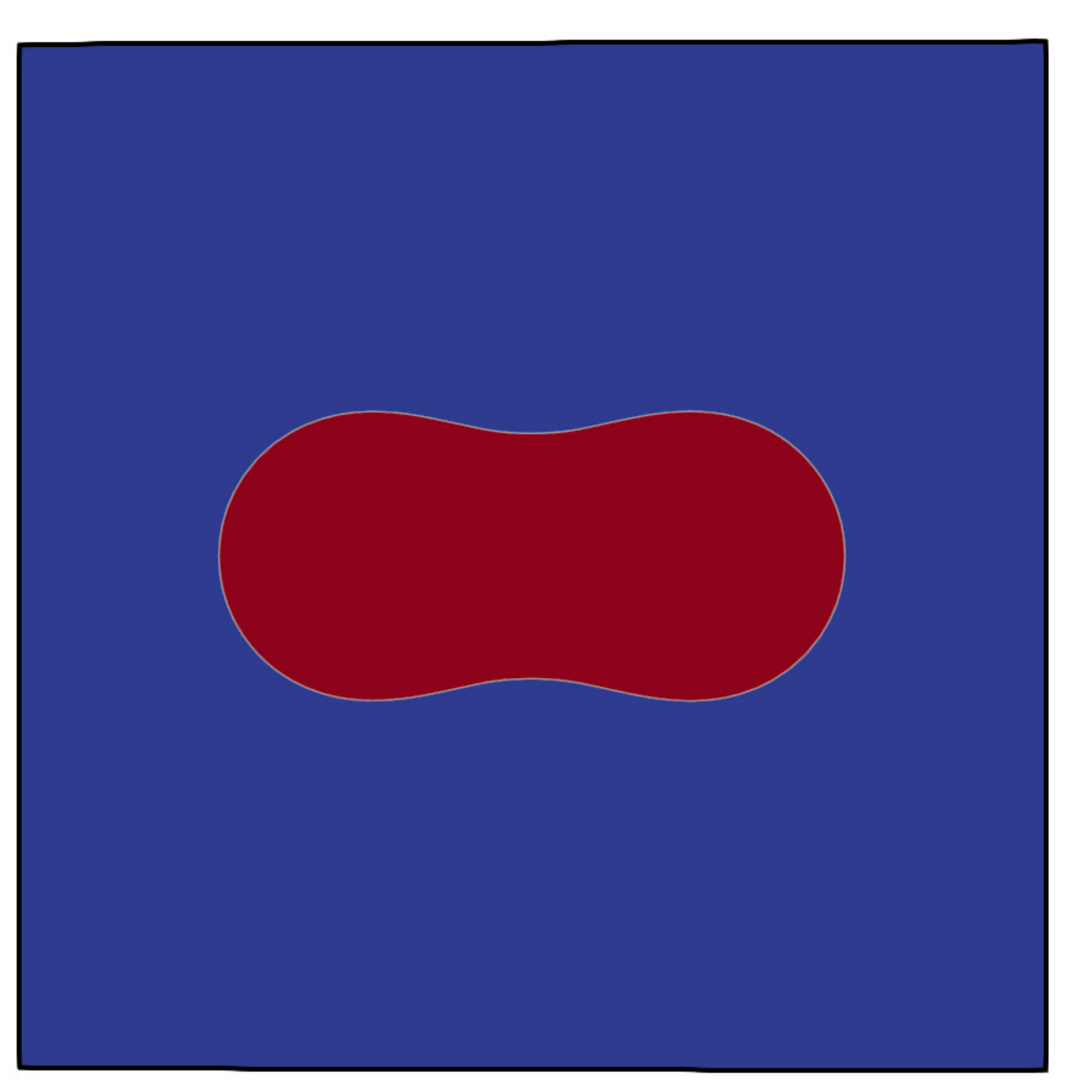} 
			% \\
			% \includegraphics[width=4cm]{Figure_1/velocity_t_0.15.eps}
			\caption*{$t=0.1$}
		\end{minipage}
		\begin{minipage}[t]{0.24\linewidth}
			\centering
			\includegraphics[width=3.1cm]{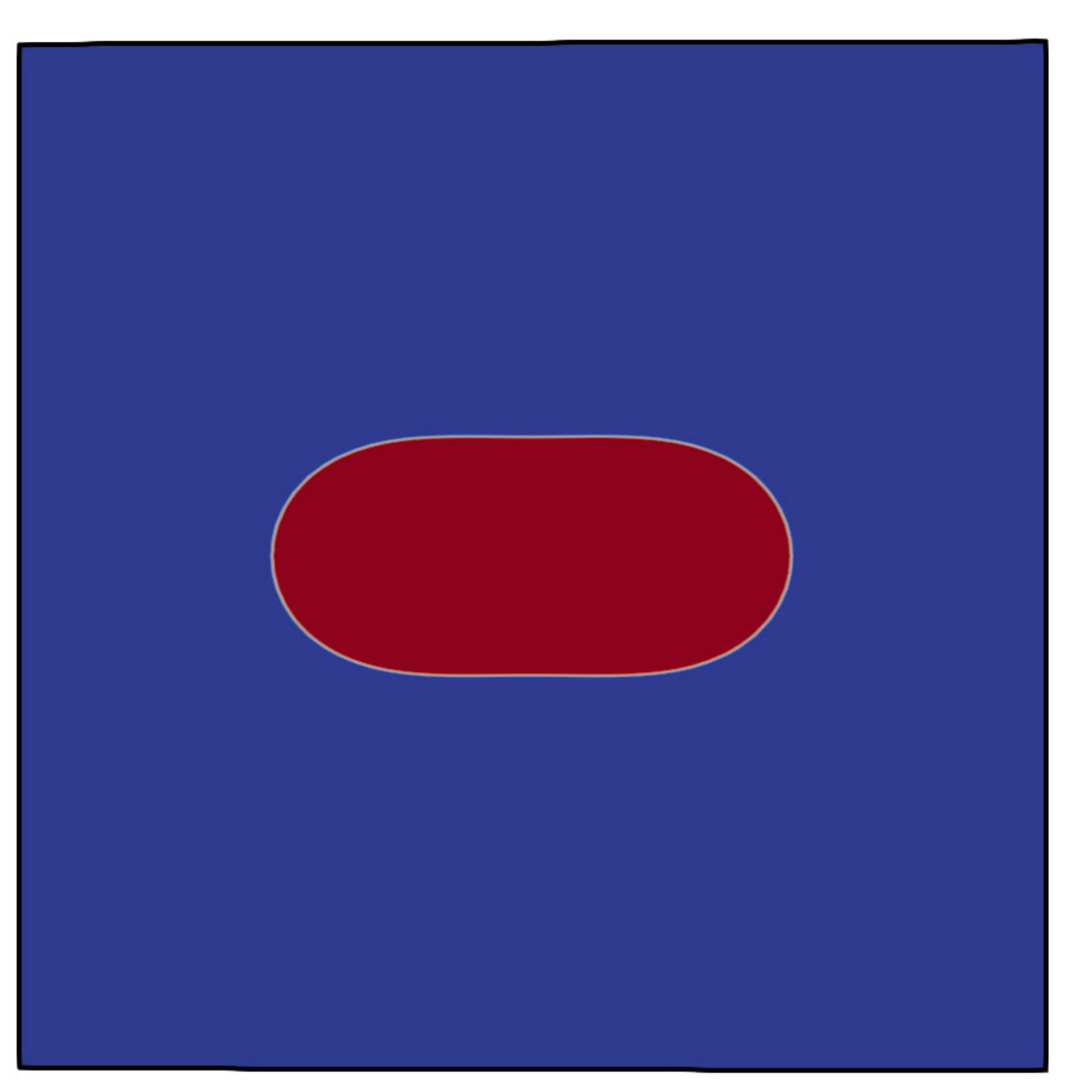} 
			% \\
			% \includegraphics[width=4cm]{Figure_1/velocity_t_0.3.eps}
			\caption*{$t=0.2$}
		\end{minipage} 
		\begin{minipage}[t]{0.24\linewidth}
			\centering
			\includegraphics[width=3.1cm]{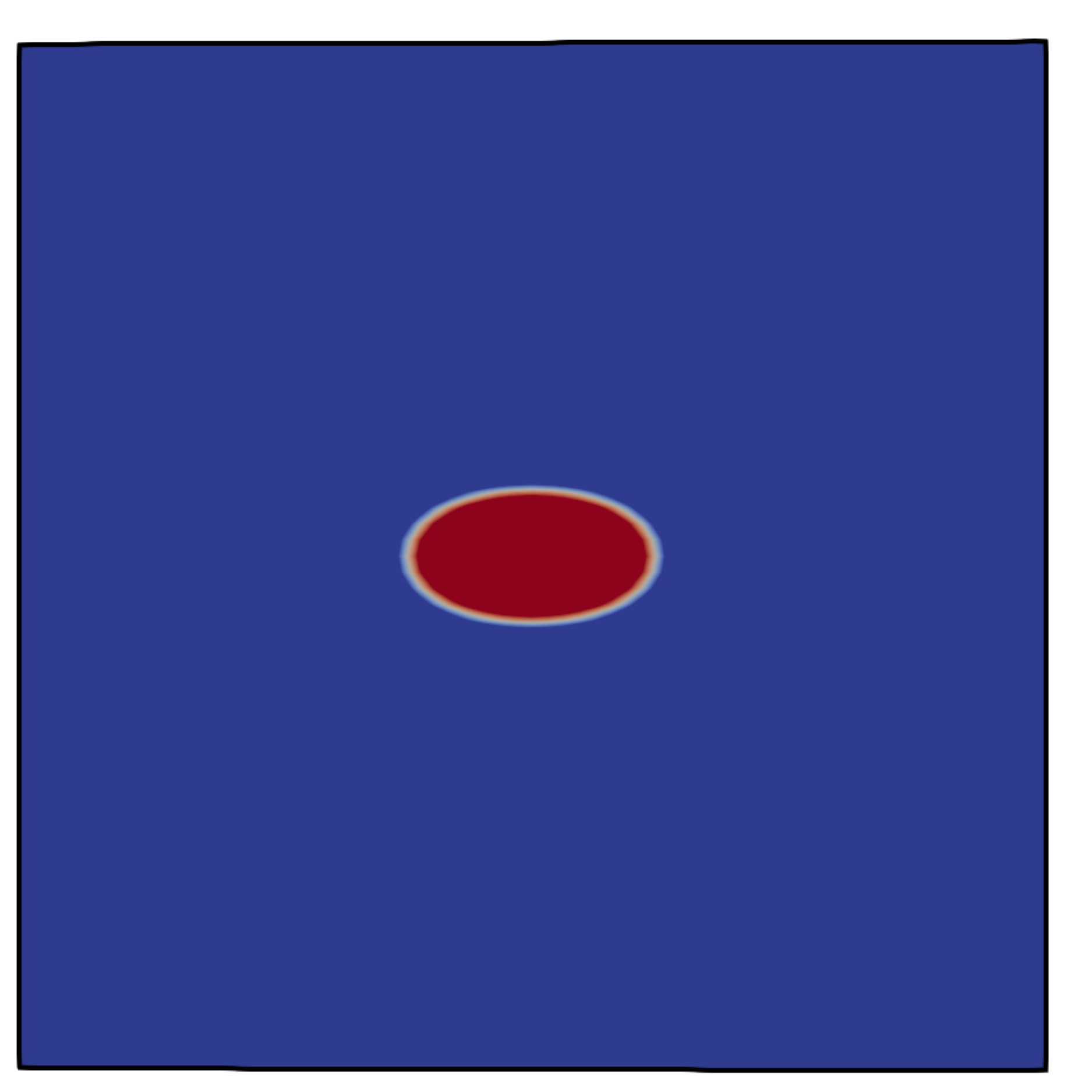} 
			% \\
			% \includegraphics[width=4cm]{Figure_1/velocity_t_0.1.eps}
			\caption*{$t=0.4$}
		\end{minipage}
		\vspace{-0.5cm}
		\caption{Time evolution of phase-Field shape relaxation and fluid self-organization in active fluid: the vector field of velocity (top row), velocity magnitude (middle row), and phase field (bottom row) at selected time instances.}
		\label{fig:Phase-field shape relaxation and self-organization in active fluid}
	\end{figure}

	\subsection{Phase-field dynamics of circular bubble shrinking and fluid self-organization in active fluid}
	\label{Phase-field dynamics of circular bubble shrinking and fluid self-organization in active fluid}
	In this subsection, we simulate the shrinking process of a circular bubble within the context of the Allen-Cahn active fluid system by \Cref{fully discrete formulations scheme}. 
	The computational domain is \( D = [0, 2\pi]^2 \), and the initial condition for the phase field is
	\[
	\phi_0(x,y) = 1 + \sum_{i=1}^{2} \tanh \Big( \frac{r_i - \sqrt{(x - x_i)^2 + (y - y_i)^2}}{0.06} \Big),
	\]
	with \( r_1 = 1.4 \), \( r_2 = 0.5 \), \( x_1 = \pi - 0.8 \), \( x_2 = \pi + 1.7 \), and \( y_1 = y_2 = \pi \). 
	The initial velocity field \( u_0(x, y) \) is randomly initialized as
	\[
	u_0(x, y) = \left( \text{rand}_2(x, y), \text{rand}_2(x, y) \right), \qquad (x, y) \in D,
	\]
	where $\text{rand}_2(x,y)$ is the same random random variable as in Subsetion \ref{Phase-field spinodal decomposition and fluid self-organization in active fluid} and \ref{Phase-field shape relaxation and fluid self-organization in active fluid}.
	The boundary conditions are: $u|_{\partial D} = \Delta u|_{\partial D} = 0$ and $\partial_n \phi = \partial_n m = 0$ on $\partial D$. 
	The parameters in this experiment are set to be $\mu = \nu = \beta = \alpha = \gamma = 1$, $\sigma = 10$, and $\kappa = 0.01$. 
	The simulation is carried out over the time interval $[0, 1.2]$
	with the constant time step size $\Delta t = 0.01$ and uniform mesh size $h = \frac{1}{64}$.

	Figure~\ref{fig:Phase-field dynamics of circular bubble shrinking and fluid self-organization in active fluid} shows the evolution of the velocity vector field (top row), velocity magnitude (middle row), and phase field (bottom row) at selected time instances. 
	At \( t = 0 \), two distinct phase-field bubbles are present, 
	while the velocity field exhibits random perturbations. 
	As time progresses to \( t = 0.5 \), the smaller bubble is gradually absorbed by the larger one due to the influence of roughening. 
	Small vortical structures begin to emerge in the velocity field, with the flow becoming more organized. 
	By \( t = 0.8 \), the smaller bubble's volume has significantly decreased, and the velocity field has formed a more coherent pattern. At the final time \( t = 1.2 \), the smaller circle has been completely absorbed into the larger circle; meanwhile, the velocity field has organized into four distinct vortices, with each rotating around a central point.
	The results demonstrate the self-organization of the active fluid system, precisely capturing both the phase-field shrinking dynamics and the velocity vortices through \Cref{fully discrete formulations scheme}.

	\begin{figure}[htbp] 
		\centering
		\begin{minipage}[t]{0.24\linewidth}
			\centering
			\includegraphics[width=3.2cm]{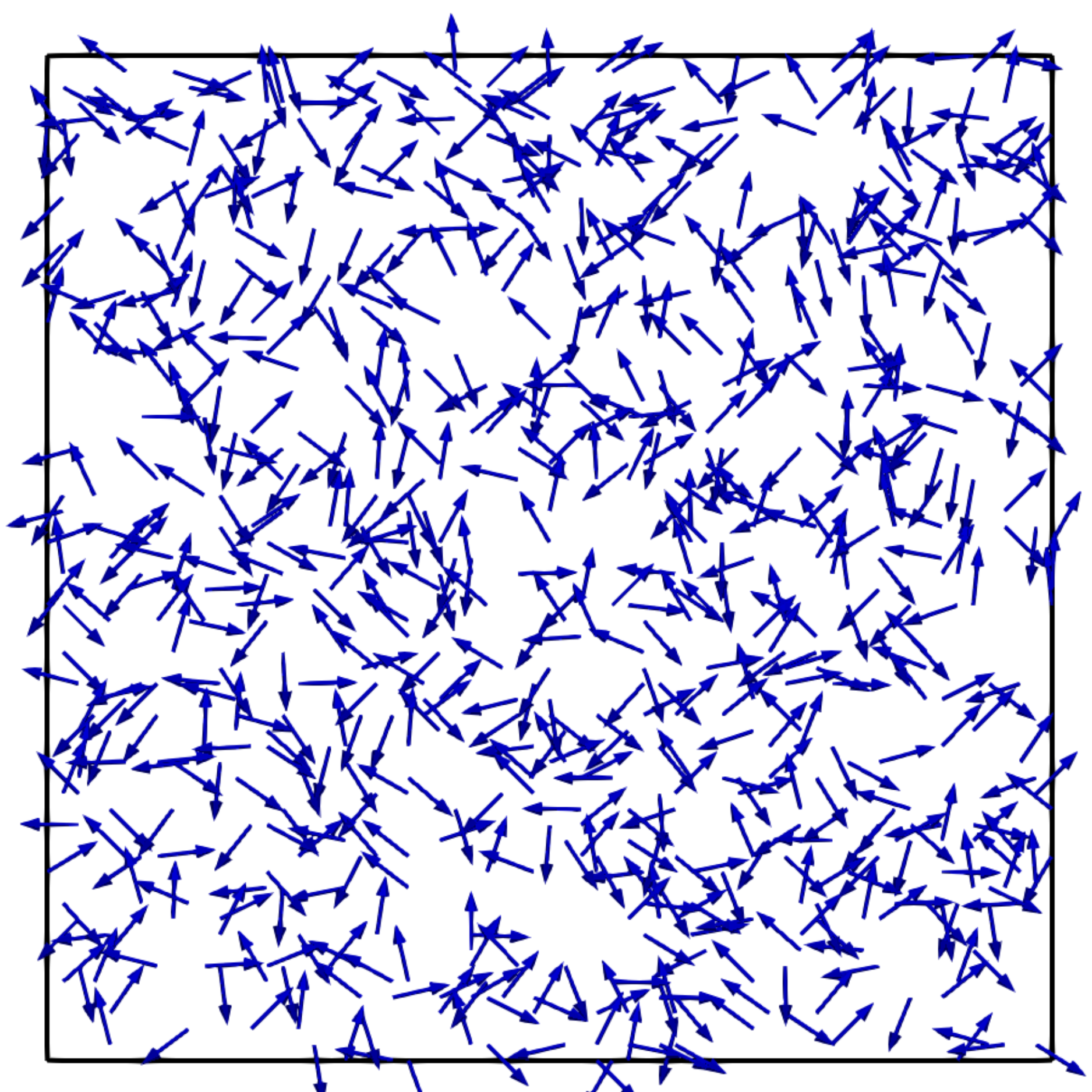} 
			% \\
			% \includegraphics[width=4cm]{Figure_1/velocity_t_0.eps}
			\caption*{$t=0$}
		\end{minipage}
		\begin{minipage}[t]{0.24\linewidth}
			\centering
			\includegraphics[width=3.2cm]{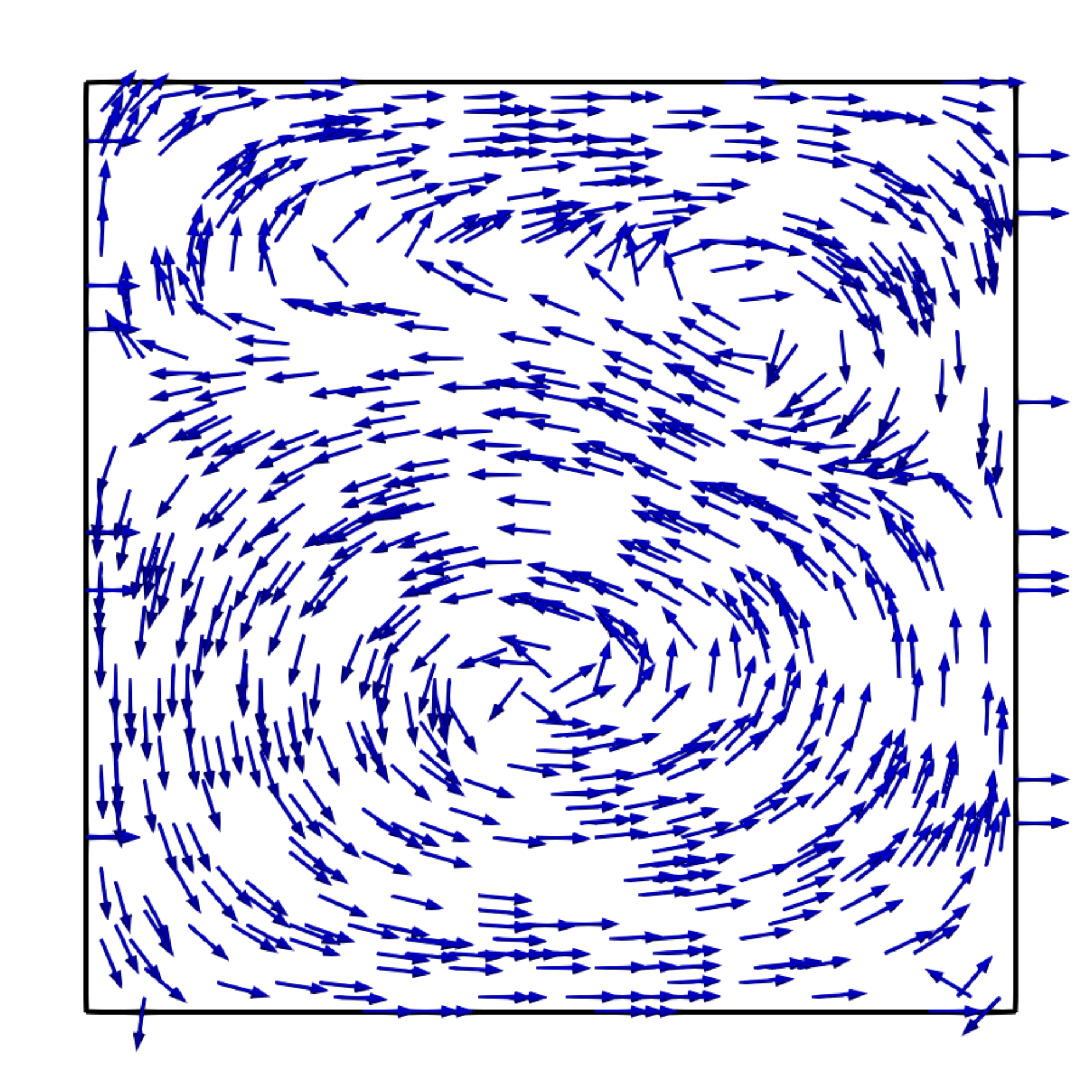} 
			% \\
			% \includegraphics[width=4cm]{Figure_1/velocity_t_0.15.eps}
			\caption*{$t=0.5$}
		\end{minipage}
		\begin{minipage}[t]{0.24\linewidth}
			\centering
			\includegraphics[width=3.2cm]{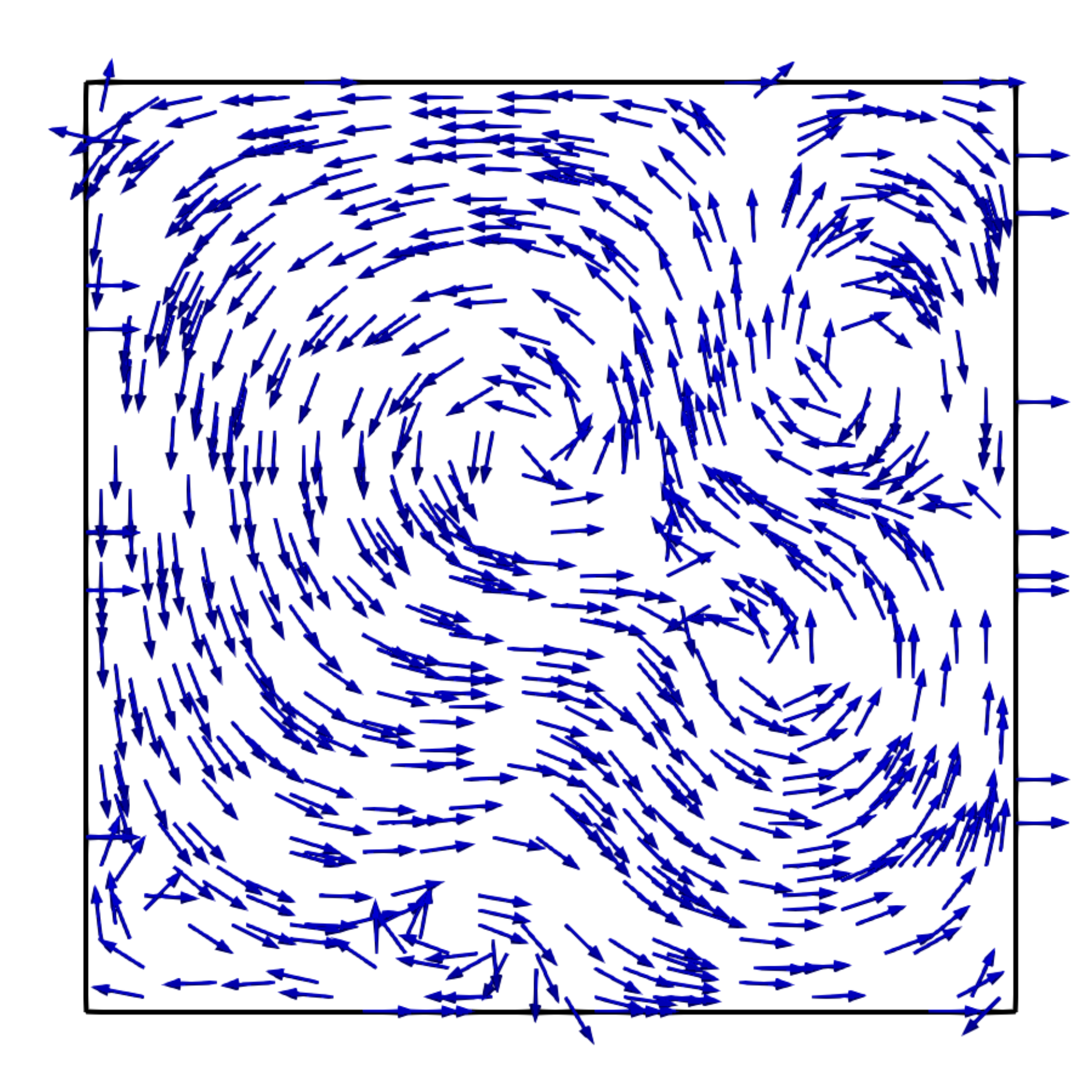} 
			% \\
			% \includegraphics[width=4cm]{Figure_1/velocity_t_0.3.eps}
			\caption*{$t=0.8$}
		\end{minipage} 
		\begin{minipage}[t]{0.24\linewidth}
			\centering
			\includegraphics[width=3.2cm]{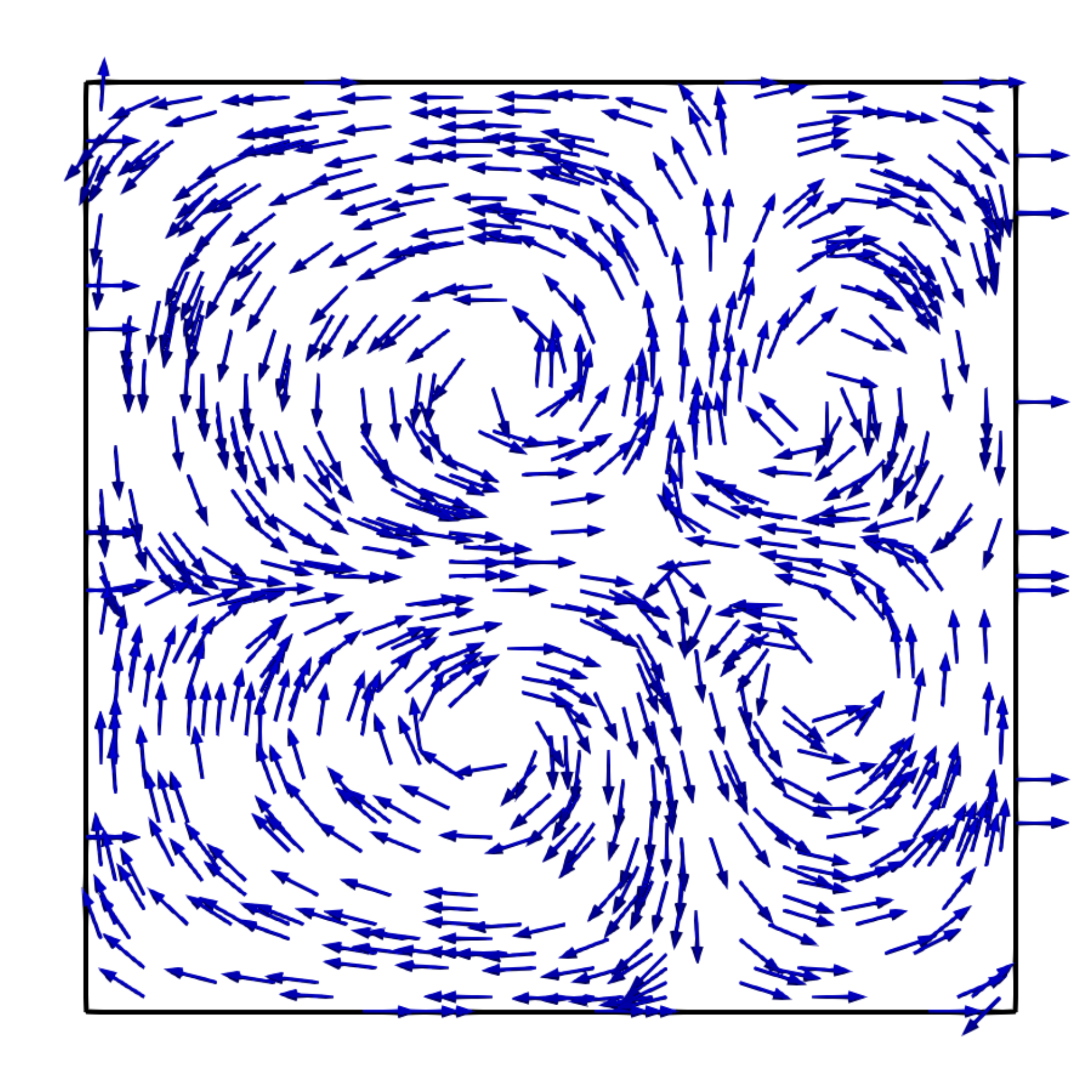} 
			% \\
			% \includegraphics[width=4cm]{Figure_1/velocity_t_0.3.eps}
			\caption*{$t=1.2$}
		\end{minipage} \\
		\vspace{-0.1cm}
		\begin{minipage}[t]{0.24\linewidth}
			\centering
			\includegraphics[width=3.1cm]{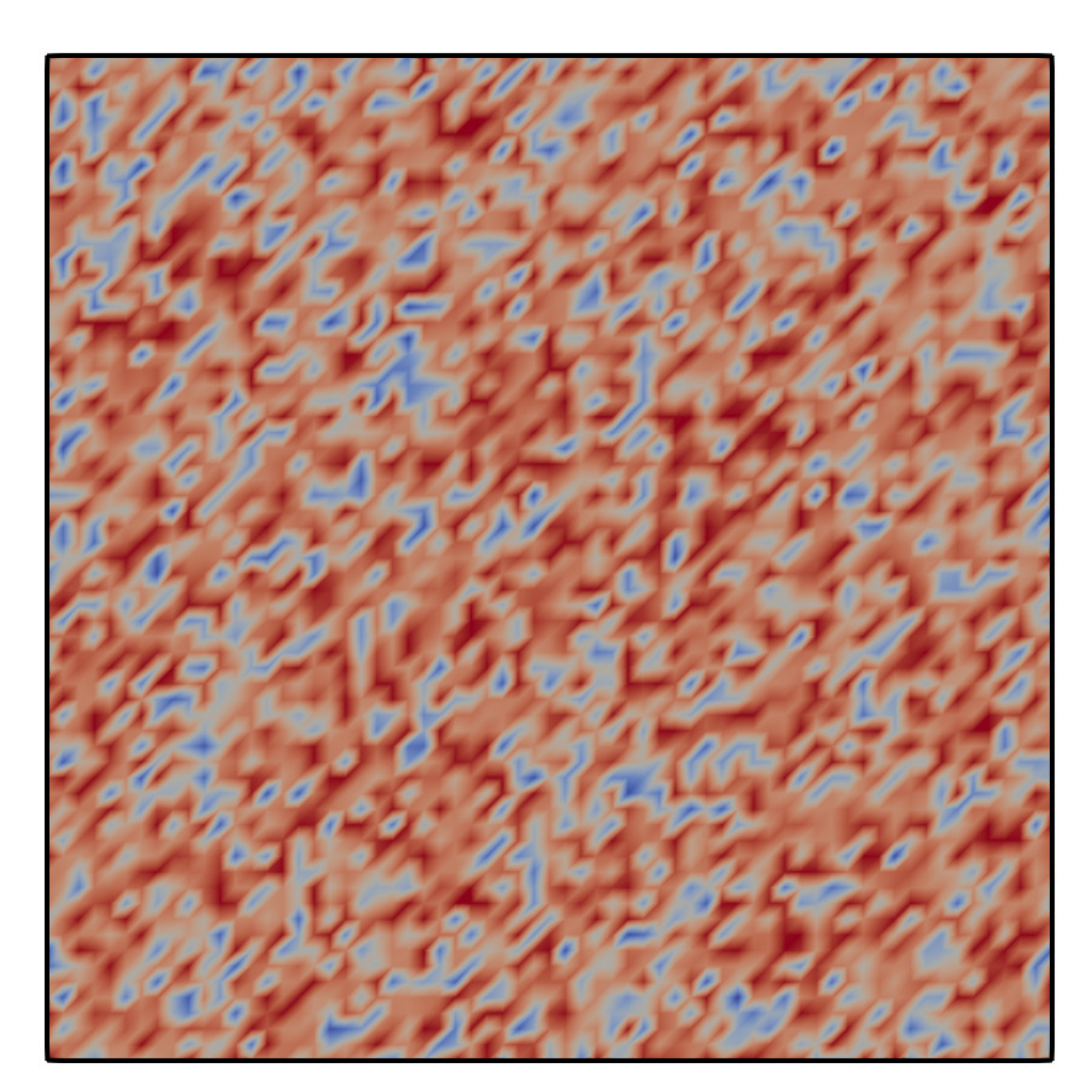} 
			% \\
			% \includegraphics[width=4cm]{Figure_1/velocity_t_0.eps}
			\caption*{$t=0$}
		\end{minipage}
		\begin{minipage}[t]{0.24\linewidth}
			\centering
			\includegraphics[width=3.1cm]{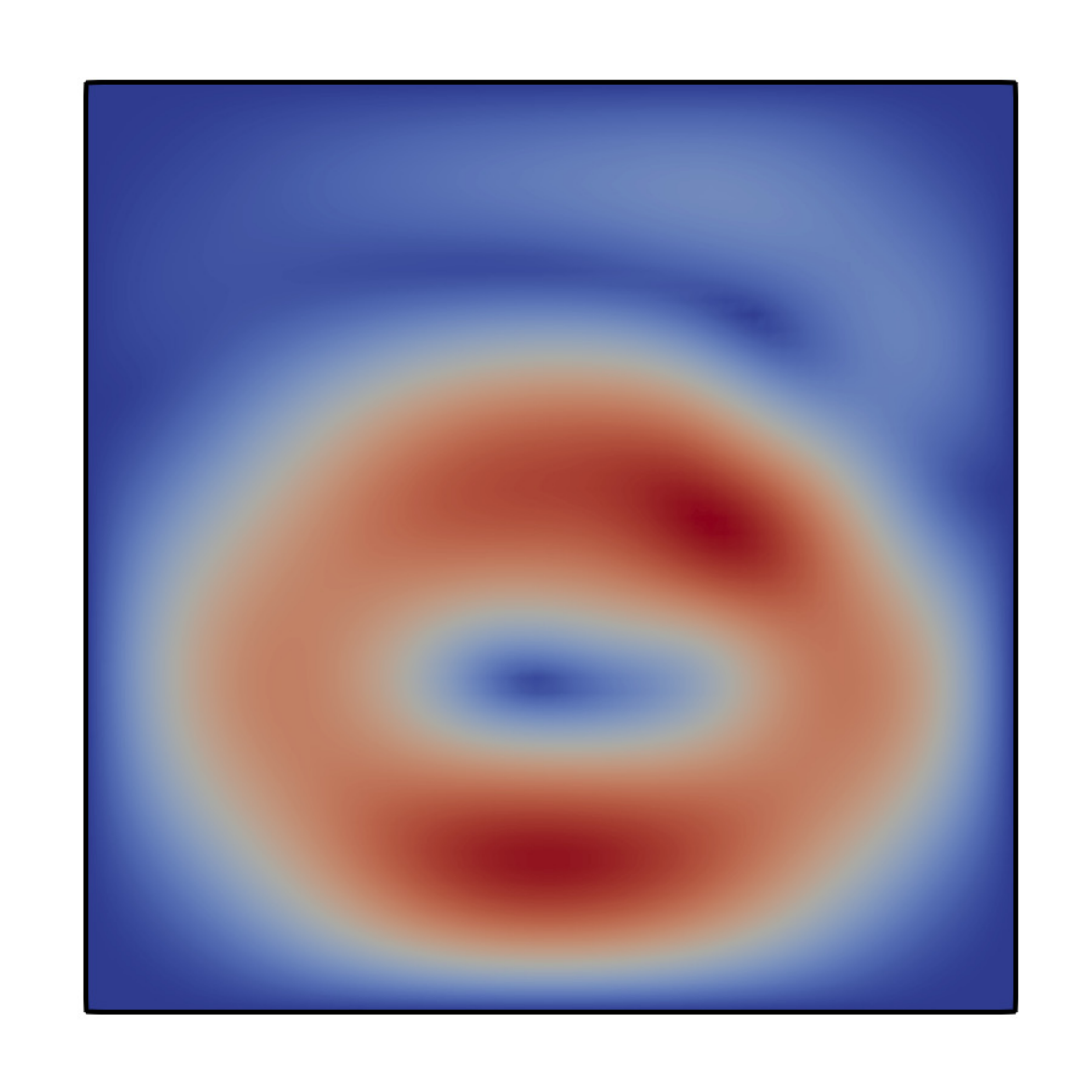} 
			% \\
			% \includegraphics[width=4cm]{Figure_1/velocity_t_0.15.eps}
			\caption*{$t=0.5$}
		\end{minipage}
		\begin{minipage}[t]{0.24\linewidth}
			\centering
			\includegraphics[width=3.1cm]{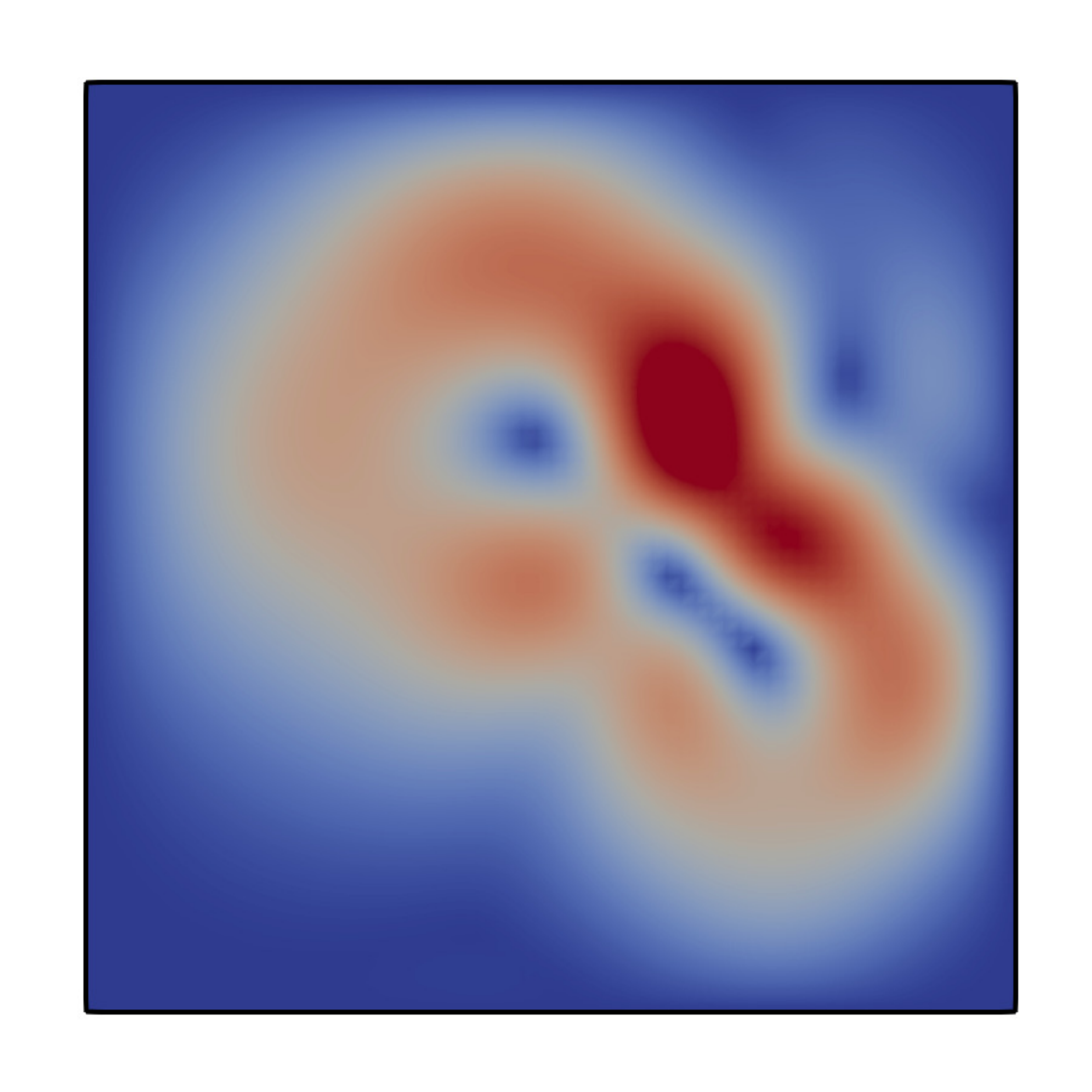} 
			% \\
			% \includegraphics[width=4cm]{Figure_1/velocity_t_0.3.eps}
			\caption*{$t=0.8$}
		\end{minipage} 
		\begin{minipage}[t]{0.24\linewidth}
			\centering
			\includegraphics[width=3.1cm]{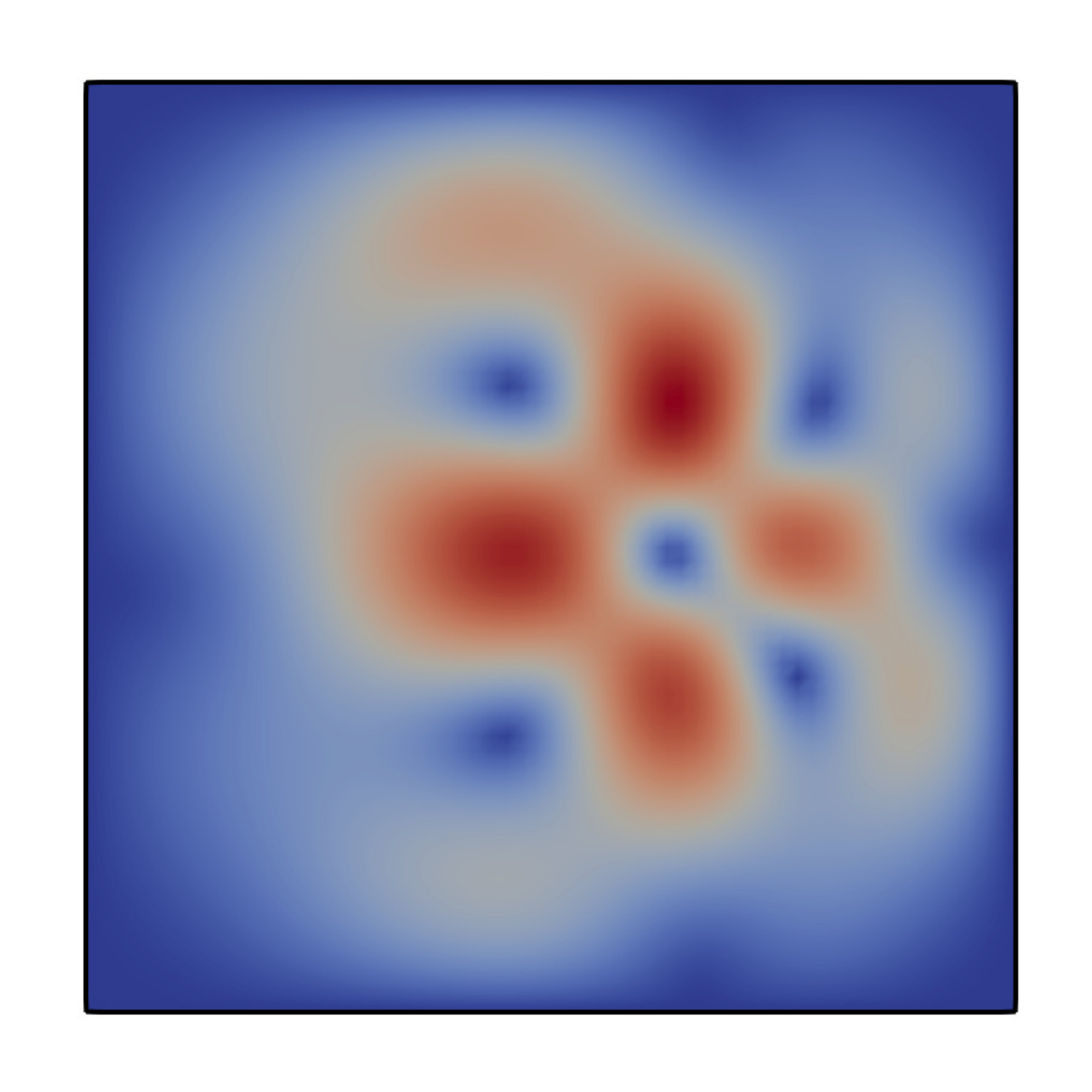} 
			% \\
			% \includegraphics[width=4cm]{Figure_1/velocity_t_0.1.eps}
			\caption*{$t=1.2$}
		\end{minipage}\\
		\vspace{-0.1cm}
		\begin{minipage}[t]{0.24\linewidth}
			\centering
			\includegraphics[width=3.1cm]{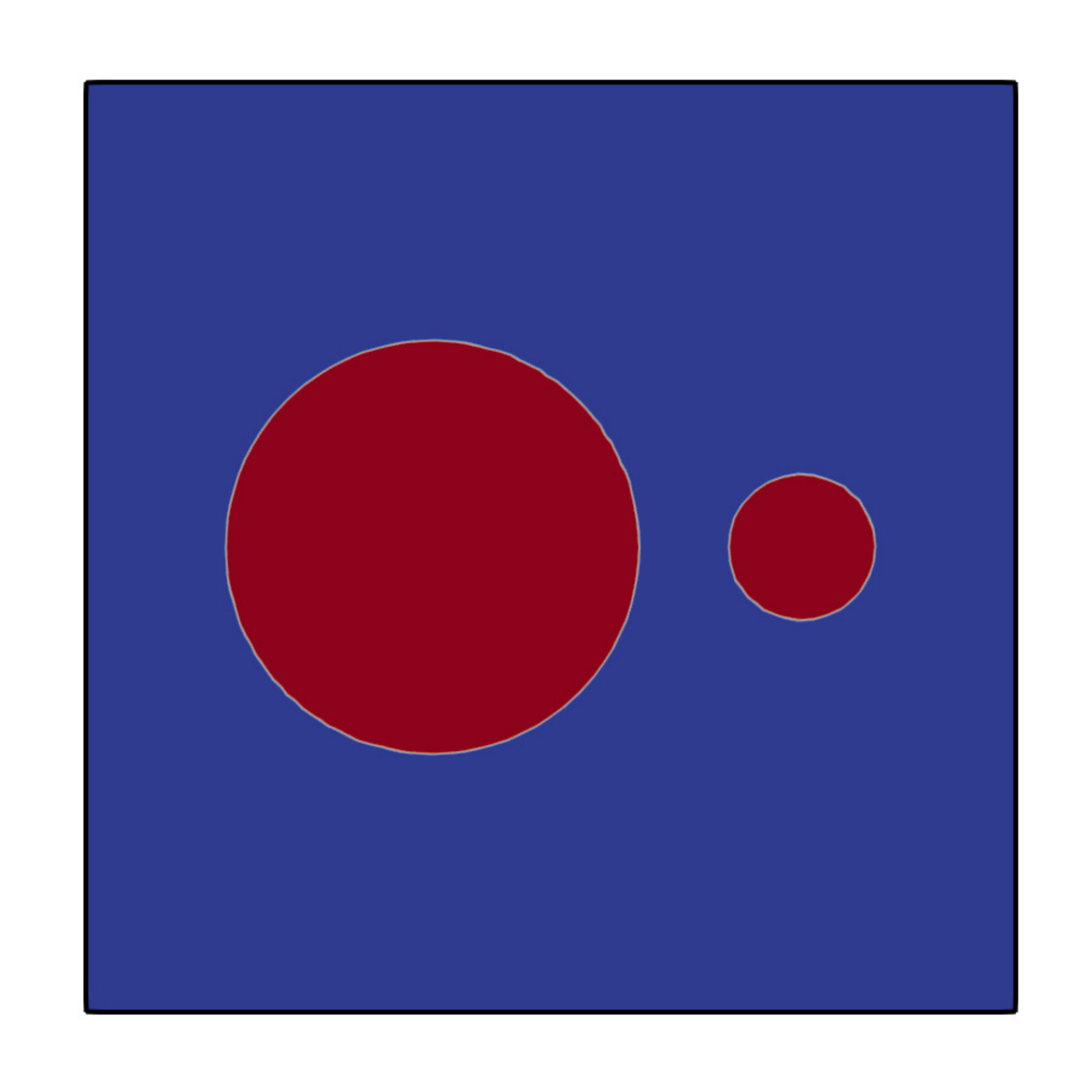} 
			% \\
			% \includegraphics[width=4cm]{Figure_1/velocity_t_0.eps}
			\caption*{$t=0$}
		\end{minipage}
		\begin{minipage}[t]{0.24\linewidth}
			\centering
			\includegraphics[width=3.1cm]{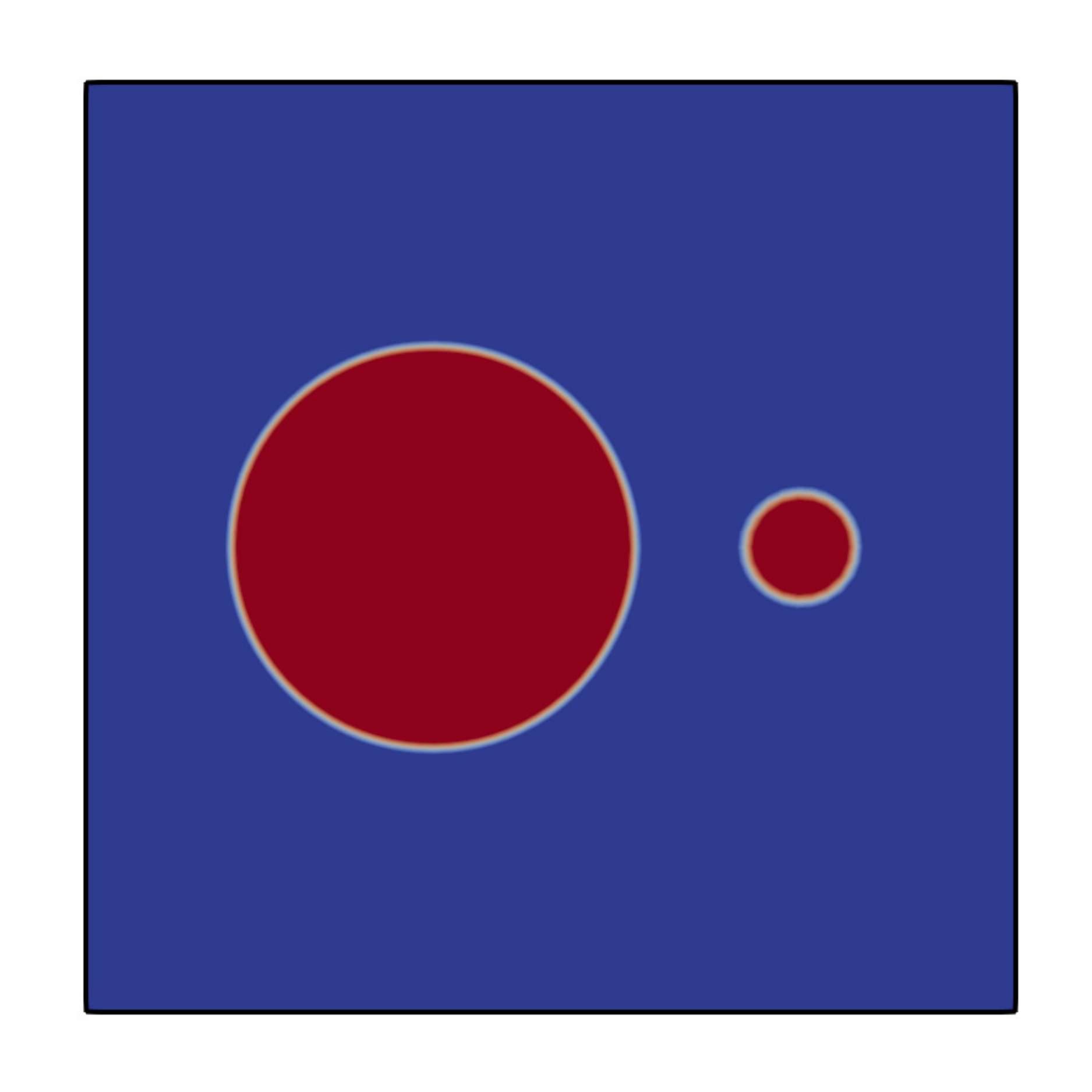} 
			% \\
			% \includegraphics[width=4cm]{Figure_1/velocity_t_0.15.eps}
			\caption*{$t=0.5$}
		\end{minipage}
		\begin{minipage}[t]{0.24\linewidth}
			\centering
			\includegraphics[width=3.1cm]{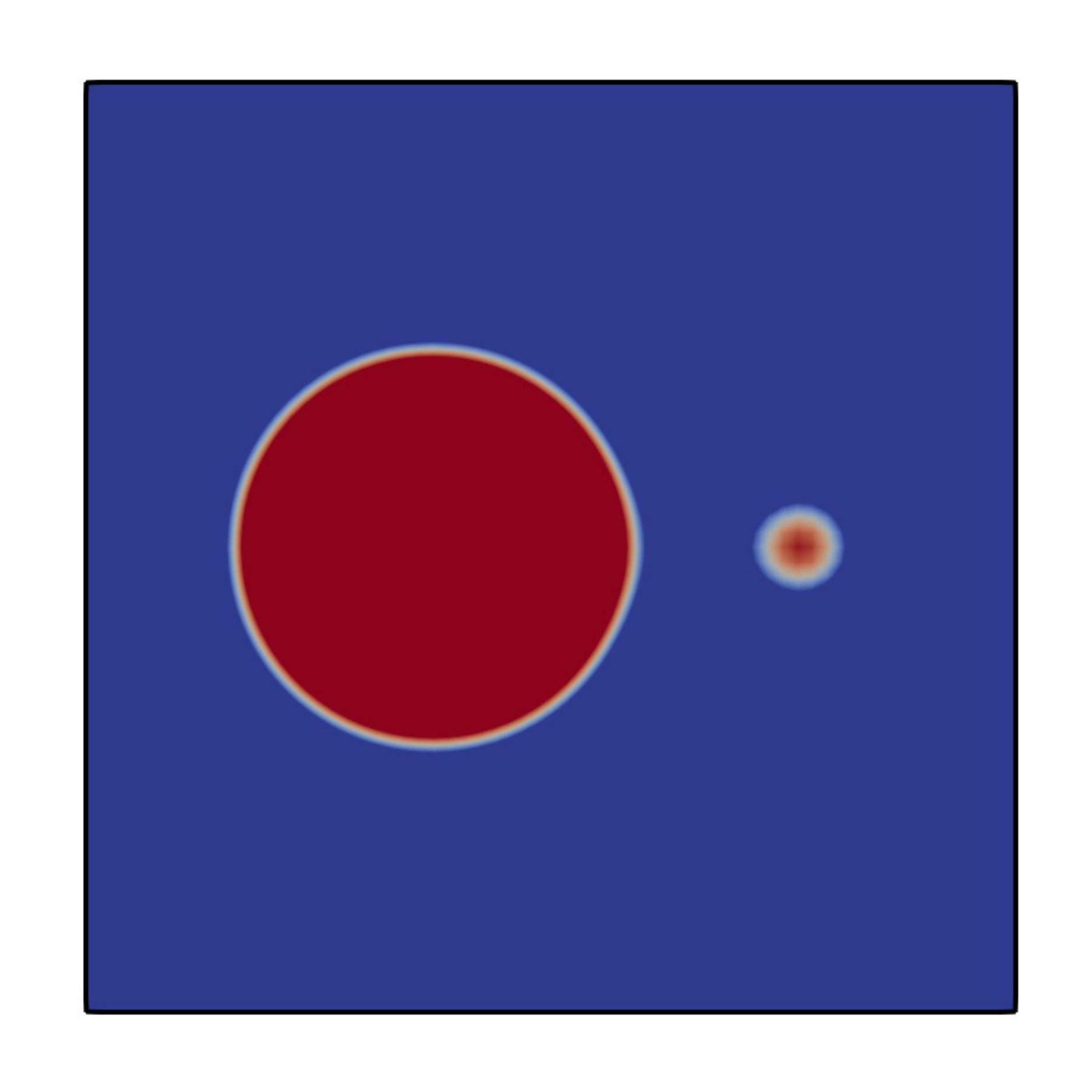} 
			% \\
			% \includegraphics[width=4cm]{Figure_1/velocity_t_0.3.eps}
			\caption*{$t=0.8$}
		\end{minipage} 
		\begin{minipage}[t]{0.24\linewidth}
			\centering
			\includegraphics[width=3.1cm]{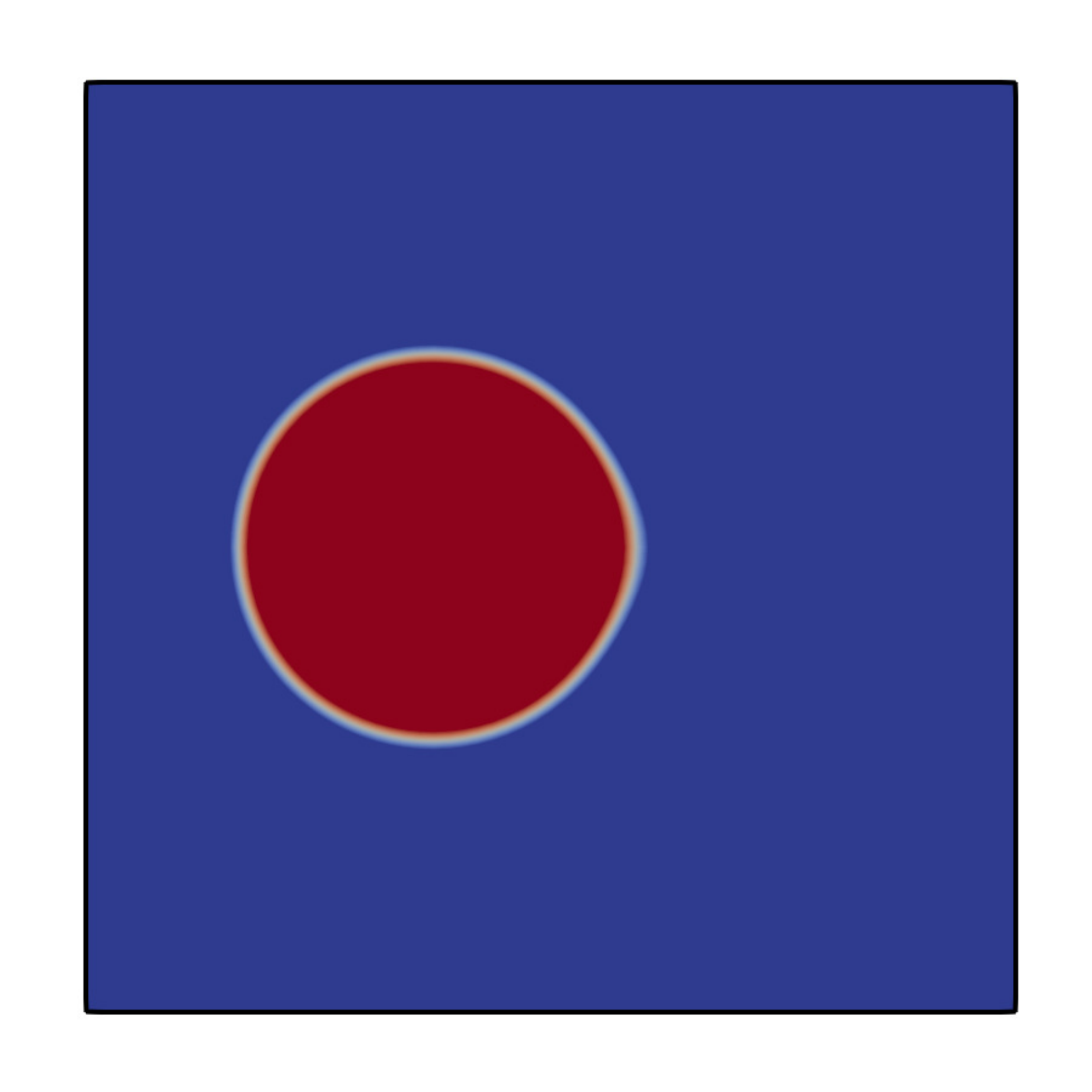} 
			% \\
			% \includegraphics[width=4cm]{Figure_1/velocity_t_0.1.eps}
			\caption*{$t=1.2$}
		\end{minipage}
		\vspace{-0.5cm}
		\caption{Time evolution of phase-field dynamics in circular bubble shrinking and fluid self-organization in active fluids: the vector field of velocity (top row), velocity magnitude (middle row), and phase field (bottom row) at selected time instances}
		\label{fig:Phase-field dynamics of circular bubble shrinking and fluid self-organization in active fluid}
	\end{figure}

	\subsection{Three-dimensional phase-field motion by mean curvature}
	\label{Three-dimensional phase-field motion by mean curvature}
	In this subsection, we apply \Cref{fully discrete formulations scheme} with uniform time grids to simulate the motion by mean curvature in a 3D domain \(\Omega = [0, 1]^3 \) within the context of the Allen-Cahn active fluid system. 
	The initial profile for the phase field is defined as a zero-level set function representing a sphere with radius 0.35 centered at \( (0.5, 0.5, 0.5) \), i.e., 
	\[
	\phi(x, y, z) = \tanh \Big( \frac{0.35 - \sqrt{(x - 0.5)^2 + (y - 0.5)^2 + (z - 0.5)^2}}{0.04\sqrt{2}} \Big). 
	\]
	The boundary conditions are: $u|_{\partial D} = \Delta u|_{\partial D} = 0$ and $\partial_n \phi = \partial_n m = 0$ on $\partial D$. 
	The parameters in this experiment are set to be $\mu = \nu = \beta = \alpha = \gamma = 1$, $\sigma = 10$, $\kappa = 0.01$. 
	The simulation is carried out over the time interval $[0, 0.2]$ with the constant time step size $\Delta t = 0.01$ and the uniform mesh size $h = \frac{1}{64}$.

	Figure~\ref{fig:Motion by Mean Curvature} demonstrates that \Cref{fully discrete formulations scheme} with uniform time grids correctly simulates the zero-level isosurfaces of the Allen-Cahn active fluid at given time instances.

	\begin{figure}[htbp]
		\centering
		\begin{minipage}[t]{0.32\linewidth}
			\centering
			\includegraphics[width=3.2cm]{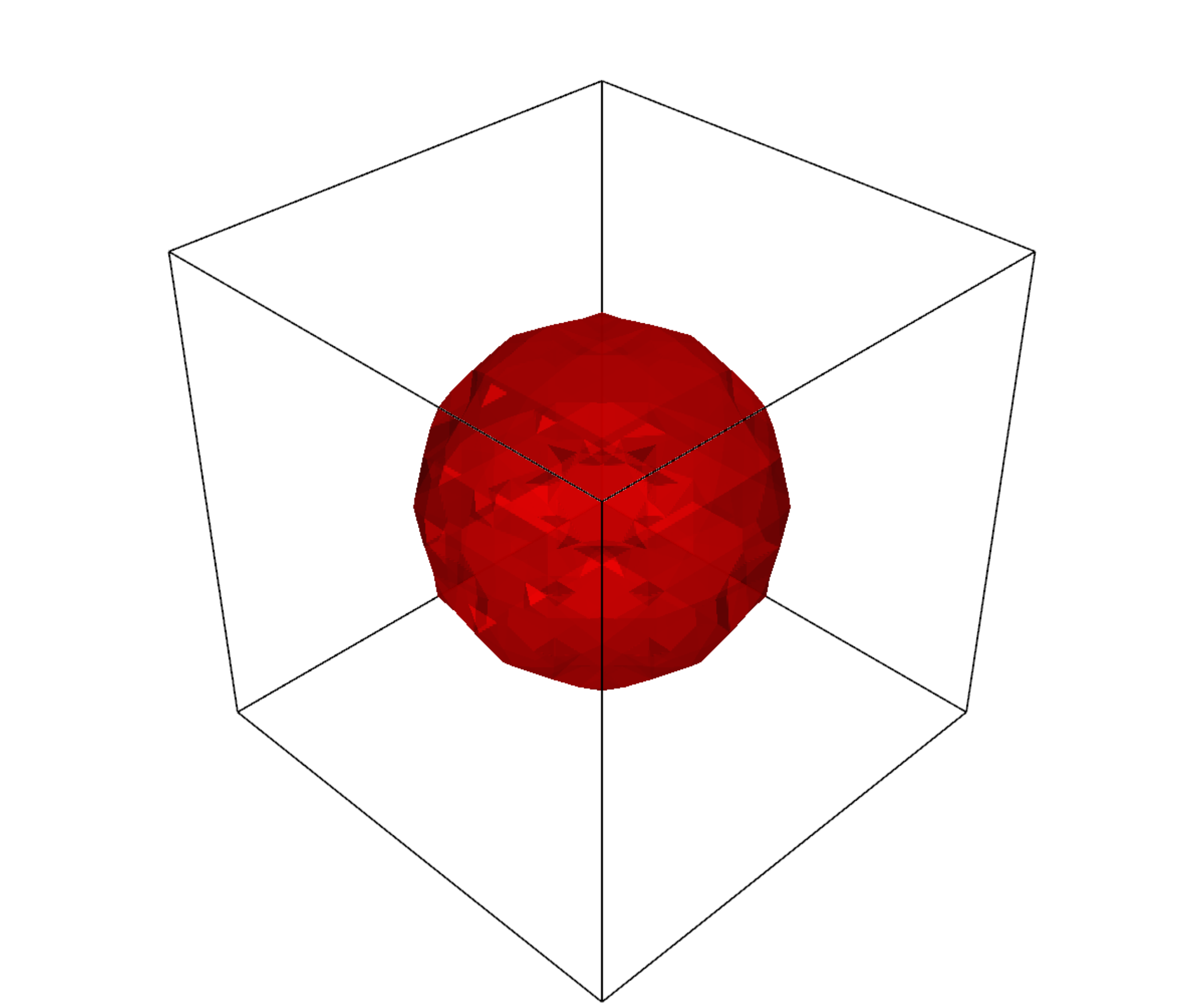}
			\caption*{$t=0$}
		\end{minipage}
		\begin{minipage}[t]{0.32\linewidth}
			\centering
			\includegraphics[width=3.2cm]{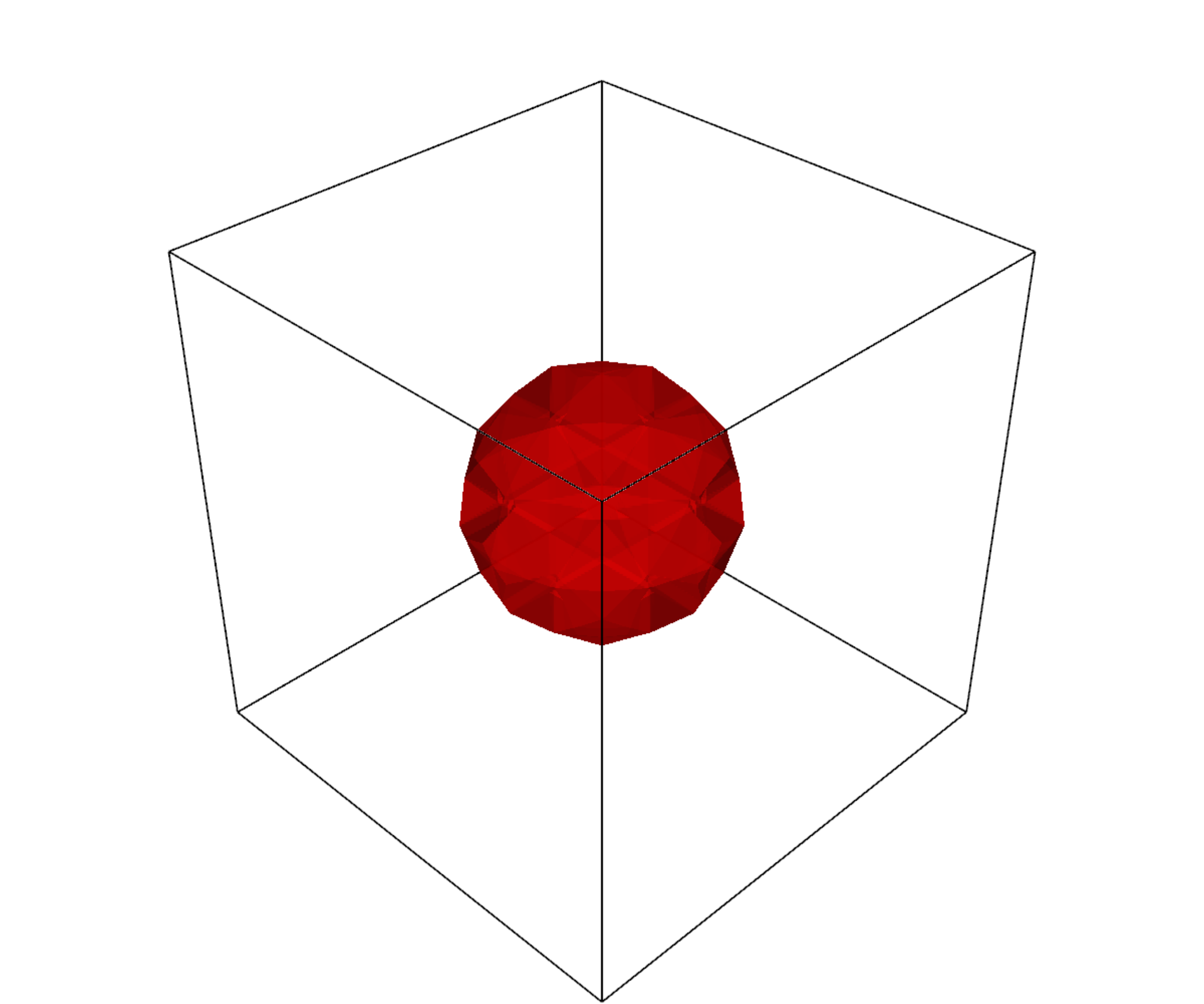}
			\caption*{$t = 0.1$}
		\end{minipage}
		\begin{minipage}[t]{0.32\linewidth}
			\centering
			\includegraphics[width=3.2cm]{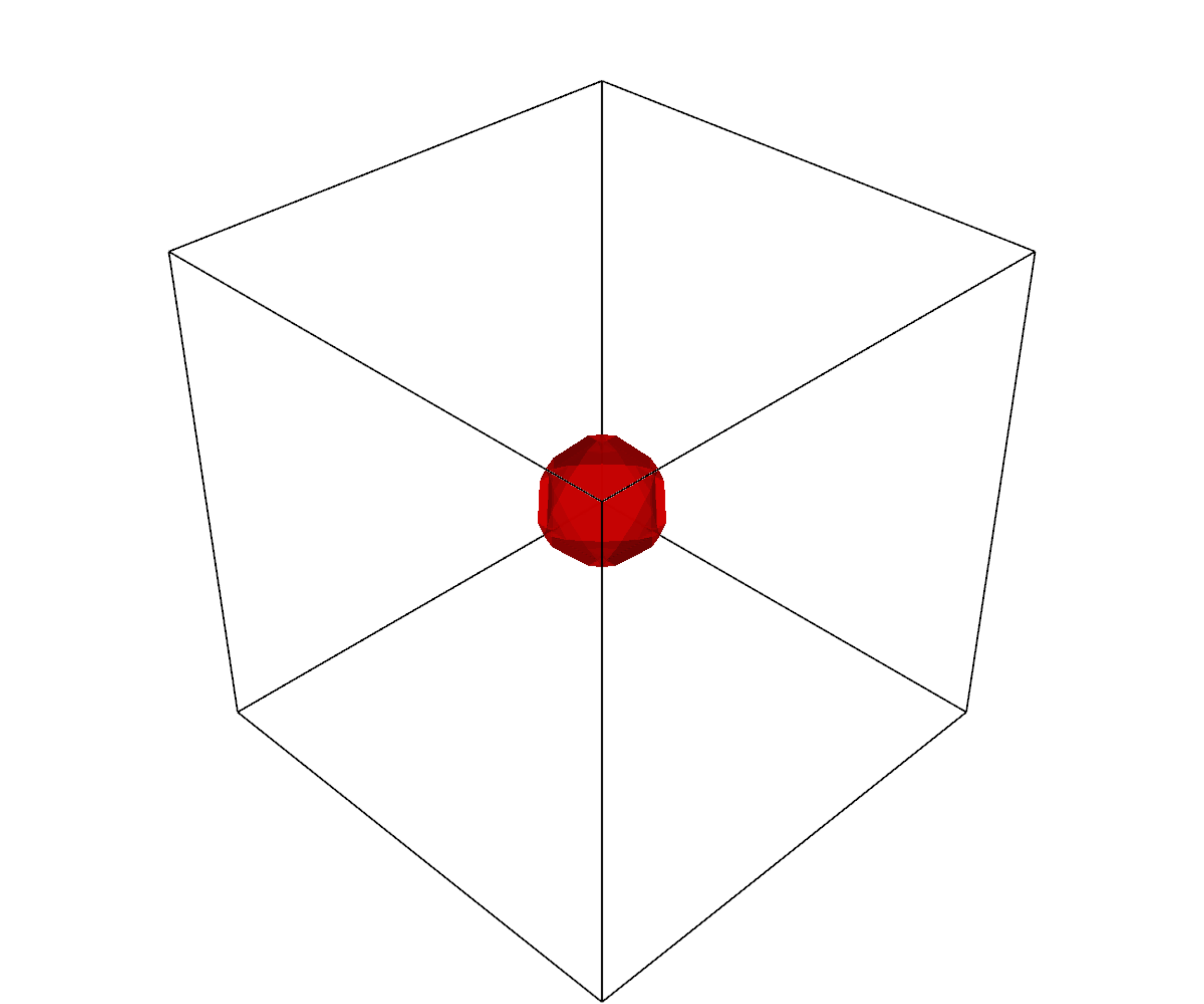}
			\caption*{$t = 0.2$}
		\end{minipage}
		\vspace{-0.1cm}
		\caption{Zero-level isosurfaces of the computational solutions at selected time instances.}
		\label{fig:Motion by Mean Curvature}
	\end{figure}

	Then, we revise the initial condition of the phase field function to the dumbbell-shaped initial condition to observe the dynamics of the numerical solution under motion by mean curvature for various initial conditions.  Now the initial phase field function becomes
	\[
	\phi(x,y,z) =
	\begin{cases}
		\tanh \bigg(\dfrac{R_{1} - \sqrt{(x-0.4)^{2} + (y-0.5)^{2} + (z-0.5)^{2}}}{\sqrt{2\varepsilon}} \bigg),
		&  x < 0.4 + \sqrt{R_{1}^{2} - R_{2}^{2}}, \\[2ex]
		\tanh \bigg( \dfrac{R_{1} - \sqrt{(x-1.6)^{2} + (y-0.5)^{2} + (z-0.5)^{2}}}{\sqrt{2\varepsilon}} \bigg),
		&x > 1.6 - \sqrt{R_{1}^{2} - R_{2}^{2}}, \\[2ex]
		\tanh \bigg( \dfrac{R_{2} - \sqrt{(y-0.5)^{2} + (z-0.5)^{2}}}{\sqrt{2\varepsilon}} \bigg),
		& \text{otherwise},
	\end{cases}
	\]
	which has a toroidal shape.
	Figure~\ref{fig:Motion by Mean Curvature dumbbell-shaped} shows that \Cref{fully discrete formulations scheme} correctly simulates the time evolution of the zero-level isosurface with the initial phase field function having a toroidal shape.

	\begin{figure}[htbp]
		\centering
		\begin{minipage}[t]{0.32\linewidth}
			\centering
			\includegraphics[width=3.2cm]{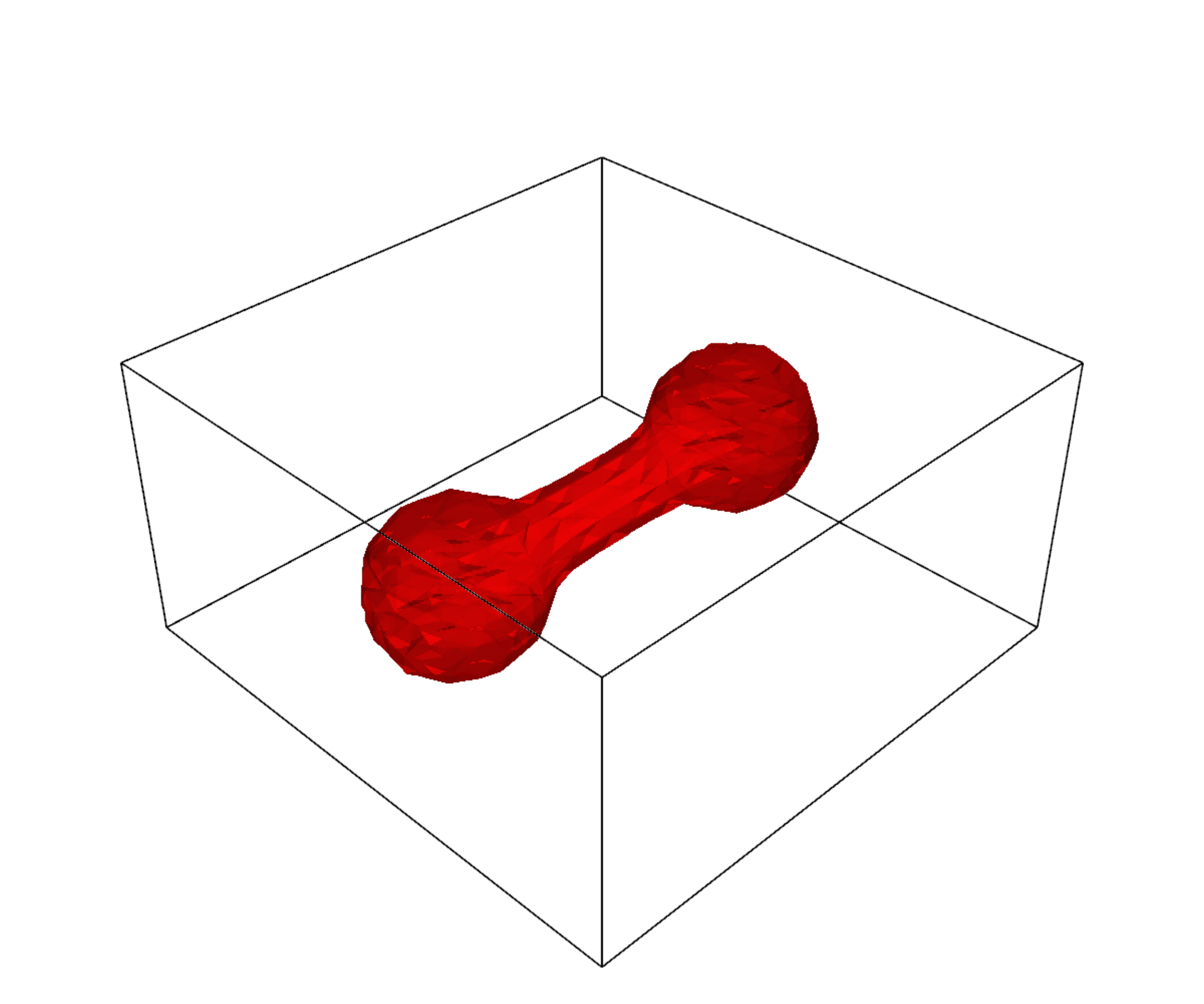}
			\caption*{$t=0$}
		\end{minipage}
		\begin{minipage}[t]{0.32\linewidth}
			\centering
			\includegraphics[width=3.2cm]{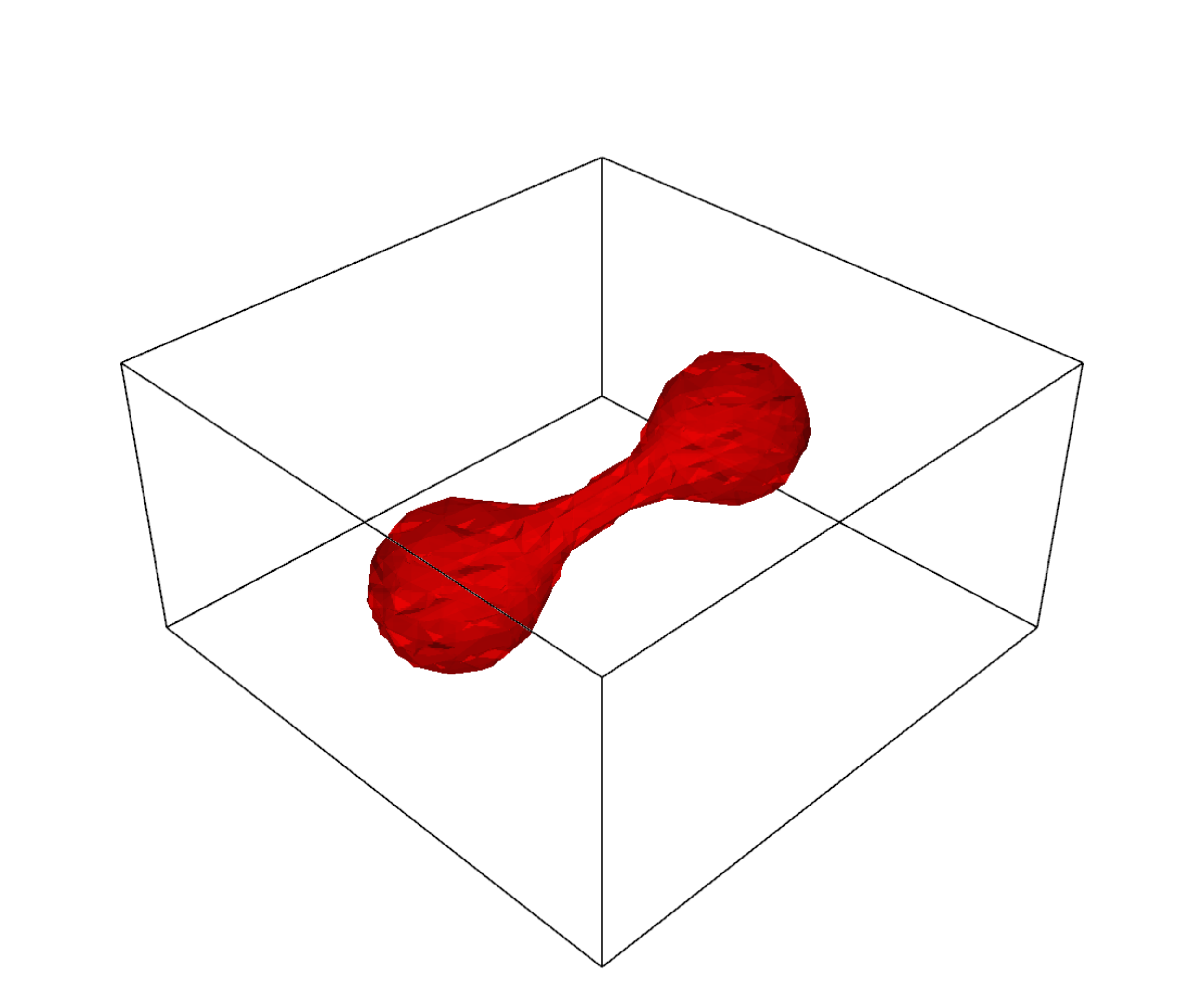}
			\caption*{$t = 0.1$}
		\end{minipage}
		\begin{minipage}[t]{0.32\linewidth}
			\centering
			\includegraphics[width=3.2cm]{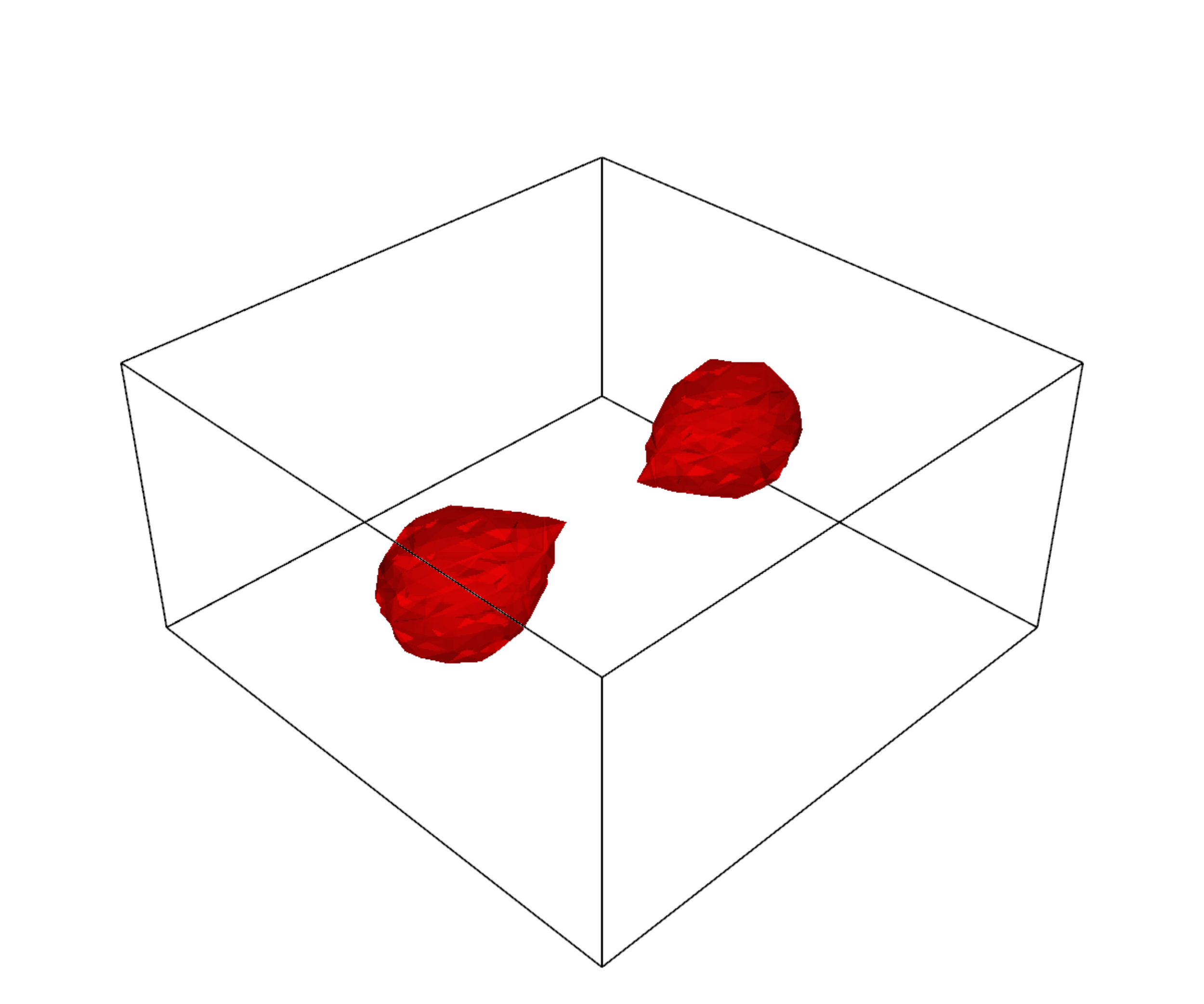}
			\caption*{$t = 0.2$}
		\end{minipage}
		\vspace{-0.1cm}
		\caption{Zero-level isosurfaces of the computational solutions at selected time instances for the initial phase field function having a toroidal shape.}
		\label{fig:Motion by Mean Curvature dumbbell-shaped}
	\end{figure}

	\subsection{Time adaptive test}
	To relieve the conflict between accuracy and computational cost as well as utilize fine properties of the DLN method under non-uniform time grids, we design a time-adaptive approach for \Cref{fully discrete formulations scheme} based on the minimum dissipation criterion proposed by Capuano, Sanderse, De Angelis, and Coppola \cite{capuano2017minimum}. 
	At each time step, we compute the numerical dissipation (ND) rate for both velocity $u$ and phase field variable $\phi$,
	the viscosity-induced dissipation (VD) for $u$, the phase field-induced dissipation (PD) for $\phi$
	\begin{align*}
		&\text{ND: }\ \ 
		\epsilon_{ND}^{u} = \frac{1}{\widehat{k}_n} \| u_{n,\alpha}^{h} \|^2, 
		\quad
		\epsilon_{ND}^{\phi} = \frac{1}{\widehat{k}_n} \| \phi_{n,\alpha}^{h} \|^2, 
		\\
		&\text{VD: }\ \ 
		\epsilon_{VD}^{u} = \mu \| \nabla u_{n,\beta}^h \|^2,
		\quad
		\text{PD: }\ \ 
		\epsilon_{PD}^{\phi} = \kappa \| \nabla \phi_{n,\beta}^h \|^2,
	\end{align*}
	and the ratios of ND over VD and PD: $\chi_u = \epsilon_{ND}^{u} / \epsilon_{VD}^{u}$, $\chi_w = \epsilon_{ND}^{\phi} / \epsilon_{PD}^{\phi}$.
	Then we adjust the next time step $k_{n+1}$ by 
	\begin{equation}
		\label{time-controller}
		\begin{split}
			&k_{n+1} = 
			\begin{cases} 
				\min \{2 k_n, k_{\max} \}, & \text{if}~  \max \{|\chi_u|,|\chi_\phi|\} \leq \delta, \\
				\max \{\frac{1}{2}k_n, k_{\min} \}, & \text{if}~ \max \{|\chi_u|,|\chi_\phi|\} > \delta,
			\end{cases}
		\end{split}
	\end{equation}
	for the required tolerance $\delta >0$.
	We observe from \eqref{time-controller} that this strategy allows a larger time step ($k_{n+1} = 2k_{n}$) for the next operation if the ratios are below the required tolerance; otherwise, it decreases the next time step by half. 
	Meanwhile, we set a maximum time step $k_{\max}$ for accuracy and a minimum time step $k_{\min}$ for efficiency. 
	The complete adaptive procedure incorporating the above time step selection strategy is summarized in Algorithm~\ref{Adaptive_procedure} below.
	\begin{algorithm}
		\caption{Adaptive procedure for solving Scheme~\ref{fully discrete formulations scheme}.}
		\label{Adaptive_procedure}
		\begin{algorithmic}
			\State Set $u_0^h = \mathcal{S}_h u_0, \phi_0^h = \mathcal{R}_h \phi_0$ and compute $u_1^h$, $w_1^h$, $\xi_1^h$, $p_1^h$, $\phi_1^h$ and $m_1^h$ using a Crank--Nicolson scheme;
			\For{$n = 1,2, \dots, M-1$}
			\State Solve active fluid system \eqref{ACAFs} by Scheme~\ref{fully discrete formulations scheme};
			\State Compute numerical indicators $\epsilon_{ND}^{u},\epsilon_{ND}^{\phi},\epsilon_{VD}^{u},\epsilon_{PD}^{\phi} $;
			\State Calculate $\chi_u = \epsilon^{u}_{ND}/\epsilon^{u}_{VD}$, $\chi_\phi = \epsilon^{\phi}_{ND}/\epsilon^{\phi}_{PD}$;
			\If{$ \max \{|\chi_u|,|\chi_\phi|\} \leq \delta$}
			\State Set $k_{n+1} = \min \{2 k_n, k_{\max}\}$;
			\Else
			\State Set $k_{n+1} = \max\{0.5 k_n, k_{\min}\}$;
			\EndIf
			
			\State Set $(u_{n}^h, w_{n}^h, \xi_{n}^h, p_{n}^h, \phi_{n}^h,m_{n}^h) \Leftarrow (u_{n+1}^h, w_{n+1}^h, \xi_{n+1}^h, p_{n+1}^h, \phi_{n+1}^h,m_{n+1}^h)$ \\
			\quad \quad and $(u_{n-1}^h, w_{n-1}^h, \xi_{n-1}^h, p_{n-1}^h, \phi_{n-1}^h,m_{n-1}^h) \Leftarrow (u_{n}^h, w_{n}^h, \xi_{n}^h, p_{n}^h, \phi_{n}^h,m_{n}^h)$ ;
			\State Go to the next step;
			\EndFor
		\end{algorithmic}
	\end{algorithm}

	We evaluate the performance of the proposed time-adaptive strategy via the former experiment in Subsection \ref{Phase-field shape relaxation and fluid self-organization in active fluid}. 
	We set $k_{\max} = 0.01$, $k_{\min} = 1.\rm{e}-4$, $\delta = 0.1$, $k_{0} = k_{\min}$, carry out the experiment with different level of Reynolds number $\rm{Re} = 1/\mu$: $5.\rm{e}+2, 5.\rm{e}+3, 5.\rm{e}+4,  5.\rm{e}+5,  5.\rm{e}+6, 5.\rm{e}+7$, and make other parameters and conditions unchanged. 
	We also compare this approach against the corresponding constant time-stepping scheme with $20000$ time steps for the effectiveness of time adaptivity. 
	Both approaches achieve very similar results of the evolution of the vector field of velocity and the velocity field over the time interval $[0,2]$. 
	However, \Cref{Comparison of Different Reynolds Numbers with Adaptive Time Steps and Constant Time Steps} shows that the constant time-stepping scheme costs many more time steps under all levels of Reynolds number selected, which emphasizes the superiority of the time-adaptive approach.

	\begin{table}
		\centering
		\caption{The constant time-stepping scheme costs many more time steps under all levels of Reynolds number selected, which emphasizes the superiority of the time-adaptive approach.}
		\begin{tabular}{@{}lrrrrrr@{}}
			\hline
			Re$^a$  & 5.\rm{e}+2 &  5.\rm{e}+3 & 5.\rm{e}+4&  5.\rm{e}+5 &  5.\rm{e}+6&  5.\rm{e}+7 \\
			\hline
			Adaptive$^b$   &  356 &  360& 370& 362  & 352&368\\
			Constant$^c$  &20000& 20000 & 20000 & 20000  & 20000& 20000  \\
			\hline
		\end{tabular}
		\label{Comparison of Different Reynolds Numbers with Adaptive Time Steps and Constant Time Steps}
		
		\vspace{1ex}
		{\footnotesize
			$^a$ Reynolds number; \\
			$^b$ Number of computational steps with adaptive time steps; \\
			$^c$ Number of computational steps with fixed time steps.
		}
	\end{table}

    \section{Conclusion}
    \label{sec:sec6}
    In this study, we have developed an efficient spatial-temporal discretization framework for solving the equivalent second-order reformulation of the Allen-Cahn active fluid system. 
	The family of variable time-stepping DLN schemes, known for its second-order accuracy and nonlinearity stability, serves as the time integrator in our approach. 
	For spatial discretization, we introduce two auxiliary variables, which leads to a divergence-free preserving mixed finite element method that is both efficient and straightforward to implement.
	
	We rigorously proved that the family of fully discrete DLN schemes preserves the discrete energy dissipation law, which is crucial for maintaining the stability of the active fluid system.
	Moreover, we have designed a time-adaptive strategy aimed at enhancing the robustness of the scheme while improving computational efficiency. This adaptive method dynamically adjusts the time step size based on the system's dissipation characteristics, ensuring optimal performance throughout the simulation.
	Several numerical experiments validate our theoretical findings, demonstrating that the family of fully discrete DLN schemes, along with the time-adaptive approach, provides an efficient framework for solving Allen-Cahn phase field coupled active fluid dynamics and other more complex systems.

    \section*{Declarations} \ 
    \vskip 0.5cm
    \noindent\textbf{Funding.}
	Wenju Zhao was partially supported by National Key R\&D Program of China (No. 2023YFA1008903), Natural Science Foundation of Shandong Province (No. ZR2023ZD38), National Natural Science Foundation of China (No. 12131014).
	
	\noindent\textbf{Conflict of interest.}
	The authors declare no potential conflict of interests.
	
	\noindent\textbf{Code available.}
	Upon request.
	
	\noindent\textbf{Ethics and consent to participate.} Not applicable

	% Non-BibTeX users please use
    \bibliographystyle{abbrv}
	\bibliography{bibliography}

\begin{thebibliography}{10}

\bibitem{abgrall2023hybrid}
R.~Abgrall and W.~Barsukow.
\newblock A hybrid finite element--finite volume method for conservation laws.
\newblock {\em Applied Mathematics and Computation}, 447:127846, 2023.

\bibitem{abgrall2020analysis}
R.~Abgrall, J.~Nordstr{\"o}m, P.~{\"O}ffner, and S.~Tokareva.
\newblock Analysis of the sbp-sat stabilization for finite element methods part
  i: linear problems.
\newblock {\em Journal of Scientific Computing}, 85(2):1--29, 2020.

\bibitem{ait2023time}
K.~Ait-Ameur, Y.~Maday, and M.~Tajchman.
\newblock Time parallel algorithm for multi-step time schemes.
\newblock In {\em Journ{\'e}es Scientifiques INRIA Chile 2023}, 2023.

\bibitem{ayuso2005postprocessed}
B.~Ayuso, B.~Garc{\'\i}a-Archilla, and J.~Novo.
\newblock The postprocessed mixed finite-element method for the navier--stokes
  equations.
\newblock {\em SIAM Journal on Numerical Analysis}, 43(3):1091--1111, 2005.

\bibitem{BAI20123265}
Z.-Z. Bai and X.~Yang.
\newblock Continuous-time accelerated block successive overrelaxation methods
  for time-dependent stokes equations.
\newblock {\em Journal of Computational and Applied Mathematics},
  236(13):3265--3285, 2012.

\bibitem{baker2024numerical}
K.~Baker, L.~Banjai, and M.~Ptashnyk.
\newblock Numerical analysis of a time-stepping method for the westervelt
  equation with time-fractional damping.
\newblock {\em Mathematics of Computation}, 93(350):2711--2743, 2024.

\bibitem{banjai2012runge}
L.~Banjai, M.~Messner, and M.~Schanz.
\newblock Runge--kutta convolution quadrature for the boundary element method.
\newblock {\em Computer methods in applied mechanics and engineering},
  245:90--101, 2012.

\bibitem{MR4500252}
R.~Cao, N.~Jiang, and H.~Yang.
\newblock Three linear, unconditionally stable, second order decoupling methods
  for the {A}llen-{C}ahn-{N}avier-{S}tokes phase field model.
\newblock {\em J. Math. Anal. Appl.}, 519(1):Paper No. 126792, 23, 2023.

\bibitem{capuano2017minimum}
F.~Capuano, B.~Sanderse, E.~DE~ANGELIS, G.~Coppola, et~al.
\newblock A minimum-dissipation time-integration strategy for large-eddy
  simulation of incompressible turbulent flows.
\newblock In {\em AIMETA 2017 Proceedings of the XXIII Conference of the
  Italian Association of Theoretical and Applied Mechanics}, pages 2311--2323,
  2017.

\bibitem{MR4927938}
Y.~Chen, D.~Luo, W.~Pei, and Y.~Xing.
\newblock Efficient variable time-stepping adaptive {DLN} algorithms for the
  {A}llen-{C}ahn equation.
\newblock {\em J. Sci. Comput.}, 104(2):Paper No. 67, 44, 2025.

\bibitem{CLPX2025_JSC}
Y.~Chen, D.~Luo, W.~Pei, and Y.~Xing.
\newblock Efficient variable time-stepping adaptive dln algorithms for the
  allen-cahn equation.
\newblock {\em Journal of Scientific Computing}, 104(2):67, Jul 2025.

\bibitem{da2025error}
L.~B. Da~Veiga, K.~Hu, and L.~Mascotto.
\newblock Error estimates for a helicity-preserving finite element
  discretisation of an incompressible magnetohydrodynamics system.
\newblock {\em ESAIM: Mathematical Modelling and Numerical Analysis},
  59(2):1075--1094, 2025.

\bibitem{da2017divergence}
L.~B. Da~Veiga, C.~Lovadina, and G.~Vacca.
\newblock Divergence free virtual elements for the stokes problem on polygonal
  meshes.
\newblock {\em ESAIM: Mathematical Modelling and Numerical Analysis},
  51(2):509--535, 2017.

\bibitem{dahlquist1983stability}
G.~G. Dahlquist, W.~Liniger, and O.~Nevanlinna.
\newblock Stability of two-step methods for variable integration steps.
\newblock {\em SIAM journal on numerical analysis}, 20(5):1071--1085, 1983.

\bibitem{decaria2022general}
V.~DeCaria, S.~Gottlieb, Z.~J. Grant, and W.~J. Layton.
\newblock A general linear method approach to the design and optimization of
  efficient, accurate, and easily implemented time-stepping methods in cfd.
\newblock {\em Journal of Computational Physics}, 455:110927, 2022.

\bibitem{ern2022invariant}
A.~Ern and J.-L. Guermond.
\newblock Invariant-domain-preserving high-order time stepping: I. explicit
  runge--kutta schemes.
\newblock {\em SIAM Journal on Scientific Computing}, 44(5):A3366--A3392, 2022.

\bibitem{pnas.1722505115}
S.~Guo, D.~Samanta, Y.~Peng, X.~Xu, and X.~Cheng.
\newblock Symmetric shear banding and swarming vortices in bacterial
  superfluids.
\newblock {\em Proceedings of the National Academy of Sciences},
  115(28):7212--7217, 2018.

\bibitem{he2005stabilized}
Y.~He, Y.~Lin, and W.~Sun.
\newblock Stabilized finite element method for the non-stationarynavier-stokes
  problem.
\newblock {\em Discrete and Continuous Dynamical Systems-B}, 6(1):41--68, 2005.

\bibitem{he2007stability}
Y.~He and W.~Sun.
\newblock Stability and convergence of the crank--nicolson/adams--bashforth
  scheme for the time-dependent navier--stokes equations.
\newblock {\em SIAM Journal on Numerical Analysis}, 45(2):837--869, 2007.

\bibitem{AAMM-16-5}
D.~Hou, J.~Lili, and Q.~Zhonghua.
\newblock A linear doubly stabilized crank-nicolson scheme for the allen–cahn
  equation with a general mobility.
\newblock {\em Advances in Applied Mathematics and Mechanics},
  16(5):1009--1038, 2024.

\bibitem{hou2025unconditionally}
D.~Hou, H.~Liu, and L.~Ju.
\newblock Unconditionally original energy-dissipative and mbp-preserving
  crank-nicolson scheme for the allen-cahn equation with general mobility.
\newblock {\em Computers \& Mathematics with Applications}, 191:86--104, 2025.

\bibitem{layton2022analysis}
W.~Layton, W.~Pei, Y.~Qin, and C.~Trenchea.
\newblock Analysis of the variable step method of dahlquist, liniger and
  nevanlinna for fluid flow.
\newblock {\em Numerical Methods for Partial Differential Equations},
  38(6):1713--1737, 2022.

\bibitem{LPT21_AML}
W.~Layton, W.~Pei, and C.~Trenchea.
\newblock Refactorization of a variable step, unconditionally stable method of
  {D}ahlquist, {L}iniger and {N}evanlinna.
\newblock {\em Appl. Math. Lett.}, 125:Paper No. 107789, 7, 2022.

\bibitem{LPT23_ACSE}
W.~Layton, W.~Pei, and C.~Trenchea.
\newblock Time step adaptivity in the method of {D}ahlquist, {L}iniger and
  {N}evanlinna.
\newblock {\em Advances in Computational Science and Engineering},
  1(3):320--350, 2023.

\bibitem{MR4293957}
B.~Li, S.~Ma, and N.~Wang.
\newblock Second-order convergence of the linearly extrapolated
  {C}rank-{N}icolson method for the {N}avier-{S}tokes equations with {$H^1$}
  initial data.
\newblock {\em J. Sci. Comput.}, 88(3):Paper No. 70, 20, 2021.

\bibitem{MR4471049}
Z.~Li and H.-L. Liao.
\newblock Stability of variable-step {BDF}2 and {BDF}3 methods.
\newblock {\em SIAM J. Numer. Anal.}, 60(4):2253--2272, 2022.

\bibitem{pei2024semi}
W.~Pei.
\newblock The semi-implicit {DLN} algorithm for the {N}avier-{S}tokes
  equations.
\newblock {\em Numer. Algorithms}, 97(4):1673--1713, 2024.

\bibitem{pei2025ensemble}
W.~Pei.
\newblock The variable time-stepping {DLN}-ensemble algorithms for
  incompressible {N}avier-{S}tokes equations.
\newblock {\em Numer. Algorithms}, 2025.

\bibitem{qi2022emergence}
K.~Qi, E.~Westphal, G.~Gompper, and R.~G. Winkler.
\newblock Emergence of active turbulence in microswimmer suspensions due to
  active hydrodynamic stress and volume exclusion.
\newblock {\em Communications Physics}, 5(1):49, 2022.

\bibitem{QHPL21_JCAM}
Y.~Qin, Y.~Hou, W.~Pei, and J.~Li.
\newblock A variable time-stepping algorithm for the unsteady {S}tokes/{D}arcy
  model.
\newblock {\em J. Comput. Appl. Math.}, 394:Paper No. 113521, 14, 2021.

\bibitem{ramaswamy2019active}
S.~Ramaswamy.
\newblock Active fluids.
\newblock {\em Nature Reviews Physics}, 1(11):640--642, 2019.

\bibitem{annurev-fluid-010816-060049}
D.~Saintillan.
\newblock Rheology of active fluids.
\newblock {\em Annual Review of Fluid Mechanics}, 50(Volume 50, 2018):563--592,
  2018.

\bibitem{MR813691}
L.~R. Scott and M.~Vogelius.
\newblock Norm estimates for a maximal right inverse of the divergence operator
  in spaces of piecewise polynomials.
\newblock {\em RAIRO Mod\'{e}l. Math. Anal. Num\'{e}r.}, 19(1):111--143, 1985.

\bibitem{SP24_IJNAM}
F.~Siddiqua and W.~Pei.
\newblock Variable time step method of {D}ahlquist, {L}iniger and {N}evanlinna
  ({DLN}) for a corrected {S}magorinsky model.
\newblock {\em International Journal of Numerical Analysis and Modeling},
  21(6):879--909, 2024.

\bibitem{MR4835947}
M.~Tan, J.~Cheng, and C.-W. Shu.
\newblock High order finite difference scheme with explicit-implicit-null
  time-marching for the compressible {N}avier-{S}tokes equations.
\newblock {\em J. Comput. Phys.}, 523:Paper No. 113626, 22, 2025.

\bibitem{toner1998flocks}
J.~Toner and Y.~Tu.
\newblock Flocks, herds, and schools: A quantitative theory of flocking.
\newblock {\em Physical review E}, 58(4):4828, 1998.

\bibitem{MR4736040}
B.~Wang, Y.~Zhang, and G.-a. Zou.
\newblock Unconditionally stable fully-discrete finite element numerical scheme
  for active fluid model.
\newblock {\em Internat. J. Numer. Methods Fluids}, 96(5):626--650, 2024.

\bibitem{XLWB19_CMAME}
J.~Xu, Y.~Li, S.~Wu, and A.~Bousquet.
\newblock On the stability and accuracy of partially and fully implicit schemes
  for phase field modeling.
\newblock {\em Comput. Methods Appl. Mech. Engrg.}, 345:826--853, 2019.

\bibitem{MR4866010}
Q.~Zhang, W.~Zhao, and G.-a. Zou.
\newblock A time variable stepping technique for solving the {N}avier-{S}tokes
  equations with an adaptive minimum dissipation criterion.
\newblock {\em Z. Angew. Math. Phys.}, 76(2):Paper No. 67, 25, 2025.

\end{thebibliography}

\end{document}